\documentclass[11pt]{amsart}
\usepackage{tikz-cd}
\usepackage{enumitem}
\setlist[description]{leftmargin=\parindent,labelindent=\parindent}
\usetikzlibrary{matrix,arrows,decorations.pathmorphing}
\usepackage{amssymb}
\usepackage{amsmath,amscd}
\usepackage{mathrsfs}
\usepackage{url}
\usepackage{cite}
\usepackage{fullpage}
\usepackage{hyperref}
\usepackage{verbatim}
\usepackage{comment}
\usepackage{fancyvrb}
\usepackage{fvextra}
\usepackage{caption}
\usepackage{ytableau}
\usepackage{stmaryrd}
\usepackage{appendix}
\usepackage{xcolor}
\usepackage[dvipsnames]{xcolor}
\usepackage{etoolbox}

\allowdisplaybreaks

\newtheorem{thm}{Theorem}
\newtheorem{prop}[thm]{Proposition}
\newtheorem{lem}[thm]{Lemma}
\newtheorem{cor}[thm]{Corollary}
\newtheorem{conj}[thm]{Conjecture}
\newtheorem*{thm11}{Theorem 10$'$}

\theoremstyle{definition}
\newtheorem{definition}[thm]{Definition}
\newtheorem{example}[thm]{Example}
\newtheorem{rem}[thm]{Remark}
\newtheorem{question}[thm]{Question}
\newtheorem{speculation}[thm]{Speculation}

\AtBeginEnvironment{definition}{\pushQED{\qed}}
\AtEndEnvironment{definition}{\popQED}

\AtBeginEnvironment{rem}{\pushQED{\qed}}
\AtEndEnvironment{rem}{\popQED}

\AtBeginEnvironment{example}{\pushQED{\qed}}
\AtEndEnvironment{example}{\popQED}

\newcommand{\X}{\mathcal{X}}
\newcommand{\T}{\mathsf{T}}
\newcommand{\J}{\mathcal{J}}
\newcommand{\N}{\mathcal{N}}

\newcommand{\Z}{\mathbb{Z}}

\newcommand{\cc}{\mathbb{C}}
\newcommand{\inv}{\mathfrak{I}_{g, s}}

\newcommand{\R}{\mathsf{R}}

\newcommand{\M}{\mathcal{M}}
\newcommand{\p}{\mathbb{P}}

\newcommand{\A}{\mathcal{A}}

\newcommand{\AJ}{{\textnormal{AJ}}}

\newcommand{\CH}{\mathsf{CH}}

\renewcommand{\H}{\mathsf{H}}

\DeclareMathOperator{\NL}{NL}

\newcommand{\E}{\mathcal{E}}
\renewcommand{\O}{\mathcal{O}}

\DeclareMathOperator{\ev}{ev}
\renewcommand{\tilde}{\widetilde}

\DeclareMathOperator{\Sp}{Sp}

\DeclareMathOperator{\Aut}{{\mathsf{Aut}}}

\newcommand{\Tor}{\mathsf{Tor}}

\DeclareMathOperator{\Sym}{Sym}

\DeclareMathOperator{\codim}{codim}

\newcommand{\taut}{\mathsf{taut}}
\newcommand{\Abar}{\overline{\mathcal A}}
\newcommand{\Mbar}{\overline{\mathcal M}}
\newcommand{\Mct}{\mathcal M^{\ct}}
\newcommand{\aj}{\operatorname{aj}}
\newcommand{\PR}{\operatorname{PR}}

\newcommand{\vir}{\mathrm{vir}}
\newcommand{\ct}{\mathrm{ct}}
\newcommand{\Jac}{\mathsf{Jac}}

\newcommand{\pii}{\rho}

\usepackage{wrapfig}
\definecolor{BrickRed}{HTML}{AA6A43}
\renewcommand{\bar}{\overline}

\newcommand {\replacemu}{\underline{g}}
\newcommand {\FM}{\textrm{FM}}
\makeatletter
\def\@defaultbiblabelstyle#1{[#1]}
\makeatother

\DeclareMathAlphabet{\pazocal}{OMS}{zplm}{m}{n}

\title[]{Torelli loci, product cycles,\\ and  the homomorphism conjecture for $\mathcal{A}_g$}

\dedicatory{To Carel Faber on the occasion of his 1000000{th} birthday}

\author{Samir Canning}
\address{Department of Mathematics, ETH Z\"urich}
\email {samir.canning@math.ethz.ch}

\author{Lycka Drakengren}
\address{Department of Mathematics, ETH Z\"urich}
\email {lycka.drakengren@math.ethz.ch}

\author{Jeremy Feusi }
\address{Department of Mathematics, ETH Z\"urich}
\email {jeremy.feusi@math.ethz.ch}

\author{\\Daniel Holmes}
\address{Institute of Science and Technology Austria (ISTA)}
\email{daniel.holmes@ist.ac.at}

\author{Aitor Iribar L\'opez}
\address{Department of Mathematics, ETH Z\"urich}
\email {aitor.iribarlopez@math.ethz.ch}

\author{Denis Nesterov}
\address{Department of Mathematics, ETH Z\"urich}
\email {denis.nesterov@math.ethz.ch}

\author{Dragos Oprea}
\address{Department of Mathematics, University of California, San Diego}
\email {doprea@math.ucsd.edu}

\author{Rahul Pandharipande}
\address{Department of Mathematics, ETH Z\"urich}
\email {rahul@math.ethz.ch}

\author{Johannes Schmitt}
\address{Department of Mathematics, ETH Z\"urich}
\email {johannes.schmitt@math.ethz.ch}

\author{Zheming Sun}
\address{Department of Mathematics, ETH Z\"urich}
\email {zhesun@ethz.ch}

\date{}
\begin{document}
\baselineskip=17pt
\footskip=1.5\normalbaselineskip

\begin{abstract}
The tautological $\mathbb{Q}$-subalgebra $\R^*(\mathcal{A}_g) \subset \mathsf{CH}^*(\mathcal{A}_g)$ of the Chow ring of the moduli space
of principally polarized abelian varieties is
generated by the Chern classes of the Hodge bundle. A canonical $\mathbb{Q}$-linear projection operator
$$\taut: \CH^*(\mathcal{A}_g) \rightarrow \R^*(\mathcal{A}_g)$$
was constructed in \cite{CMOP}. We present here new
calculations of intersection products of the Torelli locus in $\mathcal{A}_g$ with the product loci
$\mathcal{A}_{r}\times \mathcal{A}_{g-r} \rightarrow 
\mathcal{A}_g$
for $r\leq 3$. The results  suggest  that $\taut$ is a $\mathbb{Q}$-algebra homomorphism (at least for special cycles). We 
discuss a conjectural framework for the 
homomorphism property of tautological projection.

Our calculations follow
two independent approaches. The first is
a direct study of the excess intersection geometry of the fiber product of the Torelli and product morphisms. The second approach is by recasting the geometry in terms of families
Gromov-Witten classes, which are computed by a wall-crossing formula related to unramified maps \cite{DenisMax, Nesterov}.
Both approaches are of interest beyond their use here.

We define new tautological projections of cycles on the fiber products \(\mathcal X_g^s \to \mathcal A_g\) of the universal family. We compute these projections for a class of product cycles on \(\mathcal X_g^s\) in terms of a determinant involving the universal theta divisors and Poincar\'e classes. 
The Torelli pullback of the morphism \(\mathcal A_1 \times \mathcal A_5 \to \mathcal A_6\) was used in \cite{COP} to construct a nontrivial element of the Gorenstein kernel of ${\mathsf R}^5(\mathcal M^{\mathrm{ct}}_6)$. Using Abel-Jacobi pullbacks of product cycles on $\X_g^s$ and their projections, we  construct a new family of classes which we conjecture to lie in the Gorenstein kernels of the tautological rings \(\R^*(\mathcal M^{\mathrm{ct}}_{g,n})\). In particular, we construct nontrivial elements of the Gorenstein kernels of $\R^5(\M_{5,2}^{\ct})$
and $\R^5(\M_{4,4}^{\ct})$.
Whether all of these classes generate the Gorenstein kernels of all $\R^*(\M_{g, n}^{\ct})$ is an intriguing open question.
\end{abstract}

\date{June 2026}
\maketitle
\setcounter{tocdepth}{1}
\tableofcontents{}

\section{Introduction}
\subsection{Tautological classes on the moduli space of abelian varieties}

Let $\A_g$ be the moduli space of principally polarized abelian varieties of dimension $g\geq 1$. Points of $\A_g$ correspond to pairs $(X, \theta)$ of an abelian variety $X$ together with a polarization given by an ample line bundle $\theta$ satisfying $h^0(X, \mathcal O(\theta))=1$. The moduli space
$\A_g$ is an irreducible, nonsingular Deligne-Mumford stack of dimension $\binom{g+1}{2}$ equipped with a universal family and a universal zero section,
$$\pi:\X_g\rightarrow \A_g\,, \ \ \ \mathsf{z}:\A_g\rightarrow \X_g\, .$$ We refer the reader to  \cite{BL} for a foundational treatment of the moduli of abelian varieties.

The Hodge bundle $\mathsf{z}^*\Omega_\pi=\mathbb E\rightarrow \A_g$ is the rank $g$ vector bundle with fibers
$$
\mathbb E|_{[(X,\theta)]} = \mathsf{H}^0(X, \Omega_X)\, .
$$
The Chern classes $\lambda_i = c_i(\mathbb E)$ are naturally defined in the Chow ring{\footnote{All Chow and cohomology rings will be taken with $\mathbb{Q}$-coefficients.}} of the moduli space of abelian varieties.
Over the complex numbers, $\mathcal A_g$ is isomorphic to the orbifold $[\mathbb H_g / \operatorname{Sp}_{2g}(\mathbb Z)]$, where $\mathbb{H}_g$ is the Siegel upper half space. Since $\mathbb H_g$ is contractible,
$$
\H^*(\A_g) \cong \H^*(\operatorname{Sp}_{2g}(\mathbb Z), \mathbb Q)\, .
$$
By a result of Borel \cite{B}, the rational group cohomology of $\operatorname{Sp}_{2g}(\mathbb Z)$ stabilizes to a polynomial algebra in the \emph{odd} lambda classes:
$$
\lim_{g \to \infty}\H^{*}(\operatorname{Sp}_{2g}(\mathbb Z), \mathbb Q) = \mathbb Q[\lambda_1, \lambda_3, \lambda_5, \ldots]\, .
$$

Motivated by Borel's calculation, van der Geer \cite{vdg} defined  the tautological ring
$$
\R^*(\A_g) \subset \CH^*(\A_g)
$$
as the subalgebra generated by $\lambda_1, \ldots , \lambda_g$. The structure of $\R^*(\A_g)$ is well understood:
\begin{thm}[van der Geer \cite{vdg}]\label{thm:vdG_tautological}\leavevmode
    \begin{enumerate}
        \item[\emph{(i)}] The relations $c(\mathbb E\oplus \mathbb E^\vee)=1$ and $\lambda_g=0$ hold 
        in $\R^*(\A_g)$ and generate the ideal of all relations among the $\lambda$ classes.
        \item[\emph{(ii)}] $\R^{\binom{g}{2}}(\A_g)\stackrel{\sim}{=} \mathbb{Q}$ and
        is generated by the class $\prod_{i=1}^{g-1}\lambda_i$. 
        \item[\emph{(iii)}] 
        The intersection pairing
        $$
        \R^*(\A_g) \times \R^{\binom{g}{2}-*}(\A_g) \rightarrow \R^{\binom{g}{2}}(\A_g) \cong \mathbb Q
        $$
        is perfect.
    \end{enumerate}
\end{thm}
The presentation (i) yields an isomorphism $$\mathsf {R}^*(\mathcal A_g)\cong \CH^*(\mathsf{LG}_{g-1})\,,$$ where ${\mathsf{LG}_{g-1}}$ is the Lagrangian Grassmannian of $(g-1)$-dimensional Lagrangian subspaces of $\mathbb C^{2g-2}.$

\subsection{Tautological projection for \texorpdfstring{$\A_g$}{Ag}}
By definition, $\R^*(\A_g)$ is a $\mathbb{Q}$-linear subspace of $\CH^*(\A_g)$.
A canonical $\mathbb{Q}$-linear \emph{tautological projection} was constructed in \cite{CMOP}. A review of the construction is presented here.

Given $\alpha\in \CH^{\binom{g}{2}}(\mathcal{A}_g)$, consider the
evaluation,
\begin{equation*}
\epsilon^{\mathrm{ab}}:\CH^{\binom{g}{2}}(\mathcal{A}_g) \to \mathbb Q\, , \ \quad \alpha \mapsto \int_{\overline{\mathcal{A}}_g}  \overline{\alpha} \cdot \lambda_g\ \, ,
\end{equation*} 
where $\overline \alpha\in \CH^{\binom{g}{2}}(\overline{\A}_g)$ is a lift of $\alpha$ to a nonsingular 
toroidal compactification,
$$
\A_g \subset \overline {\A}_g\, ,\ \ \ 
\overline{\alpha}|_{\A_g} = \alpha\, .
$$
The evaluation $\epsilon^{\mathrm{ab}}(\alpha)$ is well-defined (independent of both the choice of the toroidal compactification and the choice of the lift) by the
results of \cite{CMOP}. 

In addition to the evaluation $\epsilon^{\mathrm{ab}}$, we also have a {\em $\lambda_g$-pairing},
\begin{equation}\label{par2}
\langle\, ,\, \rangle:\,
\mathsf{CH}^*(\mathcal{A}_g)\times \mathsf{R}^{\binom{g}{2}-*}(\mathcal{A}_g)\,\to\,  \mathbb Q\, ,\ \quad \langle \gamma, \delta\rangle=
\epsilon^{\mathrm{ab}}(\gamma\cdot \delta)=
\int_{{\overline{\mathcal{A}}_g}}\overline{\gamma}\cdot \overline{\delta}\cdot \lambda_g \, ,
\end{equation}
The results of \cite{CMOP} show that the $\lambda_g$-pairing is independent of choices and, therefore, well-defined.\\
The restriction of the $\lambda_g$-pairing to tautological classes  
$$\R^*(\A_g)\times \R^{\binom{g}{2}-*}(\A_g)
\to \R^{\binom{g}{2}}(\A_g) 
\cong \mathbb Q$$
is nondegenerate by Theorem \ref{thm:vdG_tautological}. The map to $\mathbb{Q}$ on the right is defined by
$\epsilon^{\mathrm{ab}}$ and is an isomorphism
because{\footnote{Here, $B_{2i}$ is the Bernoulli number.}}
\begin{equation}\label{constantintro}\epsilon^{\mathsf{ab}}\left(\prod_{i=1}^{g-1} \lambda_i\right)= 
\int_{\overline{\A}_g}\,
\prod_{i=1}^{g} \lambda_i
=
\prod_{i=1}^{g} 
\frac{|B_{2i}|}{4i}\, 
\, ,\end{equation}
see 
\cite[page 71]{vdg}.

\begin{definition}[\hspace{-0.5pt}\cite{CMOP}] \label{cvvt}
Let $\gamma \in \mathsf{CH}^*(\mathcal{A}_g)$.
The {\em tautological projection} $\mathsf{taut}(\gamma) \in \mathsf{R}^*(\mathcal{A}_g)$ is the unique tautological class satisfying
$$\langle \mathsf{taut}(\gamma),\delta\rangle =
\langle \gamma,\delta \rangle $$
for all classes $\delta\in \mathsf{R}^*(\mathcal{A}_g)$.
\end{definition}

\noindent $\bullet$ If $\gamma \in \mathsf{R}^*(\mathcal{A}_g)$, then
$\gamma = \mathsf{taut}(\gamma)$, so we have a $\mathbb{Q}$-linear 
projection operator:
$$\mathsf{taut}: \mathsf{CH}^*(\A_g) \to \mathsf{R}^*(\A_g)\, , \ \ \ 
\mathsf{taut}\circ \mathsf{taut}= \mathsf{taut}\, .$$

\noindent $\bullet$ For $\gamma \in \mathsf{CH}^*({\mathcal{A}}_g)$, 
tautological projection provides a canonical decomposition 
$$\gamma = \mathsf{taut}(\gamma)  + (\gamma -
\mathsf{taut}(\gamma))
$$
into purely tautological and purely non-tautological parts. 

\noindent $\bullet$ If $\gamma\in \CH^*(\A_g)$ and $\delta\in \R^*(\A_g)$, we have \begin{equation}\label{tss}\taut(\gamma\cdot \delta)= \taut(\gamma)\cdot \delta\, .\end{equation}

\subsection{Geometric cycles on \texorpdfstring{$\A_g$}{Ag}}\label{gconag}

The geometry of abelian varieties provides a collection of interesting cycles on $\CH^*(\A_g)$. We will discuss three types of geometric cycles here.

\vspace{8pt}
\noindent {\underline{Product cycles:}}
For a partition{\footnote{The parts are an ordered set of positive integers.}} $g = g_1 + \ldots + g_k$, the product map
\begin{equation}\label{eq:product_map}
    \A_{g_1} \times \cdots \times \A_{g_k} \rightarrow \A_g\, 
\end{equation}
is proper. The image is the set of $(X, \theta)\in \A_g$ for which there
exists a decomposition
\begin{equation}\label{eq:product ppav}
    (X, \theta) \cong (X_1, \theta_1)\times \cdots \times (X_k, \theta_k)
\end{equation}
as principally polarized abelian varieties. 
Let $$[\A_{g_1}\times \cdots \times \A_{g_k}]\in \CH^*(\A_g)$$
be the pushforward of the fundamental class along the morphism \eqref{eq:product_map}.

The tautological projections of all product cycles 
in $\A_g$ have been computed in \cite[Theorem 6]{CMOP}.
The cases most relevant for our study here are:
\begin{align*}
\begin{split}%\label{a1smooth}
    \mathsf{taut}\left([{\A}_1\times \A_{g-1}]\right)&=\frac{g}{6|B_{2g}|} \lambda_{g-1}\, ,
\end{split}\\
\begin{split}%\label{a2smooth*}
\mathsf{taut}\left(\left[\A_{2}\times\A_{g-2}\right]\right)&=\frac{1}{360}\cdot \frac{g(g-1)}{|B_{2g}| |B_{2g-2}|}\cdot \lambda_{g-1}\lambda_{g-3}\,,
\end{split}\\
\begin{split}%\label{a3smooth}
\mathsf{taut}\left(\left[\A_{3}\times\A_{g-3}\right]\right)&=\frac{1}{45360}\cdot \frac{g(g-1)(g-2)}{|B_{2g}||B_{2g-2}||B_{2g-4}|}\cdot \lambda_{g-1}(\lambda_{g-4}^2-\lambda_{g-3}\lambda_{g-5})\, ,
\end{split}\\
\begin{split}%\label{a3smooth}
\mathsf{taut}\left(\left[\A_{1}\times\A_{1}\times\A_{g-2}\right]\right)&=\frac{1}{36}\cdot 
\frac{g(g-1)}{|B_{2g}||B_{2g-2}|}\cdot \lambda_{g-1}\lambda_{g-2}\, .
\end{split}
\end{align*}

\vspace{8pt}
\noindent {\underline{Noether-Lefschetz cycles:}}
By permitting the decomposition \eqref{eq:product ppav} to hold up to an {\em isogeny} of polarized abelian varieties, we obtain an infinite collection of cycle classes which contain the product loci $[\mathcal A_{g_1}\times \cdots \times \mathcal A_{g_k}]\in \CH^*(\A_g)$ as special cases.
Even more generally, we can consider irreducible components of the Noether-Lefschetz locus where the abelian varieties $(A,\theta)$ have N\'eron-Severi rank at least $k$.
We will refer to the fundamental classes of any such irreducible loci as  \emph{Noether-Lefschetz cycles}. A classification of Noether-Lefschetz
loci with generic 
N\'eron-Severi rank 2 can be found in \cite{DL}.

After the product loci, the simplest examples of Noether-Lefschetz loci are:
$$
\NL_{g,d} = \{(X, \theta) \,|\, \text{ $\exists$  an elliptic subgroup }E \subset X\text{ with }\deg(\theta|_E)=d\} \subset \A_g\, .
$$
A formula{\footnote{The product in the formula is over all primes $p$ dividing $d$.}} for $\taut([\NL_{g,d}])\in \R^{g-1}(\A_g)$ was proven by Iribar L\'opez \cite[Theorem 3]{Iribar}:
\begin{equation}
\label{je3}
\mathsf{taut}([\NL_{g,d}]) = \frac{d^{2g-1}g}{6|B_{2g}|} \prod_{p|d} (1-p^{2-2g}) \lambda_{g-1}\, ,
\end{equation}
for $g\geq 2$ and $d\geq 1$. A connection to modular forms can be found in a different basis. 
Define
$$
[\widetilde{\NL}_{g,d}] = \sum_{e|d} \sigma_1\left(\frac{d}{e}\right)[\NL_{g,e}]
\ \in \ \mathsf{CH}^{g-1}(\mathcal{A}_g)
\, ,
$$
where $\sigma_r(n)= \sum_{d|n} d^r$, see \cite{GreerLian}.
Iribar L\'opez's result \eqref{je3} can then be rewritten{\footnote{Equation \eqref{aitorproj} is correct also for $g=1$ with the convention $[\NL_{1,d}]=[\mathcal{A}_1]$.}} as 
\begin{equation}\label{aitorproj}
\mathsf{taut}\left( \frac{(-1)^g}{24}\lambda_{g-1}+  \sum_{d = 1}^\infty[\widetilde{\NL}_{g,d}]\, q^d\right) = \frac{(-1)^g}{24}E_{2g}(q)\lambda_{g-1}\,,
\end{equation}
where $E_{2g}(q)$ is the Eisenstein modular function of weight $2g$ in the variable $q= e^{2\pi i \tau}$.
Further connections to modularity, including conjectures by Greer, Lian, Iribar L\'opez, and Pixton, can be found in the lecture \cite{RahulOxford}.
The tautological projections of $[\widetilde{\mathrm{NL}}_{g,d}]$ have a surprising connection  to the genus 1 Gromov-Witten theory of the Hilbert scheme of points of $\mathbb{C}^2$  proven in \cite{ILPT}.

\vspace{8pt} \noindent {\underline{Torelli cycles:}}
A natural source of principally polarized abelian varieties comes from the geometry of curves. A connected nodal curve $C$ is said to be of \emph{compact type} if the dual graph is a stable tree. The moduli space of genus $g$ curves of compact type 
$$\M_g^{\mathrm{ct}} \subset \overline{\M}_g$$
is a nonsingular Deligne-Mumford stack of dimension $3g-3$. The Torelli morphism
\begin{equation}\label{eq:torelli_map}
    \Tor : \M_g^{\ct} \rightarrow \A_g
\end{equation}
sends a compact type curve $C$ to the product of the Jacobians of the irreducible components of $C$, each equipped with the canonical principal polarization given by the corresponding theta divisor. The Torelli morphism \eqref{eq:torelli_map} is proper, and the associated pushforward of the fundamental class is the \emph{Torelli cycle},
$$\Tor_*([\M_g^{\mathrm{ct}}])=[\mathcal J_g] \in \CH^{\binom{g-2}{2}}(\A_g)\,.$$

Faber observed{\footnote{Faber's calculations, undertaken before the complete definition  of $\taut$ was available, can now be interpreted as computing $\taut([\J_g])$.}} that the calculation of 
$\taut([\mathcal{J}_g])\in \R^{\binom{g-2}{2}}(\A_g) $
is equivalent to 
the calculation of
all Hodge integrals of the form
$$\int_{\overline{\M}_g}\lambda_g\cdot \mathsf{P}(\lambda_1,\ldots, \lambda_g)\, ,$$
where $\mathsf{P}$
is an arbitrary polynomial of degree $2g-3$ in the Chern classes of the Hodge bundle \cite{Fa}.

For $g\leq 8$, we  have
\[
\taut(\left[\mathcal J_2\right]) = 1\,,
\]
\[
\taut(\left[\mathcal J_3\right]) = 2\,,
\]
\[
\taut(\left[\mathcal J_4\right]) = 16\lambda_1\,,
\]
\[
\taut(\left[\mathcal J_5\right]) = 144\lambda_1\lambda_2 - 96\lambda_3\,,
\]
\[
\taut(\left[\mathcal J_6\right]) = 768\lambda_1\lambda_2\lambda_3-2304\lambda_2\lambda_4+\frac{948096 }{691}\lambda_1\lambda_5\,,
\]
\[
\taut(\left[\mathcal J_7\right])=1536\lambda_1\lambda_2\lambda_3\lambda_4-13824\lambda_2\lambda_3\lambda_5+\frac{4418304}{691}\lambda_1\lambda_4\lambda_5+\frac{15044352}{691}\lambda_1\lambda_3\lambda_6-\frac{17685504}{691}\lambda_4\lambda_6\,,
\]
\begin{align*}
\taut(\left[\mathcal J_8\right])=& \ \  \ \, \frac{500106387456}{2499347}\lambda_2\lambda_6\lambda_7-\frac{311646117888}{2499347}\lambda_3\lambda_5\lambda_7-\frac{203316609024}{2499347}\lambda_1\lambda_2\lambda_5\lambda_7
\\ & +\frac{139564449792}{2499347}\lambda_1\lambda_3\lambda_4\lambda_7
-\frac{14966784}{691}\lambda_4\lambda_5\lambda_6+ \frac{25731072}{691}\lambda_1\lambda_3\lambda_5\lambda_6\\ & -\frac{12533760}{691}\lambda_2\lambda_3\lambda_4\lambda_6 + \frac{552960}{691}\lambda_1\lambda_2\lambda_3\lambda_4\lambda_5\, .
\end{align*}
Formulas for $\taut([\mathcal{J}_g])$ for $g\leq 11$ are available at the link \cite{hodgedata}.

Since Faber's seminal paper
\cite{Fa}, many closed formulas for Hodge integrals have been found (especially when the
integrand includes the top Chern class $\lambda_g$), see \cite{FP1, FP2, FP4}. However, at present, a closed formula for
$\taut(\left[\mathcal{J}_g\right])$ is not known.

\subsection{The homomorphism property}
Tautological projection is constructed as a $\mathbb{Q}$-linear map.    
The following property was introduced in \cite{Iribar} to study the
relationship of $\taut$ with the multiplicative  structure of $\CH^*(\A_g)$.

\begin{definition}[Iribar L\'opez \hspace{-0.5pt}\cite{Iribar}] \label{d3333}
    Let  $\alpha$, $\beta \in \CH^*(\A_g)$.  The pair $(\alpha, \beta)$ satisfies the \emph{homomorphism property} if
    \begin{equation} \label{eq:hom_prop}
    \taut(\alpha \cdot \beta) = \taut(\alpha) \cdot \taut(\beta) \in \R^*(\A_g)\,. \qedhere
    \end{equation} 
\end{definition}

By equation \eqref{tss}, if either $\alpha\in \R^*(\A_g)$ or $\beta\in \R^*(\A_g)$, then the homomorphism property holds for $(\alpha,\beta)$.
Failure of the homomorphism property is therefore an obstruction to being tautological.

The homomorphism property has been recently proven to hold 
in
the following nontrivial cases:
\begin{itemize}
\item[(i)] If {both} $\alpha$ and $\beta$ are supported on the locus in $\A_g$ 
of non-simple abelian varieties, then the homomorphism
property holds for $(\alpha, \beta)$, see \cite{Iribar}.
    \item[(ii)] The homomorphism property holds for the pair $([\mathcal J_g],[\NL_{g,d}])$ for all $g,d \geq 1$ by a geometric calculation of both sides of the equality 
    \eqref{eq:hom_prop}, see \cite{ILPT}.
    \item[(iii)] If $\beta$ is supported on the Noether-Lefschetz locus of abelian varieties that decompose (up to isogeny) as a product of $m>0$ elliptic curves and a $(g-m)$-dimensional abelian variety, then the homomorphism property holds for $(\alpha,\beta)$ for all $\alpha \in \CH^*(\A_g)$, see \cite{aitordenisjeremy}.
    In particular, the homomorphism property holds for $(\alpha,[\A_1 \times \cdots \times \A_1 \times \A_{g-m}])$.
\end{itemize} 
In case (ii), the homomorphism property holds even though
$[\mathrm{NL}_{g,d}]\in \CH^*(\A_g)$ is known to be non-tautological in many cases 
\cite{COP,Iribar}.

We study here the homomorphism property for pairs of the form $$(\left[\mathcal J_g\right], [\mathcal A_{g_1} \times \cdots \times \mathcal A_{g_k}])\, .$$ 
The main geometric difficulty is the calculation of the pullback
\begin{equation}
\label{lee4}
\Tor^*([\mathcal A_{g_1}\times \cdots \times \mathcal A_{g_k}])\in \CH^*(\M^{\mathrm{ct}}_{g})\, .
\end{equation}
We use two different approaches to compute \eqref{lee4}.
The first is
a direct study of the excess intersection geometry of the fiber product of the Torelli and product morphisms following \cite{COP}. The second approach is by recasting the geometry in terms of families
Gromov-Witten classes, which are computed by a wall-crossing formula related to unramified maps \cite{Nesterov,DenisMax}.
\subsection{The Torelli fiber product}\label{tfp}

Let $\replacemu = (g_1, \ldots , g_k)$ be a partition of $g$. We study the fiber product
\[
\begin{tikzcd}
\mathcal Z \arrow[r] \arrow[d] & \mathcal{M}_g^{\ct} \arrow[d, "\Tor"] \\
\mathcal{A}_{g_1} \times \cdots \times \mathcal{A}_{g_k} \arrow[r] & \ \mathcal{A}_g\,
\end{tikzcd}
\]
in Sections \ref{fptm}
and \ref{exccalc}.
The stack $\mathcal Z$ parametrizes curves of compact type with decomposable Jacobians. Our first result provides a complete description of the fiber product. For the
intersection calculation \eqref{lee4}, an understanding of all intersections
of the irreducible components of $\mathcal{Z}$ is required. In case $\replacemu=(1,g-1)$, a parallel result was proven in \cite{COP}.

\begin{thm}\label{th:fiberprod_description}
    There is a natural stratification of $\mathcal Z$ by nonsingular Deligne-Mumford stacks indexed by stable trees $\mathsf T$ with a genus assignment $$\mathsf g: \mathsf {V}(\mathsf T) \rightarrow \mathbb Z_{\geq 0}$$ and a coloring of the vertices of positive genus
    $$\mathsf c: \mathsf{V}(\mathsf T)_{>0} \rightarrow \{1, \ldots , k\}$$ subject to the following two conditions:
    \begin{itemize}
        \item [\textnormal{(i)}] For each $i=1, \ldots ,k$, $\sum_{v\in {\mathsf c}^{-1}(i)} {\mathsf g}(v)=g_i$. 
        \item [\textnormal{(ii)}] For each edge $e$ of $\mathsf T$, there is at least one path containing $e$ with endpoints given by two vertices of positive genus of
 {\em different} colors and all interior vertices of genus $0$.
    \end{itemize}
    For each $(\mathsf T, \mathsf g, \mathsf c)$, there is a morphism $$
    \mathcal M_{\mathsf T}^{\mathsf {ct}} = \prod_{v \in \mathsf {V}(\mathsf T)} \M_{\mathsf g(v), \mathsf n(v)}^{\ct}\, \to \mathcal Z
    . $$
 If $\mathsf T$ has no vertices of genus $0$, then $\M_{\mathsf T}^{\mathsf {ct}}$ corresponds to an irreducible component of $\mathcal{Z}$. Trees with genus $0$ vertices correspond to strata of the intersections of the irreducible components.
\end{thm}
We will refer to the trees 
in Theorem \ref{th:fiberprod_description} as {\it $\replacemu$-colored extremal trees}.
Let $\mathsf{Tree} (\replacemu)$ be the set of $\replacemu$-colored extremal trees. 
An automorphism of $\mathsf{T}\in \mathsf{Tree}(\replacemu)$ 
is an automorphism of the underlying tree that respects both the genus assignment and the coloring. Let $\mathsf{Aut}(\T)$ be the automorphism group.

\begin{thm}[Excess contributions]\label{thm:contributions}
    We have $$\Tor^*([\mathcal A_{g_1} \times\cdots \times \mathcal A_{g_k}])=
    \sum_{\mathsf T \in \mathsf{Tree}(\replacemu)}
    \frac{1}{|\mathsf{Aut}(\mathsf{T})|}
    \left(\xi_{\mathsf T}\right)_* (\mathsf{Cont}_{\mathsf T})\, \in \R^*(\M_g^{\ct}) ,
    $$
    where $\mathsf{Cont}_{\mathsf T}\in \R^*(\M_{\mathsf T}^{\mathsf{ct}})$ is a recursively computed tautological class, and $$\xi_{\mathsf T} : \M_{\mathsf T}^{\mathsf{ct}} \rightarrow \M_g^{\ct}$$ is the gluing map.
\end{thm}

\noindent The definitions of the tautological rings
$\R^*(\M_{\mathsf{T}}^{\ct})$ and $\R^*(\M_g^{\ct})$ will be reviewed in Section \ref{frfr4}.
Theorem \ref{thm:contributions} extends \cite [Theorem 3]{COP} from the case $\replacemu=(1,g-1)$ to arbitrary partitions $\replacemu$.

A key step in the proof of Theorem \ref{thm:contributions} is the reducedness  of the fiber product $\mathcal Z$ proven by
Drakengren \cite{Drakengren}. Reducedness is crucial here: we can then write local equations for $\mathcal{Z}$ and compute the pullback class via excess intersection theory \cite{Fulton} on the strata of $\mathcal{M}_g^{\ct}$. The result is a graph sum formula expressing the pullback in terms of tautological classes on the moduli of curves.

Reducedness of $\mathcal{Z}$ is a fortunate property. For example, Drakengren \cite{DrakengrenTorSq} has proven that the Torelli square in genus $4$,
\[
\begin{tikzcd}
\mathcal F \arrow[r] \arrow[d] & \mathcal{M}_4^{\ct} \arrow[d, "\Tor"] \\
\mathcal{M}_4^{\ct} \arrow[r, "\Tor"] & \ \mathcal{A}_4\, ,
\end{tikzcd}
\]
has a {\em nonreduced} fiber product $\mathcal{F}$.

\subsection{Wall-crossing}\label{wccc}
The wall-crossing method can be used to calculate 
$$\Tor^*([\mathcal A_{g_1}\times \cdots \times \A_{g_k}])\in \CH^*(\M^{\ct}_g)\, .$$
The calculation of $\Tor^*([\mathcal A_{1}\times \A_{g-1}])$ by wall-crossing was carried out in \cite{DenisMax}. The wall-crossing geometry pursued here  for 
$\Tor^*([\mathcal A_{r}\times \A_{g-r}])$ is more subtle. 

Let $1\leq r \leq g-1$, and let $$\pi_r : \mathcal X_r \rightarrow \A_r$$ be the universal family of principally polarized abelian $r$-folds. Let $$\beta_{\mathrm{min}} = \frac{\theta^{r-1}}{(r-1)!}$$ be the minimal class, and let $$
\M_g^{\ct}(\pi_r)=
\M^{\ct}_{g}(\pi_r, \beta_{\text{min}})/\mathcal X_r$$ be the space of stable maps to the fibers of $\pi_r$ whose domain is a curve of compact type and whose image lies in the homology class $\beta_{\text{min}}$, up to translations of the fibers.

The Jacobian of a curve $C$ of compact type splits as a product of principally polarized abelian varieties of dimension $r$ and $g-r$,
$$\mathsf{Jac(C)} = {X}_r \times {Y}_{g-r}\, ,$$
if and only if there is a stable map $C \to X_r$ that represents the minimal class. The correspondence extends to families, so we obtain an isomorphism of stacks
\begin{equation}\label{eq:isom_stablemaps}
    \mathcal Z \cong \M^{\ct}_{g}(\pi_r)\,, \ \ \ \ \ \ \ 
\begin{tikzcd}
\mathcal Z \arrow[r] \arrow[d] & \mathcal{M}_g^{\ct} \arrow[d, "\Tor"] \\
\mathcal{A}_{r} \times  \mathcal{A}_{g-r} \arrow[r] & \ \mathcal{A}_g\,.
\end{tikzcd}
\end{equation} When $r=1$, the above isomorphism is explained in \cite [Proposition 21]{COP}. For higher $r$, the argument is parallel.

The space of stable maps $\M^{\ct}_{g}(\pi_r)$ carries a 2-term perfect obstruction theory and a virtual fundamental class. Since an $r$-dimensional abelian variety has $h^{0,2}=\binom{r}{2}$, the virtual fundamental class can be reduced $h^{0,2}$ times, yielding a reduced{\footnote{See \cite{BryanOberPY} for a study of the reduced virtual fundamental class of the moduli space of stable maps to an abelian variety.}} 
virtual fundamental class $[\M^{\ct}_{g}(\pi_r)]^{\text{red}}$. By a result of \cite {aitordenisjeremy} extending  \cite[Theorem 1.2]{GreerLian},  there is an equality of cycles
\begin{equation} \label{ks33}
\Tor^![\mathcal A_r \times \mathcal A_{g-r}] = [\M^{\ct}_{g}(\pi_r)]^{\mathrm{red}}\, \in \CH^*(\M_g^{\ct}(\pi_r))
\end{equation}
under the isomorphism \eqref{eq:isom_stablemaps}.

The wall-crossing formula established in \cite[Theorem 6.21]{Nesterov} relates the virtual fundamental classes of  moduli spaces of stable  and unramified maps. Let
\[
\M^{\ct,\mathrm{un}}_{g_0}(\pi_r)=\M_{g_0}^{\ct, \mathrm{un}}(\pi_r , \beta_{\mathrm{min}})/\mathcal X_r\, 
\]
be the analogously defined  space of unramified maps to the fibers of $\pi_r$. More generally, for the statement of the wall-crossing formula,  spaces of unramified  maps from possibly disconnected curves and with fixed multiplicities over moving points on the target are required{\footnote{We use the superscript $\circ$ for 
connected domains and  the superscript
$\bullet$ for possibly disconnected domains. If no superscript is given, then the
domains are understood to be connected.
}},
\[
\M^{\ct,\mathrm{un}}_{g_0}(\pi_r,\mu^1,\hdots,\mu^m)^\bullet,
\]
where $\mu^i=(\mu^i_1, \hdots, \mu^i_\ell)$ are the partitions that encode the multiplicities.  
Unramified maps and unramified maps with prescribed ramifications are higher dimensional generalizations of admissible covers, see \cite{KKO, Nesterov}. The precise definitions are reviewed in Section \ref{wallcr}, and an expository treatment can be found in \cite{13andahalf}.

Over the moduli space of unramified maps, before taking the quotient by translations of the target, derivatives of maps at points associated with multiplicity profiles $\mu^i$ provide evaluation morphisms to $\p(T_{\mathcal X_r/\mathcal A_r})^{\ell(\mu^i)}$, see \cite{KKO} for details. 
After taking the quotient, we obtain evaluation morphisms to $\mathbb P(\mathbb E^\vee_r)$, the tangent space at the zero section $\mathsf{z}$:
\[ 
\mathrm{ev}_i \colon 
\M^{\ct,\mathrm{un}}_{g_0}(\pi_r,\mu^1,\hdots, \mu^m)^\bullet \rightarrow \p(\mathbb E^\vee_r)^{\ell(\mu^i)}, \quad i=1, \hdots, m\, . 
\]
The classes responsible for the difference between the enumerative
geometry of stable and unramified maps live on the moduli spaces
$$\M^{\mathrm{ct}}_{h,\,\ell(\mu)} \times \p(\mathbb E^\vee_r)^{\ell(\mu)}\, ,$$
and take the form
\begin{equation}\label{ifun}
I_{h,\mu}(z) = \left.\frac{1}{|\Aut(\mu)|}\,  \frac{\prod^{r}_{j=1}\Lambda^\vee(z-\alpha_j)}{\prod^{\ell(\mu)}_{k=1}(z-\psi_k-H_k)}\right|_{z^{\geq 0}}\in \CH^*\left(\M^{\mathrm{ct}}_{h,\,\ell(\mu)} \times \p(\mathbb E^\vee_r)^{\ell(\mu)}\right)[z]\, ,
\end{equation}
where $\Lambda^\vee(t)= \sum_{i=0}^h c_{i}(\mathbb{E}^\vee_h)t^{h-i},$
$\alpha_j$ are Chern roots of $\mathbb E_r$, $H_k$ are the hyperplane classes on the factors of $\p(\mathbb E^\vee_r)^{\ell(\mu)}$, and $\psi_k$ are the cotangent classes on  $\M^{\mathrm{ct}}_{h,\, \ell(\mu)}$. The following result is established in \cite {aitordenisjeremy}.

\begin{thm}[Wall-crossing] \label{thm:wallcr} For $1\leq r \leq g-1$, we have
\[
[\M^{\ct}_{g}(\pi_r)]^{\textnormal{red}}= \sum_{\mathsf{S}\in \mathsf{Star}(g)}\frac{1}{|\Aut(\mathsf{S})|} \, (\xi_{\mathsf{S}})_* \left([\M^{\ct,\mathrm{un}}_{g_0}(\pi_r,\mu^1,\hdots, \mu^m)^\bullet]^{\textnormal{red}}\boxtimes \prod^m_{i=1} \mathrm{ev}^*_i I_{g_i,\mu^i}(-\Psi_i)\right),
\]
where the sum is taken over star-shaped graphs $\mathsf{S}$ with partition labels $\mu^1, \ldots, \mu^m$ on the edges that correspond to connected domains after gluing and satisfy
\[
g=\sum^{m}_{i=1}(g_i+\ell(\mu^{i}))+ g_0-m\, . 
\]
The $\Psi_i$ are the cotangent classes on $\M^{\ct,\mathrm{un}}_{g_0}(\pi_r,\mu^1,\hdots, \mu^m)^\bullet$ associated to the markings of the target. The morphism 
\[
\xi_{\mathsf{S}} \colon \M^{\ct,\mathrm{un}}_{g_0}(\pi_r,\mu^1,\hdots, \mu^m)^\bullet \times \prod^m_{i=1} \M^{\mathrm{ct}}_{g_i,\ell(\mu^i)} \rightarrow \M^{\ct}_{g}(\pi_r) 
\]
is given by applying the forgetful morphism 
$$\M^{\ct,\mathrm{un}}_{g_0}(\pi_r,\mu^1,\hdots, \mu^m)^\bullet \rightarrow  \M^{\ct}_{g_0,\sum \ell(\mu^i)}(\pi_r)^\bullet$$ followed by the gluing morphism which attaches contracted components. 
\end{thm}

A formal definition of star-shaped graphs $\mathsf{S}$ and
the cotangent line classes $\Psi_i$ will be given in Section \ref{wallcr}.
The complexity of the moduli spaces of unramified maps which appear in the wall-crossing formula
depends upon $r$:

\vspace{6pt}
\noindent $\bullet$ For $r=1$, the only moduli spaces of unramified maps which occur in the wall-crossing formula are 
\[
\M^{\ct,\mathrm{un}}_{1}(\pi_1,(1),\hdots,(1))^\circ,
\]
where $\mu^i=(1)$ for all $i$, see \cite{DenisMax}.

\vspace{6pt}
\noindent $\bullet$ For $r=2$, the only moduli spaces of unramified maps which occur  are
\begin{equation}\label{eq:M2ctun}
    \M^{\ct,\mathrm{un}}_{2}(\pi_2,(1),\hdots, (1))^\circ \ \ \ \text{and} \ \  \
\M^{\ct,\mathrm{un}}_{1}(\pi_2,(1,1),(1),\hdots, (1))^\bullet\, ,
\end{equation}
The partition $\mu=(1,1)$ arises for unramified maps to 
fibers of $\pi_2$ given by  products $E_1\times E_2$ of elliptic curves: 
the associated domain curves are two disconnected genus 1 curves mapping to a nodal genus 2 curve consisting of two elliptic components. The point on the target corresponding to the node of the genus 2 curve has multiplicity $(1,1)$. 

Since curves in the minimal class $\beta_{\mathrm{min}}$ are at worst nodal, maps to an abelian surface can only have multiplicities $(1)$ and $(1,1)$. By admissibility, the curves have degree at most $2$ over rational components of Fulton-MacPherson degenerations of abelian surfaces. Degree $2$ curves in $\mathbb{P}^2$ are also at worst nodal. Hence, the same multiplicities $(1)$ and $(1,1)$ occur for maps to Fulton-MacPherson degenerations.

A result of \cite {aitordenisjeremy} exhibits the spaces \eqref{eq:M2ctun} as an 
iterated blowup of $\M^{\ct}_{2,n}$ and $\Mct_{1,n}$ along boundary strata
in both the connected and disconnected cases.
As a result, the moduli space is nonsingular of the expected dimension, and the reduced virtual fundamental class agrees with the fundamental class. Therefore, Theorem \ref{thm:wallcr} provides an explicit method to evaluate the wall-crossing formula in terms of tautological classes on the moduli spaces of curves. A full analysis is 
provided in Section \ref{wallcr} with examples.

%\vspace{6pt}
%\noindent $\bullet$ For $r\geq 3$, only partitions of the type 
%$\mu=(1, \hdots, 1)$ can occur since the minimal curves can have at worst ordinary multiple point singularities.
%The analysis for $r\geq 3$ is more complex and will not be treated here.

\subsection{Results about the homomorphism property}
The codimension of $\mathcal{J}_g \subset \A_g$ is 
$$\mathsf{cod}_{\mathcal{J}} = \binom{g+1}{2}- (3g-3)\,.$$
The codimension of the product cycle 
$\mathcal A_{g_1}\times \cdots \times \A_{g_k}$ in $\A_g$ is
$$\mathsf{cod}_{\replacemu}=\binom{g+1}{2} - \sum_{i=1}^k \binom{g_i+1}{2}=\sum_{1\leq i<j\leq k} g_ig_j\, .$$
The homomorphism property
\begin{equation} \label{k3dd4}
\taut([\mathcal J_g]\cdot [\mathcal A_{g_1} \times \cdots \times \mathcal A_{g_k}]) \stackrel{?}{=} \taut([\mathcal J_g])\cdot \taut([\A_{g_1}\times \cdots \times \A_{g_k}]) \, \in \, \R^{\mathsf{cod}_{\mathcal{J}}+ \mathsf{cod}_{\replacemu}}   (\A_g)
\end{equation}
is trivial (and true) if
\begin{equation}\label{ineq}\mathsf{cod}_{\mathcal{J}}+ \mathsf{cod}_{\replacemu} > \binom{g}{2}\,.
\end{equation}
Simple arithmetic then shows that the validity of the homomorphism property \eqref{k3dd4} is a non-trivial question only in the following cases:
\begin{itemize}
    \item[(i)] $\A_1 \times \A_{g-1} \rightarrow \A_g$ ,
    \item[(ii)] $\A_2 \times \A_{g-2} \rightarrow \A_g$ ,
    \item[(iii)] $\A_3 \times \A_3 \rightarrow \A_6$ ,
    \item[(iv)] $\A_1\times \A_1 \times \A_{g-2} \rightarrow \A_g$ .
\end{itemize}
Case (i) was affirmatively answered by combining \cite[Theorem 4]{COP} and \cite [Theorem 7]{Iribar}, while case (iv) follows from the results of \cite{aitordenisjeremy}. Theorems \ref{thm:contributions} and \ref{thm:wallcr}   are suitable to cover cases (ii) and (iii).

We have implemented the computations described in Sections \ref{tfp} and \ref{wccc} in the programs \cite{code, codeFeusi}
which can be evaluated with the software  $\mathsf{admcycles}$ \cite{admcycles} designed
 for calculations in the tautological ring of the moduli spaces of curves. 
 The excess intersection calculations are discussed in Sections \ref{compimp} -- \ref{general}, and the wall-crossing calculations are discussed in Sections \ref{wc4} -- \ref{wch}.
 The outcome is the verification of the homomorphism property in the following new cases.

\begin{thm}\label{thm:A3A3_A2Ax}
    The homomorphism property holds for the pairs:
    \begin{itemize}
    \item[\emph{(ii)}] $([\mathcal J_g],\, [\A_2 \times \A_{g-2}])$ on $\mathcal A_g$ for all
    $2\leq g \leq 8\,,$
    \item[\emph{(iii)}]
    $([\mathcal J_6],\, [\A_3 \times \A_3])$ on $\mathcal A_6\,.$
    \end{itemize}
\end{thm}
\noindent 
The only difficulty in extending the results (ii) to higher $g$ is
computational complexity. We expect the homomorphism property (ii) will hold for all $g\geq 2$. 

For partitions $\replacemu=(g_1, \ldots, g_k)$ that do not fall into cases (i) - (iv), inequality \eqref{ineq} yields $$\mathsf{cod}_{\replacemu}>\binom{g}{2}-\mathsf{cod}_\J=2g-3\,.$$
By  either Theorem \ref{thm:contributions} or Theorem \ref{thm:wallcr}, the Torelli pullbacks of
product classes lie in the tautological ring
$$\R^*(\M_g^{\ct}) \subset \CH^*(\M_g^{\ct})\, .$$
The vanishing $\R^{>2g-3}(\M_g^{\ct})=0$ of \cite{FP5, GV} then implies that $$\mathsf{Tor}^*([\A_{g_1}\times \cdots \times \A_{g_k}])=0\,,$$
unless $(g_1, \ldots, g_k)$ is one of the partitions listed in (i) - (iv). For instance, for $\A_{3} \times \A_{g-3} \rightarrow \A_g$, we have
the following result related to case (iii):
$$\Tor^*([\A_{3} \times \A_{g-3}]) = 0 \in \CH^{3g-9}(\M_g^{\ct})$$ 
for $g\geq 7$. After pushing forward under $\Tor$, we obtain the vanishing: $$[\J_g]\cdot [\A_{3} \times \A_{g-3}]=0\in \CH^{\binom{g+1}{2}-6}(\A_g), \quad g\geq 7\,.$$ The same reasoning also yields the following sporadic vanishings in $\mathsf{CH}^*(\A_g)$: $$[\J_g]\cdot[\A_4\times \A_{g-4}]=0\,, \quad 8\leq g\leq 13\,,\quad \text{and }\quad[\J_g]\cdot [\A_5\times \A_{g-5}]=0\,, \quad 10\leq g\leq 11\,.$$

\subsection{Conjectures}

The widest interpretation of our calculations is that the homomorphism property  holds
in general.

\begin{conj} \label{conj:Hom}
    The $\mathbb{Q}$-linear map $\taut: \CH^*(\A_g) \rightarrow \R^*(\A_g)$
    is a 
    homomorphism of $\mathbb{Q}$-algebras.
\end{conj}

An explanation for Conjecture \ref{conj:Hom} could be the existence of a compact subvariety $$\iota: V\hookrightarrow\mathcal A_g$$ of codimension $g$
for which $\mathsf{H}^*(V)$ is generated by pullbacks of tautological classes on $\A_g$.
By the ampleness of $\lambda_1$ and the Gorenstein property of $\R^*(\A_g)$, 
we would then have
$$\mathsf{H}^*(V) \stackrel{\sim}{=} \R^*(\A_g)$$
as $\mathbb{Q}$-algebras.
Moreover, $\taut$ would then be realized as the pullback 
$$\taut: \CH^*(\A_g) \stackrel{\iota^*} \rightarrow \mathsf{H}^*(V) \stackrel{\sim}{=} \R^*(\A_g)\, $$
and, therefore, would be a homomorphism of $\mathbb{Q}$-algebras.

For $g\geq 3$, there are no such subvarieties $V \subset \A_g$ over the complex numbers \cite{keelsadun}. However, in characteristic $p$, the locus $V_0\subset \A_g$ of abelian varieties with $p$-rank $0$ provides a compact subvariety of the correct codimension (which represents
a multiple of $\lambda_g$), see \cite{NO, vdg}. Unfortunately, the Chow groups of $V_0$ have not yet been determined. A specialization argument shows that the homomorphism property  for cycles on $\mathcal A_g(\overline{\mathbb F}_p)$ for all $p$ implies the homomorphism
property for cycles on $\mathcal A_g(\overline{\mathbb Q})$. Since every cycle over $\mathbb C$ is homologous to a cycle defined over $\overline{\mathbb Q}$, Conjecture \ref{conj:Hom} would then follow.

%Nevertheless,
%the existence of $V_0$ in characteristic $p$ may be viewed as
%motivation for the following weakening of Conjecture \ref{conj:Hom}:

% \begin{conj}[$\overline{\mathbb{Q}}$ form] \label{conj:weakHom}
%     The homomorphism property holds for all classes $\alpha, \beta \in \CH^*(\A_g)$ defined over $\overline{\mathbb Q}$.
% \end{conj}

Conjecture \ref{conj:Hom} is a rather
bold interpretation of the current data.
Another possibility is that the homomorphism property holds in
the restricted case where at least one of the classes $\alpha$ or $\beta$
is a Noether-Lefschetz cycle or, more generally, a special cycle \footnote{See \cite {MO} for a discussion of special cycles in the context of $\A_g$.}.

\subsection{Stable curves and  toroidal compactifications of $\A_g$}
Let $\A_g\subset \overline{\A}_g$ be
a toroidal compactification to which the Torelli map can be lifted:
\[
\begin{tikzcd}
\M_g^{\ct} \arrow[r] \arrow[d,"\Tor"] & \overline{\mathcal{M}}_g \arrow[d, "\Tor"] \\
\A_{g}  \arrow[r] & \ \overline{\A}_g\,.
\end{tikzcd}
\]
The two standard examples are the perfect cone compactification and the second Voronoi compactification, see \cite{AMRT} for the definitions.  
A natural question is to understand the
Torelli fiber product
\begin{equation} \label{ndrf}
\begin{tikzcd}
\overline{\mathcal Z} \arrow[r] \arrow[d] & \overline{\M}_g \arrow[d, "\Tor"] \\
\overline{\mathcal{A}_{g_1} \times \cdots \times \mathcal{A}_{g_k}} \arrow[r] & \ \overline{\mathcal{A}}_g\,,
\end{tikzcd}
\end{equation}
where, in the lower left corner, a compactification of the product that maps to $\overline{\A}_g$ is taken.
In the case of the perfect cone compactification, the  product of the compactifications maps to $\overline{\A}_g$, as proven in \cite {SB}, so we can take
$$\overline{\mathcal{A}_{g_1} \times \cdots \times \mathcal{A}_{g_k}} =
\overline{\mathcal{A}}_{g_1} \times \cdots \times \overline{\mathcal{A}}_{g_k} \,. $$
If we further assume that the image 
$\Tor(\overline{\M}_g)$ lies in the nonsingular locus of $\overline{\A}_g$,
a property that holds for both perfect cone and second Voronoi \cite[Section 5.2]{2ndvor}, then
the following question is well-defined.

\begin{question} \label{qqq} Compute
$\Tor^*([\, \overline{\mathcal{A}_{g_1} \times \cdots \times \mathcal{A}_{g_k}}\, ])
\in \CH^*(\overline{\M}_g)\, .$
\end{question}

At the moment, Question \ref{qqq} is open. In principle, both excess intersection theory and the wall-crossing approach can be applied. For the 
former, the scheme theoretic structure of the fiber product \eqref{ndrf} over the divisor $\Delta_0 \subset \overline{\M}_g$ 
would have to be determined. The wall-crossing approach
has an advantage in that the families
Gromov-Witten theory is well-defined for 
the family $$\overline{\pi}_r : \overline{\mathcal X}_r \rightarrow \overline{\A}_r\ , $$ since $\overline{\pi}_r$ is log smooth.
 Wall-crossing then provides a formula for the families Gromov-Witten class in terms of simpler unramified contributions. However,
the comparison \eqref{ks33} of
the families Gromov-Witten class to the
intersection theory is not immediate 
over the compactification.
If we restrict our attention to the partial
compactification of $\A_g$ determined by the {\em torus rank 1 locus}{\footnote{See \cite{Mumford999} for the definitions.}} $\A_g^ {\leq 1} \subset \overline{\A}_g$, we obtain the fiber diagram
\[
\begin{tikzcd}
\M_g^{\ct} \arrow[r] \arrow[d,"\Tor"] & {\mathcal{M}}^{\leq 1}_g \arrow[d, "\Tor"] \\
\A_{g}  \arrow[r] & \ {\A}^{\leq 1}_g\,,
\end{tikzcd}
\]
where $\M_g^{\leq 1}$ denotes the moduli space of stable curves with at most 1 loop in the dual graph. The wall-crossing method \cite{aitordenisjeremy} provides
a complete calculation of
$$\Tor^*\left(\left[\, \overline{\mathcal{A}_{1} \times 
\mathcal{A}_{g-1}}\big{|}_{\A^{\leq 1}_g}\, \right]\right)
\in \R^*({\M}^{\leq 1}_g)\, .$$ 

In \cite{COP}, the kernel in $\R^*(\M^{\ct}_6)$ of the $\lambda_6$-pairing
was constructed geometrically by the class
$$\Tor^*\Big([\A_1\times \A_5] - \taut([\A_1\times \A_5])\Big)\in \R^5(\M_6^{\ct})\, .$$
Using the geometry of
the torus rank $1$ locus $\A_6^{\leq 1} \subset \overline{\A}_6$,
a parallel construction holds for the
1-dimensional kernel
 of the $\lambda_5$-pairing
on
 $\R^*(\M_{5,2}^{\ct})$.

\begin{thm}[Feusi-Iribar L\'opez-Nesterov \cite{aitordenisjeremy}]\label{trank1}
 Let $\iota: \M_{5,2}^{\ct} \rightarrow \M_6^{\leq 1}$ be the boundary map. Then, the class{\footnote{The definition of the class $\taut^{\leq 1}\left([\overline{\A_1\times \A_{g-1}}\big{|}_{\A_g^{\leq 1}}]\right)\in \mathsf {CH}^{g-1}(\A_g^{\leq 1})$ extending $\taut ([\A_1\times \A_{g-1}])\in \CH^{g-1}(\A_g)$ over the boundary 
 is given in \cite{aitordenisjeremy}.}}
    $$
    \iota^* {\Tor}^*\left(
\left[\overline{\mathcal{A}_{1} \times 
\mathcal{A}_{5}}|_{\A^{\leq 1}_6}\right] -
\taut^{\leq 1}\left(\left[\overline{\mathcal{A}_{1} \times 
\mathcal{A}_{5}}|_{\A^{\leq 1}_6}\right]\right)\right)
     \in \R^{5}(\M_{5, 2}^{\mathsf{ct}})
    $$
    is the generator of the 1-dimensional kernel of the $\lambda_5$-pairing on $\R^*(\M_{5, 2}^{\ct})$. 
\end{thm}

A different approach to the kernel of the $\lambda_5$-pairing on 
$\R^*(\M_{5,2}^{\ct})$ via the geometry of the universal family $\X_5\rightarrow \A_5$
will be presented in Section \ref{gouf}.

\subsection{The geometry of the universal fiber product $\mathcal{X}^s_g$} \label{gouf}

A difficulty in studying Question \ref{qqq} 
and extending Theorem \ref{trank1} to higher torus rank loci is the complexity of the boundary geometry of toroidal compactifications $\A_g\subset \overline{\A}_g$. To generate elements of the Gorenstein kernel of $\R^*(\M^{\ct}_{g,2s})$ for general $s$, a different approach can be pursued using the $s$-fold fiber product
$$\pi^s_g:\mathcal{X}^s_g \rightarrow \A_g\, $$
of the universal abelian variety $\pi_g: \mathcal{X}_g \rightarrow \A_g$. The 
universal fiber product geometry stays over
$\A_g$, so the intricacies of the boundary of $\overline{\A}_g$ are avoided.

There is a natural proper morphism
$\A_1 \times \X_{g-1} \rightarrow  \X_g$
defined{\footnote{We have dropped the principal polarizations here for notational convenience.}} by
$$
({E},(Y, p\in Y)) \mapsto ({E} \times Y, (0,p)\in {E}\times Y)\, .
$$
Hence, we also have a morphism $\A_1 \times \X_{g-1} \rightarrow \A_g$ which factors through 
$$
\A_1 \times \A_{g-1} \rightarrow \A_g\, .
$$

\begin{definition}\label{prlocus}
    The generalized product locus is
    \begin{equation}
    \PR_{g,s} = \underbrace{(\A_1 \times \X_{g-1}) \times_{\A_g} \cdots \times_{\A_g}  (\A_1 \times \X_{g-1})}_{s}
    \rightarrow \X_g^s\, . \qedhere
    \end{equation}
\end{definition}
\noindent An immediate property of the definition is the factorization
$$
\PR_{g,s} \rightarrow  \mathcal{X}^s_g|_{\A_1\times \A_{g-1}} \rightarrow \X_g^s\, ,
$$
where both morphisms are proper. Let
$$
[\PR_{g,s}]\in \CH^{g-1+s}(\X_g^s)
$$
denote the pushforward of the fundamental class of $\PR_{g,s}$. If $s=0$, $\X_g^0=\A_g$ and
$$
[\PR_{g,0}]=[\A_1\times \A_{g-1}] \in \CH^{g-1}(\A_g)\, .
$$

The algebra of tautological classes $\R^*(\X_g^s)\subset \CH^*(\X_g^s)$ is generated over
$\R^*(\A_g)$ by the classes
$$
\big\{\, \theta_i\in \CH^1(\X_g^s)\, \big\}_{1\leq i \leq s}\ \ \  \text{and} \ \ \
\big\{\, \eta_{ij} \in \CH^1(\X_g^s)
\, \big\}_{1\leq i< j\leq s}\, ,
$$
where $\theta_i $ is the pullback of the universal $\theta$ class{\footnote{The $\theta$
class is symmetric and normalized: the restriction of $\theta$ to the zero section of $\pi:\X_g \rightarrow \A_g$ is
required to vanish.}} from the
$i^{th}$ factor  $\X_g \rightarrow \A_g$  of the universal fiber product and
$\eta_{ij}$ is the pullback of the first Chern class of the universal Poincar\'e bundle{\footnote{The Poincar\'e bundle is normalized by the standard convention: the restriction is required to be trivial on the two zero sections. By symmetry, $\eta_{ij}=\eta_{ji}$, see Section \ref{tautcomp}.}} 
from the factors $i$ and $j$ of the universal fiber product. The definitions here follow \cite {GZ, BMP, GrushevskyHulekTommasiStable, MoonenChow}. 

In Section \ref{s5454}, we define a projection operator 
$$\taut^s: \CH^*(\X_g^s) \rightarrow 
\R^*(\X_g^s)\,.$$
The projections of the generalized product loci are computed in Section 
\ref{pr13}.

\begin{thm}\label{thm:proj_product}
    Let $\kappa_{g,s} = \prod_{k=0}^{s-1}(g+\frac{k}{2})$. Then, 
    $$
    \taut^s([\PR_{g,s}])=
    \frac{g}{6\kappa_{g,s}|B_{2g}|} \,\det\begin{pmatrix}
        \theta_1 & \eta_{12}/2 & \ldots & \eta_{1s}/2\\
        \eta_{12}/2 & \theta_2 & \ldots & \eta_{2s}/2\\
        \vdots & \vdots & \ddots & \vdots \\
        \eta_{1s}/2& \eta_{2s}/2  & \ldots & \theta_s
        \end{pmatrix} \cdot  \lambda_{g-1}\, \in \R^*(\X^s_g)\, .
    $$
\end{thm}

We will see that the classes $[\PR_{g, s}]\in \CH^*(\X_g^s)$ are not tautological in general\footnote{In fact,
$[\PR_{g,s}]$ is certainly not tautological for 
$g=6$, $g=12$ and for all $g\geq 16$ even, see Remark \ref{r49}}. As in the case of product cycles on $\A_g$, a natural approach to studying $[\PR_{g, s}]$ is
to pullback  the class to the moduli space of curves. 

\begin{definition}
The Abel-Jacobi map   $\aj: \M_{g,2s} \rightarrow \X_g^s$ is defined by
\begin{equation}
\aj(C,p_1,\ldots,p_{2s})=(\mathsf{Jac}(C), \mathcal{O}_C(p_1-p_2), \mathcal{O}_C(p_3-p_4), \ldots, \mathcal{O}_C(p_{2s-1}-p_{2s}))\, .   \end{equation} The Abel-Jacobi map admits an extension to curves of compact type $$\aj:\M_{g, 2s}^{\ct}\to \X_g^s\,,$$ obtained by incorporating canonical twists over the boundary \cite{GZ2,Hai,JPPZ,P1}. \qedhere
\end{definition}

Using the techniques of Theorems \ref{thm:contributions} and \ref{thm:wallcr}, we give two proofs of the following result in Section \ref{s6363}.
\begin{thm} \label{thm:ajaj} The Abel-Jacobi pullbacks of the generalized product cycles are tautological: 
$$\aj^*([\PR_{g, s}])\in \R^{g-1+s}(\M^{\ct}_{g,2s})\, .$$
\end{thm}
We next consider the non-tautological components of the generalized product cycles,  
$$\Delta_{g,s}= [\PR_{g,s}] -\taut^s([\PR_{g,s}]) \in \CH^{g-1+s}(\X_g^s)\,.$$ 
When $s=0$, it was proven in \cite {COP} that the Torelli pullbacks of $\Delta_{g, 0}$ lie in 
the Gorenstein kernels of $\R^*(\M_{g}^{\ct}),$ whose definition we recall. 

\begin{definition}
The Gorenstein kernel, $\mathsf {K}_{g,n} \subset \R^*(\M^{\ct}_{g,n})$,
is the kernel of the $\lambda_g$-pairing,
\begin{equation}
\R^*(\M_{g,n}^{\ct}) \times \R^*(\M_{g,n}^{\ct}) \rightarrow \mathbb{Q}\, ,\ \ \ \langle \alpha,\beta \rangle = \int_{\overline{\M}_{g,n}} \overline{\alpha} \cdot\overline{\beta}\cdot \lambda_g\, . \qedhere
\end{equation}
\end{definition}

Motivated by the case $s=0$, it is natural to investigate the relationship between the generalized product cycles and the Gorenstein kernels for $s\geq 1$. In Section \ref{s65}, we show  
\begin{thm}\label{thm:inkernel} The pullback $\aj^*(\Delta_{g, 1})$ lies in the Gorenstein kernel $\mathsf K_{g, 2}\subset \R^*(\M_{g, 2}^{\ct})$ for all $g\geq 1$. In addition, 
$\aj^*(\Delta_{g, 2})$ lies in the Gorenstein kernel $\mathsf K_{g, 4}\subset \R^*(\M_{g, 4}^{\ct})$ for $1\leq g\leq 4.$ 
\end{thm}

For $s=1$, the classes $\aj^*(\Delta_{g-1,1})\in \R^{g-1}(\M_{g-1, 2}^{\mathsf{ct}})$ can be computed on the canonical partial compactification $\A_g^{\leq 1}$. More precisely, $\aj^*(\Delta_{g-1,1})$ 
and
$$
    \iota^* {\Tor}^*\left(
\left[\overline{\mathcal{A}_{1} \times 
\mathcal{A}_{g-1}}|_{\A^{\leq 1}_g}\right] -
\taut^{\leq 1}\left(\left[\overline{\mathcal{A}_{1} \times 
\mathcal{A}_{{g-1}}}|_{\A^{\leq 1}_g}\right]\right)\right)
     \in \R^{g-1}(\M_{g-1, 2}^{\mathsf{ct}})
    $$
are proven to be equal in \cite{aitordenisjeremy}. Independently of the above equality and Theorem \ref{trank1}, we can calculate $\aj^*(\Delta_{5,1})$ directly to obtain:

\begin{thm11}\label{thm:10'}
    The class $\aj^*(\Delta_{5,1})\in \R^{5}(\M_{5,2}^{\ct})$ is the generator of the 1-dimensional kernel $$\mathsf K_{5, 2}\subset \R^*(\M_{5,2}^{\ct})\, .$$
\end{thm11}

The construction
of $\aj^*(\Delta_{g,s})$ for $s>1$  
may be viewed as probing the higher torus 
rank loci of toroidal compactifications. In the simplest next case, $s=2$, we show the following:

\begin{thm}\label{thm:k44} The class 
$\aj^*(\Delta_{4,2})\in \R^{5}(\M_{4, 4}^{\mathsf{ct}})$
is a nontrivial element of the kernel $$\mathsf K_{4, 4}\subset \R^*(\M_{4,4}^{\ct})\, .$$ 
\end{thm}

As a consequence, we obtain examples of non-tautological classes on the universal family of abelian varieties and its square. 
 
\begin{cor}\label{cor:nont} The classes $[\PR_{5, 1}]%=[\A_1\times \X_4]
\in \CH^5(\X_5)$ and %is not tautological. 
%Similarly, the class 
$[\PR_{4, 2}]%=[(\A_1\times \X_3)\times_{\A_4} (\A_1\times \X_3)]
\in \CH^5(\X_4^2)$ are not tautological. 
\end{cor}
The above results suggest the following general conjecture: 
\begin{conj} \label{jecc}
Let $g\geq 1$ and $s\geq 0$.
The class 
$$\aj^*(\Delta_{g,s})\in \R^{g-1+s}(\M^{\ct}_{g,2s})$$
 lies in the
Gorenstein kernel of $\R^{g-1+s}(\M^{\ct}_{g,2s})$ and is nonzero for all but finitely many pairs of the form $(g\geq 2,s)$.
\end{conj} 
\noindent The non-vanishing statement of Conjecture \ref{jecc} implies that $[\PR_{g, s}]$ is not tautological on $\X_g^s$ for all but finitely many pairs $(g\geq 2, s).$

\subsection{Gorenstein kernel}\label{section:intro_GorensteinKernel}
The system of Gorenstein kernels $\big\{\mathsf{K}_{g,n}\subset \R^*(\M_{g,n}^{\ct})\big\}_{g,n}$ satisfies
several properties that follow directly from the definitions:
\begin{enumerate}
\item[(i)] $\mathsf{K}_{g,n} \subset \R^*(\M^{\ct}_{g,n})$ is an ideal,
\item[(ii)] $\pii_{i*}:\mathsf{K}_{g,n}\rightarrow \mathsf{K}_{g,n-1}$,
where $\pii_i:\M_{g,n}^{\ct} \rightarrow \M_{g,n-1}^{\ct}$ is the map forgetting the
$i^{th}$ marking,
\item[(iii)] $\pii^*_{i}:\mathsf{K}_{g,n-1}\rightarrow \mathsf{K}_{g,n}$,
\item[(iv)] $\xi_*: {\mathsf K}_{g_1,n_1+1}\otimes \R^*(\M^{\ct}_{g_2,n_2+1}) \rightarrow
{\mathsf K}_{g_1+g_2,n_1+n_2}$, where
$$\xi:
\M^{\ct}_{g_1,n_1+1}\times \M^{\ct}_{g_2,n_2+1} \rightarrow
\M^{\ct}_{g_1+g_2,n_1+n_2}
$$
is the gluing map.
\end{enumerate}
To state the compatibility with respect to pullback along the gluing map $\xi$, we first define the correspondence map
$$\eta: \R^*(\M^{\ct}_{g_1+g_2,n_1+n_2})\otimes \R^*(\M^{\ct}_{g_2,n_2+1}) \rightarrow 
\R^*(\M^{\ct}_{g_1,n_1+1})\, \, ,\ \ \ 
\eta(\alpha\otimes \beta) = \pii_{1*}\left(\xi^*(\alpha) \cup \pii_2^*(\beta)\right) \, ,
$$
where $\pii_i: \M^{\ct}_{g_1,n_1+1}\times \M^{\ct}_{g_2,n_2+1} \rightarrow
\M^{\ct}_{g_i,n_i+1}
$ are the projections
and
$$\pii_{1*}: \R^*(\M_{g_1,n_1+1}^{\ct}) \times
\R^*(\M_{g_2,n_2+1}^{\ct}) \rightarrow \R^*(\M^{\ct}_{g_1,n_1+1})$$
is defined by the $\lambda_{g_2}$-evaluation on the second factor. Then, we have:
\begin{enumerate}
\item[(v)] $\eta: \mathsf {K}_{g_1+g_2,n_1+n_2}\otimes \R^*(\M^{\ct}_{g_2,n_2+1}) \rightarrow
{\mathsf K}_{g_1,n_1+1}$.
\end{enumerate}
Conjecture \ref{jecc} motivates a natural question.

\begin{question}
    Is $\big\{\mathsf{K}_{g,n}\subset \R^*(\M_{g,n}^{\ct})\big\}_{g,n}$ the smallest 
    system of ideals that
        satisfies properties (i)-(v) 
        and the further condition 
        $\aj^*(\Delta_{g,s}) \in \mathsf{K}_{g,2s}$\,?
\end{question}

In other words, do the elements
$\aj^*(\Delta_{g,s}) \in \mathsf{K}_{g,2s}$ generate all the Gorenstein kernels of $\R^*(\M_{g,n}^{\ct})$
using
properties (i)-(v)? Our evidence is 
limited: the assertion is true for
$\mathsf{K}_{6,0}$, $\mathsf{K}_{7,0}$, and $\mathsf{K}_{5,2}$. These cases correspond to \cite [Theorem 5]{COP},   \cite[Proposition 40]{COP}, and Theorem \hyperref[thm:10']{10$'$} above, respectively, and depend
upon the calculations of the Gorenstein kernels $\mathsf{K}_{g,n}$ for low values of $(g,n)$ in \cite{CLS}. A conjecture, incorporating further Abel-Jacobi pullbacks, is proposed in Section \ref{KerCon}. 

\subsection{Acknowledgments} The calculation of $\Tor^*([\A_2\times \A_{g-2}]) \in \R^*(\M_g^{\ct})$ 
with the goal of determining whether
the homomorphism property holds for
$([\J_g],[\A_2\times \A_{g-2}])$ 
was a group project at ETH Z\"urich in the Fall term of 2025 with both theoretical and computational sides. We dedicate our paper to Carel Faber: his work \cite{Fa} is the starting point of several of the themes developed here, and he pioneered the computational exploration of the tautological ring
of the moduli space of curves.

We thank David Holmes for discussions about
$\A_2 \times \A_{g-2}\rightarrow \A_g$ on the Strand Noordwijk aan Zee in 2023,
Gerard van der Geer for discussions about  the homomorphism property in Leiden in November 2025, 
and 
Aaron Pixton for discussions in 
Ann Arbor in December 2025 about Conjecture \ref{jecc} and Section \ref{ajmaps}. We also benefitted from
conversations about the moduli space of abelian varieties with Sam Grushevsky, Sam Molcho, and Dan Petersen.

S.C. was supported by the SNF \emph{Ambizione} grant PZ00P2-223473.
L.D. was supported by SNF-200020-219369.
A.I.L. was supported by SNF-200020-219369. 
D.N. was supported by a Hermann-Weyl-instructorship at the Forschungsinstitut
f\"ur Mathematik at ETH Z\"urich.
R.P. was supported by SNF-200020-219369 and Swiss\-MAP.
J.S. was supported by a SwissMAP fellowship.

\section{Fiber product with the Torelli map}\label{fptm}
\subsection{Overview} Let $\replacemu=(g_1,\dots,g_k)$ be a partition of $g$ with $g_i\geq 1$. Consider the fiber diagram
\[
\begin{tikzcd}
\mathcal Z\arrow[r] \arrow[d] & \mathcal{M}_g^{\ct} \arrow[d, "\Tor"] \\
\A_{g_1}\times\dots\times \A_{g_k} \arrow[r] & \mathcal{A}_g\,. 
\end{tikzcd}
\]
We describe here the local structure of the stack $\mathcal Z$.
Our presentation follows \cite {COP} where the case $\replacemu=(1,g-1)$  was treated. The structure result for $\mathcal Z$ for all partitions $\replacemu$ is proven by Drakengren \cite{Drakengren}.

\subsection{Proof of Theorem \ref{th:fiberprod_description}} The starting point of the analysis of the
stratification of the 
fiber product $\mathcal Z$  by $\replacemu$-colored extremal trees is the following result \cite[Corollary 3.23]{CG}.
\begin{prop}
    A principally polarized abelian variety $(X,\theta)$ decomposes uniquely, up to reordering, as a product of indecomposable principally polarized abelian varieties.
\end{prop}
As an immediate consequence, we can describe when the Jacobian of a compact type curve lies in a product locus.
\begin{cor}\label{cor:prodjac}
    Let $C$ be a compact type curve of genus $g$ with Jacobian isomorphic (as principally polarized abelian varieties) to a product 
    \[
    \mathsf{Jac }(C) \cong X_{g_1} \times\dots\times X_{g_k} \text{ with } (X_{g_i},\theta_i)\in \A_{g_i}\,.
    \]
    Then there is a partition of $\mathrm{Irr}(C)_{>0}$, the set of irreducible components of $C$ of positive genus, into disjoint nonempty subsets $$\mathrm{Irr}(C)_{>0}=S_1\sqcup \ldots \sqcup S_k$$ such that, for all $1\leq i\leq k$, we have $$\prod_{D\in S_i} \mathsf{Jac}(D)\cong X_{g_i}\,.$$ In particular, for all $1\leq i\leq k$, we have $\sum_{D\in S_i} g(D)=g_i$\,.
\end{cor}

Using Corollary \ref{cor:prodjac}, we stratify the fiber product $\mathcal Z=\Tor^{-1}(\A_{g_1}\times\dots\times \A_{g_k})$. The strata are indexed by $\replacemu$-colored extremal trees $\mathsf T$ as defined  in the statement of Theorem \ref{th:fiberprod_description} of Section \ref{tfp}. For each such tree $\mathsf T$, we have
\[
\M_\mathsf{T}^{\ct} = \prod_{v\in \mathsf{V}(\mathsf{T})} \M_{\mathsf{g}(v),\mathsf{n}(v)}^{\ct}\,,
\]
where $\mathsf{n}(v)$ is the valence of $v$. Using the coloring $\mathsf{c}$, we obtain a canonical Torelli map
\[
\Tor_{\mathsf T}:\M_{\mathsf{T}}^{\ct}\rightarrow \A_{g_1}\times \dots\times \A_{g_k}\,, \quad C\mapsto \mathsf {Jac} (C_1)\times \dots \times \mathsf {Jac}(C_k)\,,\] where the subcurve $C_i$ is the union of all positive-genus components of $C$  whose corresponding vertices are assigned the color $i$. 
Let $$\xi_{\mathsf{T}}:\M_{\mathsf{T}}^{\ct}\rightarrow \M_g^{\ct}$$ be the gluing map associated to the underlying tree $\mathsf{T}$. Since $\xi_{\mathsf{T}}$ and $\Tor_{\mathsf T}$ agree after mapping to $\A_g$, we obtain a canonical map
\[
\epsilon_\mathsf T:\M_\mathsf{T}^{\ct}\rightarrow \mathcal Z\,.
\] The images of these maps stratify the fiber product $\mathcal Z$. 

The next definition allows us to describe incidences between the various strata constructed above. 

\begin{definition}\label{defdef}
Let $\mathsf{T}, \mathsf T'$ be two $\replacemu$-colored extremal
trees with color assignments $$\mathsf c:\mathsf V(\mathsf T)_{>0}\to \{1, \ldots, k\}\, , \quad \mathsf c':\mathsf V(\mathsf T')_{>0}\to \{1, \ldots, k\}\, .$$
A $\mathsf{T}'$-structure on $\mathsf{T}$
is given by a 
partition
of the vertex set of $\mathsf{T}$,
$$\mathsf{V}({\mathsf{T}})=\mathcal V_1\cup \ldots \cup \mathcal V_m\,,$$
together with a surjective map
$$\phi: \mathsf{V}(\mathsf{T}) \rightarrow
\mathsf V(\mathsf T')$$
satisfying the following properties:
\begin{enumerate}
\item[(i)] For each $1\leq i\leq m$, the vertex subset $\mathcal V_i$ determines a connected nonempty subtree of $\mathsf T$.
\item [(ii)] The assignment $\phi$ sends all vertices in $\mathcal V_i$ to the same vertex $\phi(\mathcal V_i)\in \mathsf V(\mathsf T')$. If $i\neq j$, then $$\phi(\mathcal V_i)\neq \phi(\mathcal V_j)\,.$$
\item [(iii)] 
We have $$\sum_{v\in \mathcal V_i} \mathsf g(v)= \mathsf g (\phi(\mathcal V_i))\,.$$ Moreover, all positive genus vertices $v\in \mathcal V_i$ (if any) have the same color and in this case $$\mathsf c(v)=\mathsf c'(\phi(\mathcal V_i))\,.$$ 
\item[(iv)] An edge $e'\in \mathsf{E}(\mathsf{T}')$ connects the vertices
$v',w'\in \mathsf{V}(\mathsf{T'})$ with $v'=\phi(\mathcal V_i)$ and $w'=\phi(\mathcal V_j)$ if and only
if there exists an edge $e\in \mathsf{E}(\mathsf{T})$ which connects a vertex
of $\mathcal{V}_{i}$ to a vertex
of $\mathcal{V}_{j}$.\qedhere
\end{enumerate}

\end{definition}

Item (ii) implies that the sets $\mathcal V_1, \ldots, \mathcal V_m$ are uniquely determined as the fibers of the surjection $\phi$. We note that if $\mathsf T$ has a $\mathsf T'$-structure, then we must have $|\mathsf V(\mathsf T')|\leq|\mathsf V(\mathsf T)|$, with equality if and only if $\phi$ is an isomorphism preserving the genus and color assignments. 

We define $\mathsf{T}$ to be a {\em nontrivial degeneration} of $\mathsf{T}'$ if $\mathsf{T}$ carries a nontrivial $\mathsf{T}'$-structure.
We also refer to $\mathsf{T}'$ as a \emph{smoothing} of $\mathsf{T}$. For each  degeneration $\T$ of $\T'$, the assignment $\phi$ determines a gluing map
\[
\xi_{\T',\T}: \M_{\T}^{\ct}\rightarrow \M_{\T'}^{\ct}\,.
\] 

An irreducible $\replacemu$-colored extremal tree $\mathsf I$ is a $\replacemu$-colored extremal tree without genus $0$ vertices. Irreducible trees $\mathsf I$ admit no nontrivial smoothings and index the irreducible components of the fibered product $\mathcal Z$. Let ${\mathsf I}_1, \ldots, {\mathsf I}_{\ell}$ be irreducible $\replacemu$-colored extremal trees. The intersection of the  images of $\M^{\ct}_{\mathsf {I}_1}, \ldots, \M^{\ct}_{\mathsf{I}_\ell}$ in $\mathcal Z$ is the union of the images of $\M^{\ct}_{\mathsf T}$ taken over all $\replacemu$-colored extremal trees $\mathsf T$ which admit $\mathsf{I}_1, \ldots, \mathsf{I}_{\ell}$-structures. \qed

 \vspace{8pt}
The strict stratum $\M^{\ct,\star}_\T$ associated to a $\replacemu$-colored extremal tree $\T$ is the complement of the images of all gluing maps $\xi_{\mathsf T', \mathsf T}$
for all $\replacemu$-colored nontrivial degenerations $\T'$ of $\T$.
Of course, the strict strata also cover $\mathcal Z.$ 

\subsection{Local equations}
We now describe the stack structure on the fibered product $\mathcal Z$. The map 
\[
\A_{g_1}\times\dots\times\A_{g_k}\rightarrow \A_g
\]
is not an embedding, but the differential is an injection on tangent spaces. Therefore, locally in the analytic topology, $\mathcal Z$ is a closed substack of $\M_g^{\ct}$. Let $x\in \mathcal Z$ be a point in the strict stratum $\M^{\ct,\star}_{\T}$. We take a small open neighborhood $W$ of $x$, which we may assume is the versal deformation space of $x$. We have a map
\[
\nu: W\rightarrow \prod_{e\in \mathsf{E}(\mathsf {T})} \mathbb{C}_e
\]
to the deformation space of the nodes of the underlying curve $C$. Here, $\mathbb{C}_e$ is the one-dimensional versal deformation space of $C$ corresponding to a node, and hence an edge $e$ of $\mathsf{T}$. 

Let $z_e$ be the standard coordinate for each factor $\cc_e$. We write local equations near $x$ using these coordinates. 

\begin{definition} Let $v_1,v_2\in \mathsf{V}(\T)_{>0}$ such that $\mathsf{c}(v_1)\neq \mathsf{c}(v_2)$. Let $\mathsf{P}$ be the minimal path of edges connecting $v_1$ to $v_2$. We say that $\mathsf{P}$ is \emph{critical} if the only vertices of positive genus on the path are $v_1$ and $v_2$. \end{definition}

To each critical path $\mathsf P_{\textrm{crit}},$ we associate the monomial
\[
\mathsf{Mon}(\mathsf{P}_{\mathrm{crit}})=\prod_{e} z_{e}\,,
\] where the product is taken over all edges appearing in $\mathsf P_{\textrm{crit}}.$ We have the following generalization of \cite [Proposition 29]{COP}. 
\begin{thm}[Drakengren \cite{Drakengren}]\label{prop:localequations}
    The local equations for $\mathcal Z$ near the point $x$ in the strict stratum indexed by the $\replacemu$-colored extremal tree $\mathsf{T}$ are given by pullback from $\prod_{e\in \mathsf{E}(\mathsf{T})} \cc_e$ of the monomial set
    \[
    \{\mathsf{Mon}(\mathsf{P}_{\mathrm{crit}}): \mathsf{P}_{\mathrm{crit}} \text{ is a critical path} \}\subset \cc[\{z_e\}_{e\in \mathsf{E}(\mathsf{T})}]\,.
    \]
    In particular, $\mathcal{Z}$ is a reduced Deligne-Mumford stack.
\end{thm}

In \cite{COP}, the local structure of $\mathcal{Z}$ was determined in case $\replacemu=(1,g-1)$ by interpreting $\mathcal{Z}$ as a moduli space of stable
maps to a moving elliptic curve. The
deformation theory of stable maps was then used to prove reducedness.
For the general case, an interpretation in terms of stable maps is not available.
Drakengren \cite{Drakengren}
directly studies 
the Torelli map near the boundary of 
$\M_{g}^{\ct}$ using the results of \cite{HN}.

\section{Excess intersection calculation}\label{exccalc}
\subsection{Overview}
We prove a graph sum formula for 
\begin{equation} \label{ce44}
\Tor^*([\A_{g_1}\times\dots\times\A_{g_k}])
\in \R^*(\M_g^{\ct})
\end{equation}
using the method of \cite[Section 5]{COP}. The calculations here generalize those of \cite{COP}, which concern only the case $\replacemu=(1,g-1)$.

\subsection{Tautological classes} \label{frfr4}
Before studying the pullback \eqref{ce44}, we review first a few standard definitions regarding the tautological ring of the moduli space of curves $\overline{\mathcal M}_{g, n}$. The tautological classes of $\M_{g, n}^{\ct}$ are obtained by restriction $$\mathsf R^*(\overline{\mathcal M}_{g, n})\to \mathsf R^*(\M_{g, n}^{\ct})\to 0\,.$$ A detailed discussion of the tautological rings of $\Mbar_{g, n}$ can be found in \cite {FP3, P1, P2}. The following tautological classes will appear in our calculations:
\begin{itemize}
\item the cotangent classes $\psi_1, \ldots, \psi_n$ given as $$\psi_1=c_1(\mathbb L_1),\ldots , \psi_n=c_1(\mathbb L_n)\,,$$ where the fiber of $\mathbb L_i$ over $(C, x_1, \ldots, x_n)$ equals $T^*_{x_i}C\,,$
\item $\lambda$-classes obtained as Chern classes of the Hodge bundle $\mathbb E_g\to \M_{g, n}^{\ct}$, $$\lambda_i=c_i(\mathbb E_g)\,,$$ pulled back from $\A_g$ under the Torelli map, 
\item the $\kappa$-classes, $$\kappa_i=\pi_*\left(c_1(\omega_\pi(s_1+\ldots+s_n))^{i+1}\right)$$ where $s_1, \ldots, s_n$ are the universal sections of the universal curve over $\M_{g, n}^{\ct}.$
\end{itemize}

The tautological ring $\R^*(\M^{\ct}_{\mathsf{T}}) \subset \CH^*(\M_{\mathsf{T}}^{\ct})$
for any $\replacemu$-colored extremal tree is
generated by the pullbacks of the
tautological rings of the factors of the product
$$\mathcal M_{\mathsf T}^{\mathsf {ct}} = \prod_{v \in \mathsf {V}(\mathsf T)} \M_{\mathsf g(v), \mathsf n(v)}^{\ct}\, .$$ Pushforward under the gluing map $\xi_{\mathsf T}: \M_{\mathsf T}^{\mathsf {ct}}\to \M_g^{\ct}$ sends tautological classes on $\mathcal M_{\mathsf T}^{\mathsf {ct}}$ to tautological classes on $\M_g^{\ct}.$
\subsection{Proof of Theorem \ref{thm:contributions}}
We will express the class \eqref{ce44}
as a
sum of contributions $\mathsf{Cont}_\mathsf{T}$ supported on $\M^{\ct}_\mathsf{T}$ for each $\replacemu$-colored extremal tree $\mathsf{T}$. Only trees $\mathsf{T}$ with at most
$$
d=\mathsf{cod}_{\underline {g}}=\codim_{\A_g}(\A_{g_1}\times\dots\times \A_{g_k})=\sum_{1\leq i<j\leq k}g_i g_j
$$
edges contribute. The corresponding contributions are computed inductively, with the base cases provided by the  irreducible $\replacemu$-colored extremal trees.

Each contribution is expressed in terms of the Chern classes of the normal bundle of the product $$\A_{g_1}\times\dots\times \A_{g_k}\to \A_g$$ and the Chern classes of the normal bundles of the strata of $\mathcal Z$ indexed by the $\replacemu$-colored extremal trees. These contributions can be found using excess residual intersections as in \cite[Chapter IX]{Fulton}. The exact residual terms are {\it universal expressions} depending only on the normal bundle data, so we can compute them in a convenient local model. The local equations in Theorem \ref{prop:localequations} are used for the local model
calculations.

Let $c$ be the number of critical paths between all possible unordered pairs of positive genus vertices in $\T$ that have distinct colors. A simple check shows that $c\leq d$. 
Let $\N$ be a rank $d$ vector bundle on $\prod_{e\in \mathsf{E}(\T)}\cc_e$ of the form
\[
\N = \O^{c}\oplus L_1\oplus \dots\oplus L_{d-c}\,,
\]
where the $L_i$ are arbitrary torus equivariant line bundles on $\prod_{e\in \mathsf{E}(\T)}\cc_e$. Set 
\[
s = (\mathsf{Mon}(\mathsf{P}^1_{\mathrm{crit}}),\dots,\mathsf{Mon}(\mathsf{P}^c_{\mathrm{crit}}),0,\dots,0)\,.
\]
Let $\mathsf{Z}(s) \subset \prod_{e\in \mathsf{E}(\T)}\cc_e$ be the zero locus of the section $s$. The zero locus $\mathsf{Z}(s)$
has precisely the equations of our local model by Theorem
\ref{prop:localequations} and
is independent of the ordering of the
monomials in $s$. 
We may calculate the contribution of $\mathsf{T}$ in the excess intersection problem for $c_{d}(\N)$ determined by $\mathsf{Z}(s)$. We start with the equation
\[
c_d(\N) = (\ell_1\cdots\ell_{d-c})\prod_{\mathsf{P}_{\mathrm{crit}}} (\sum_{e\in \mathsf{P}_{\mathrm{crit}}} z_e)\,,
\]
where the $\ell_i$ are the equivariant first Chern classes of the line bundles $L_i$.

\begin{itemize}
    \item [(i)] Consider first the case where the $\replacemu$-colored extremal tree $\mathsf{I}$ has no genus $0$ vertices. Then, $\mathsf{I}$ corresponds to an irreducible component of $\mathcal Z$, and every edge  of $\mathsf{I}$ is a critical path. The contribution $\mathsf{Cont}_\mathsf{I}$ can be computed by the usual excess intersection formula:
    \begin{equation}\label{compcont}
    \mathsf{Cont}_\mathsf{I}=\left[\frac{c(\N)}{\prod_{e\in \mathsf E(\mathsf I)} \,(1+z_e)\,}\right]_{d-c}\,.
    \end{equation} The subscript indicates that only the part of degree $d-c$ is taken. 
    The pushforward $\iota_{{\mathsf I}*}$ to the ambient torus-equivariant $\prod_{e\in \mathsf{E}(\T)}\cc_e$ is found by multiplying by the top Chern class of the normal bundle, which equals $z_1\cdots z_c$. Therefore,
    \[
    \iota_{\mathsf{I}*}\mathsf{Cont}_{\mathsf{I}}=z_1\cdots z_c\left[\frac{c(\N)}{\prod_{e\in \mathsf E(\mathsf I)} \,(1+z_e)\,}\right]_{d-c}\,.
    \]
    \item [(ii)] Next,  let $\mathsf{T}$ be an arbitrary $\replacemu$-colored extremal tree. By induction, we can assume we have computed $\mathsf{Cont}_{\mathsf{T}'}$ for all smoothings $\mathsf{T}'$ of $\mathsf{T}$.
    We set
    \begin{equation}\label{inductive}
    \mathsf{Cont}_{\mathsf{T}}\cdot\prod_{e\in \mathsf{E}(\mathsf{T})} z_e=c_{d}(\N)-\sum_{\mathsf{T}'}\iota_{\mathsf{T}'*}\mathsf{Cont}_{\mathsf{T}'}\,.
    \end{equation}
    Equation \eqref{inductive} determines $\mathsf{Cont}_{\mathsf{T}}$.   
    \end{itemize}

The formula for
$\mathsf{Cont}_{\mathsf T}$
in terms of tautological classes is then obtained via substitution of
variables:
\begin{itemize}
    \item[$\bullet$] we
replace each edge variable $z_e$ by the
normal factor corresponding
to the smoothing of the edge $e$ (the sum of tangent
lines corresponding to the
two half-edges of $e$),
 \item[$\bullet$] we replace 
the Chern classes of $\mathcal{N}$
by the Chern classes of the normal bundle of the 
product
$$\A_{g_1}\times\dots\times \A_{g_k}\to \A_g\, .$$
\end{itemize}
In the end, $\mathsf{Cont}_{\mathsf T}$ is expressed in terms
of tautological $\psi$ and $\lambda$ classes obtained from the
moduli of curves. \qed

\subsection{Examples}
We present  full calculations of the contributions of Theorem \ref{thm:contributions} in a few cases.

\begin{example}
    Let $\mathsf{I}$ be an irreducible $\replacemu$-colored extremal tree, and let $$\mathsf V(\mathsf I)=S_1\sqcup \ldots \sqcup S_k$$ be the partition of the vertices determined by the $k$ colors. Then, $\M_{\mathsf{I}}^{\ct}$ determines
    an irreducible component of $\mathcal Z=\Tor^{-1}(\A_{g_1}\times\dots\times \A_{g_k})$ with contribution calculated by \eqref{compcont}. The normal bundle to the product $\A_{g_1}\times\dots\times \A_{g_k}\rightarrow \A_g$ is 
    \[
    \Sym^2(\mathbb{E}_{g_1}^{\vee}\boxplus\dots \boxplus \mathbb{E}_{g_k}^{\vee})-\sum_{i=1}^k \Sym^2(\mathbb{E}_{g_i}^{\vee}) = \sum_{i < j} \mathbb{E}^{\vee}_{g_i}\boxtimes\mathbb{E}^{\vee}_{g_j}\,.
    \]
    When pulled back to $\M^{\ct}_{\mathsf{I}}$, the normal bundle splits as 
    \[
    \sum_{i<j} \left(\bigoplus_{v\in S_i}\mathbb{E}^{\vee}_{\mathsf g(v)}\right)\boxtimes \left(\bigoplus_{w\in S_j}\mathbb{E}^{\vee}_{\mathsf g(w)}\right)\,.
    \]
    The normal bundle of $\M^{\ct}_{\mathsf{I}}$ in $\M_g^{\ct}$ is the sum of factors corresponding to the smoothings of each node. Therefore, the contribution of $\mathsf{I}$ is
    \begin{equation}\label{excess}
     \mathsf{Cont}_\mathsf{I}=
    \left[\frac{\prod_{i<j} \prod_{v\in S_i}\prod_{w\in S_j} c\left(\mathbb{E}^{\vee}_{\mathsf g(v)}\boxtimes \mathbb E^{\vee}_{\mathsf g(w)}\right)}{\prod_{e\in \mathsf{E}(\mathsf{I})} (1-\psi_e'-\psi''_e)}\right]_{d-c}\,,\end{equation} where $\psi'_e, \psi''_e$ are the cotangent classes at the node associated to $e$. 
\end{example}

For the partition $\replacemu=(1,g-1)$, several explicit computations involving trees $\mathsf T$ which are not irreducible are detailed in \cite[Examples 30-35]{COP}. We record here a new example. 

\begin{example} 
    Consider the partition $\replacemu=(2,4)$ for $g=6$. Then, $d=8$. We use two colors, blue and green, to paint the positive genus vertices. Consider the following $\replacemu$-colored extremal tree $\T$:

    \begin{center}
\begin{tikzpicture}[scale=.9]

    \node[circle, fill=gray!40, minimum size=6pt, inner sep=5.5pt] (T) at (0,0) {};
    \node[circle, fill=RoyalBlue!20, minimum size=6pt, inner sep=3pt] (A) at (-1.4,-1.3) {1};
    \node[circle, fill=RoyalBlue!20, minimum size=6pt, inner sep=3pt] (B) at (1.4,-1.3) {1};
    \node[circle, fill=YellowGreen!30, minimum size=6pt, inner sep=3pt] (C) at (0,1.6) {4};

    \draw[very thick] (T) -- (A);
    \draw[very thick] (T) -- (B);
    \draw[very thick] (T) -- (C);

\end{tikzpicture}
\end{center}
Here, the gray vertex of $\mathsf T$ has genus $0$. We decorate the edge incident to the vertex of genus $4$ by $z_3$, while $z_1$, $z_2$ correspond to the remaining two edges. The tree $\mathsf T$ admits two smoothings $\mathsf R$ and $\mathsf{S}$ respectively:\vskip.2in

\begin{center}
\begin{tikzpicture}[scale=.55]

\node[circle, fill=YellowGreen!30, minimum size=6pt, inner sep=3pt] (D) at (0-6,0) {4};
    \node[circle, fill=RoyalBlue!20, minimum size=6pt, inner sep=3pt] (E) at (3-6,0) {2};

    \draw[very thick] (D) -- (E);

    \node[circle, fill=RoyalBlue!20, minimum size=6pt, inner sep=3pt] (A) at (3,0) {1};
    \node[circle, fill=YellowGreen!30, minimum size=6pt, inner sep=3pt] (B) at (6,0) {4};
    \node[circle, fill=RoyalBlue!20, minimum size=6pt, inner sep=3pt] (C) at (9,0) {1};

    \draw[very thick] (A) -- (B) -- (C);

\end{tikzpicture}
\end{center}
\vskip.1in
The smoothings $\mathsf{R}$ and $\mathsf{S}$ correspond to irreducible components of $\Tor^{-1}(\A_2\times \A_{4})$. Therefore, we have
$$\iota_{\mathsf{R}*}\mathsf{Cont}_{\mathsf{R}}
    =z_3\left[\frac{c(\mathcal N)}{1+z_3}\right]_{7}, \quad  
    \iota_{\mathsf{S}*}\mathsf{Cont}_{\mathsf{S}}
    =z_1z_2\left[\frac{c(\mathcal N)}{(1+z_1)(1+z_2)\,}\right]_{6}\,,$$
    where $$c(\mathcal N)=(1+\ell_1)\cdots (1+\ell_6)(1+z_1+z_3)(1+z_2+z_3)\,.$$
     By \eqref{inductive}, we have 
    \begin{align*}
    \mathsf{Cont}_{\mathsf{T}}\cdot z_1z_2z_3=(z_1+z_3)(z_2+z_3)(\ell_1\cdots\ell_6)-\iota_{\mathsf R*}\mathsf{Cont}_{\mathsf{R}}-\iota_{\mathsf S*}\mathsf{Cont}_{\mathsf{S}}\,.
    \end{align*} 
    After solving for $\mathsf {Cont}_{\mathsf T}$, we obtain: 
    \begin{align*}
    \mathsf{Cont}_{\mathsf{T}}
    =&
    -3 c_5(\N)
    +(4z_1 + 4z_2 + 6z_3)\, c_4(\N)
    - (5z_1^2 + 5z_1z_2 + 10z_1z_3 + 5z_2^2 + 10z_2z_3 + 10z_3^2)\,c_3(\N)
    \\
    +& (6z_1^3 + 6z_1^2z_2 + 15z_1^2z_3 + 6z_1z_2^2 + 15z_1z_2z_3 + 20z_1z_3^2 + 6z_2^3 + 15z_2^2z_3 + 20z_2z_3^2 + 15z_3^3) \,c_2(\N)
    \\
    -& (7z_1^4 + 7z_1^3z_2 + 21z_1^3z_3 + 7z_1^2z_2^2 + 21z_1^2z_2z_3 + 35z_1^2z_3^2 + 7z_1z_2^3 
    \\
    &+ 21z_1z_2^2z_3 + 35z_1z_2z_3^2 + 35z_1z_3^3 + 7z_2^4 + 21z_2^3z_3 + 35z_2^2z_3^2 + 35z_2z_3^3 + 21z_3^4)\,c_1(\N)
    \\
    +& (8z_1^5 + 8z_1^4z_2 + 28z_1^4z_3 + 8z_1^3z_2^2 + 28z_1^3z_2z_3 + 56z_1^3z_3^2 + 8z_1^2z_2^3 + 28z_1^2z_2^2z_3+56z_1^2z_2z_3^2
    \\
    &+   70z_1^2z_3^3 + 8z_1z_2^4 + 28z_1z_2^3z_3 + 56z_1z_2^2z_3^2 + 70z_1z_2z_3^3 + 56z_1z_3^4 + 8z_2^5 + 28z_2^4z_3 + 56z_2^3z_3^2
    \\
    &+ 70z_2^2z_3^3 + 56z_2z_3^4 + 28z_3^5)\,. \qedhere
    \end{align*}
 
\end{example}
\subsection{Computer implementation}\label{compimp}

Theorem \ref{thm:A3A3_A2Ax} concerns partitions with $2$ parts $\replacemu=(h, g-h)$, specifically  for $\replacemu$ in the set 
\begin{equation}\label{list} \{\,(2,2)\, , (2,3) \,, (2, 4)\, , (3, 3)\, ,(2, 5)\, , (2, 6)\, %, (2,7)\, , (2,8)\,
\}\,.\end{equation} 
For these cases, the number of $\replacemu$-colored extremal trees with at most $d$ edges is 
$$9\,, 37\,,  153\, , 210\, ,622\, , 2569\,,% 10559\,, 43684,
$$ respectively. A computer calculation yields the required Torelli pullbacks in $\mathsf R^*(\M_g^{\ct}).$

We have implemented the computation of $\mathsf{Cont}_{\mathsf{T}}$ using \eqref{inductive} in the \texttt{Julia} package $\mathsf{TorelliTrees.jl}$ \cite{code}.
Given the partition $\replacemu$, we first enumerate all $\replacemu$-colored extremal trees up to isomorphism in increasing order of the number of vertices. The smoothings are obtained by first computing the \emph{minimal smoothings}: 
the smoothing that can not be obtained by applying two or more non-trivial smoothings.
Each edge $\{v_1,v_2\}\in \mathsf{E}(\T)$ with $\mathsf{g}(v_1)\mathsf{g}(v_2)=0$ determines a unique minimal smoothing. 
All other smoothings are obtained as iterations of minimal smoothings.
We then apply \eqref{inductive} to obtain $\textsf{Cont}_{\T}$ for each $\T$.

Finally, we substitute for $\{z_e\}_{e\in\mathsf{E}(\T)}$ and the Chern classes of $\mathcal{N}$ as explained above.
The output generated is $\mathsf{admcycles}$ code, which, when executed, provides an expression for $$\Tor^*([\A_{h}\times \A_{g-h}])\in \R^*(\M_g^{\ct})$$ in terms of tautological classes. 

To illustrate, the $9$ graphs with at most $4$ edges which appear for the partition $\replacemu=(2, 2)$ are shown below. There are two colors, blue and green, and as usual, the gray vertices have genus $0$. The Torelli fiber product has $4$ irreducible components corresponding to the following graphs: 
\vskip.2in
\begin{center}
\begin{tikzpicture}[
  v/.style= {circle, fill=gray!40, minimum size=6pt, inner sep=3pt},
  vg/.style= {circle, fill=YellowGreen!30!, minimum size=6pt, inner sep=3pt},
   vb/.style= {circle, fill=RoyalBlue!30!, minimum size=6pt, inner sep=3pt},
  e/.style={very thick},
x=1.5cm, y=1.5cm
]

\node[vb] (A1) at (0,0) {$2$};
\node[vg] (A2) at (1,0) {$2$};
\draw[e] (A1)--(A2);

\node[vb] (BL1) at (2.5,0) {$1$};
\node[vg] (BL2) at (3.5,0) {$2$};
\node[vb] (BL3) at (4.5,0) {$1$};
\draw[e] (BL1)--(BL2)--(BL3);

\node[vg] (BR1) at (6, 0) {$1$};
\node[vb] (BR2) at (7,0) {$2$};
\node[vg] (BR3) at (8,0) {$1$};
\draw[e] (BR1)--(BR2)--(BR3);

\node[vb] (C1) at (0+3,-1) {$1$};
\node[vg] (C2) at (1+3, -1) {$1$};
\node[vb] (C3) at (2+3,-1) {$1$};
\node[vg] (C4) at (3+3,-1) {$1$};
\draw[e] (C1)--(C2)--(C3)--(C4);

\end{tikzpicture}
\end{center}
The remaining $5$ graphs correspond to the intersections of these components which appear in the calculation.  
\vskip.2in
\begin{center}
\begin{tikzpicture}[
  v/.style= {circle, fill=gray!40, minimum size=6pt, inner sep=3pt},
  vg/.style= {circle, fill=YellowGreen!30!, minimum size=6pt, inner sep=3pt},
   vb/.style= {circle, fill=RoyalBlue!30!, minimum size=6pt, inner sep=3pt},
  e/.style={very thick},
x=1.4cm, y=1.3cm
]

\node[vb] (D1) at (0,-1) {$1$};
\node[v] (D2) at (1,-1) {\phantom{0}};
\node[vb] (D3) at (2,-1) {$1$};
\node[vg] (D4) at (1,+0) {$2$};
\draw[e] (D1)--(D2)--(D3);
\draw[e] (D2)--(D4);

\node[vg] (E1) at (3.5,-1) {$1$};
\node[v] (E2) at (4.5,-1) {\phantom{0}};
\node[vg] (E3) at (5.5,-1) {$1$};
\node[vb] (E4) at (4.5,+0) {$2$};
\draw[e] (E1)--(E2)--(E3);
\draw[e] (E2)--(E4);

\node[vb] (F1) at (1+6,-1) {$1$};
\node[v] (F2) at (2+6,-1) {\phantom{0}};
\node[vb] (F3) at (3+6,-1) {$1$};
\node[vg] (F4) at (2+6,0) {$1$};
\node[vg] (F5) at (2+6,-2) {$1$};
\draw[e] (F1)--(F2)--(F3);
\draw[e] (F2)--(F4);
\draw[e] (F2)--(F5);

\node[vb] (G1) at (0+1,-3) {$1$};
\node[vg] (G2) at (1+1,-3) {$1$};
\node[v] (G3) at (2+1,-3) {\phantom{0}};
\node[vb] (G4) at (3+1,-3) {$1$};
\node[vg] (G5) at (2+1,-4) {$1$};
\draw[e] (G1)--(G2)--(G3)--(G4);
\draw[e] (G3)--(G5);

\node[vg] (I1) at (4.5+1,-3) {$1$};
\node[vb] (I2) at (5.5+1,-3) {$1$};
\node[v] (I3) at (6.5+1,-3) {\phantom{0}};
\node[vb] (I4) at (7.5+1,-3) {$1$};
\node[vg] (I5) at (6.5+1,-4) {$1$};
\draw[e] (I1)--(I2)--(I3)--(I4);
\draw[e] (I3)--(I5);
\end{tikzpicture}
\end{center}

\subsection {Proof of Theorem \ref{thm:A3A3_A2Ax}}\label{general} For the $6$ partitions $\replacemu=(h, g-h)$ listed in \eqref{list}, the procedure described in Section \ref{compimp} confirms that \begin{equation}\label{vanishgr}\Tor^*\left([\A_{h}\times \A_{g-h}]\right)=\Tor^*\left(\taut ([\A_{h}\times \A_{g-h}])\right)\in \R^*(\M_g^{\ct})\,.\end{equation}
The homomorphism property stated in Theorem \ref{thm:A3A3_A2Ax} then follows. Indeed, pushing forward \eqref{vanishgr} under $\Tor$, we obtain
$$
[\A_{h}\times \A_{g-h}]\cdot [\J_g]=\mathsf{taut}([\A_h\times \A_{g-h}])\cdot [\mathcal J_g]\in \CH^*(\A_g)\,.
$$
Applying the operator $\mathsf{taut}$ and using \eqref{tss}, we conclude
$$
\mathsf{taut}([\A_h\times \A_{g-h}]\cdot [\J_g])=\mathsf{\taut}(\taut([\A_h\times \A_{g-h}])\cdot [\J_g])=\mathsf{taut}([\A_h\times \A_{g-h}])\cdot \taut([\J_g])\,,
$$
as required. \qed

\vspace{8pt}

Motivated by the calculations for $g\leq 8$, we expect the equality \begin{equation}\label{vanishx}\Tor^*([\A_{2}\times \A_{g-2}])\stackrel{?}{=}\Tor^*(\taut ([\A_{2}\times \A_{g-2}]))\in \mathsf R^{2g-4}(\M_g^{\ct})\,\end{equation} to hold for all $g\geq 2$. The homomorphism property for the pair $([\A_2\times \A_{g-2}], [\mathcal J_g])$
is then implied.

Equality \eqref{vanishx} is equivalent to an identity for Hodge integrals. The $\lambda_g$-pairing $$\mathsf{R}^1(\M_g^{\ct})\times \mathsf{R}^{2g-4}(\M_g^{\ct})\to \mathsf R^{2g-3}(\M_g^{\ct})\stackrel{\epsilon}{\longrightarrow} \mathbb Q$$ is perfect  by \cite[Theorem 4]{CLS}.  Furthermore, $\mathsf{R}^1(\M_g^{\ct})$ is generated by the boundary divisors and the class $\kappa_1$, see \cite [Theorem 1]{arbarellocornalba}. The proof of \cite [Theorem 4]{COP} shows that both sides of \eqref{vanishx} intersect the boundary divisors trivially. The key ingredient here is the assertion that all product loci in $\A_g$ intersect trivially \cite [Proposition 17]{COP}, so in particular 
$$[\A_2\times \A_{g-2}]\cdot [\A_r\times \A_{g-r}]=0 \in \CH^*(\A_g)\,.$$ Thus, equation \eqref{vanishx} is equivalent to 
\begin{equation} \label{cvvbx}
    \kappa_1\cdot \Tor^*([\A_{2}\times \A_{g-2}])=\kappa_1\cdot \Tor^*(\taut([\A_{2}\times \A_{g-2}]))\in \R^{2g-3}(\M_g^{\ct})\,.
    \end{equation}
    By applying Theorem \ref{thm:contributions} and the $\lambda_g$-evaluation $\epsilon$,  \eqref{cvvbx} can be equivalently rewritten as an identity involving Hodge integrals: \begin{equation}\label{numericalx}\sum_{\mathsf T \in \mathsf{Tree}(2, g-2)} \frac{1}{|\mathsf{Aut}(\mathsf{T})|} \int_{{\overline{\mathcal M}_{\mathsf T}}} \xi_{\mathsf T}^*(\kappa_1\cdot \lambda_g)\cdot 
    \mathsf{Cont}_{\mathsf T}=\int_{\Mbar_g} \kappa_1\cdot \lambda_g\cdot \taut([\A_2\times \A_{g-2}])\,,\end{equation} where $\mathsf {Cont}_{\mathsf T}$ denotes the canonical extension to the compactified space $$\overline{\mathcal M}_{\mathsf T} = \prod_{v \in \mathsf {V}(\mathsf T)} \overline{\M}_{\mathsf g(v), \mathsf n(v)}\,.$$
Using the formula for $\taut([\A_2\times \A_{g-2}])$ in Section \ref{gconag}, the right side of \eqref{numericalx} becomes $$\frac{1}{360} \cdot \frac{g(g-1)}{|B_{2g}| |B_{2g-2}|} \int_{\Mbar_g} \kappa_1 \cdot \lambda_g \cdot \lambda_{g-1} \lambda_{g-3}=\frac{g(g-1)}{1440}\cdot \frac{1}{(2g-2)!}\,,$$ where the Hodge integral above is computed by \cite[Theorem 3]{FP2}, in particular equation (50) therein. The evaluation of the left hand side of \eqref{numericalx} for all $g\geq 2$ is an open question. 

By similar reasoning, the homomorphism property for the pair $([\A_1\times \A_1\times \A_{g-2}], \,[\J_g])$, established in \cite {aitordenisjeremy}, is equivalent to the identity  
\begin{equation}\label{a1a1x}\sum_{\mathsf T \in \mathsf{Tree}(1, 1, g-2)} \frac{1}{|\mathsf{Aut}(\mathsf{T})|} \int_{{\overline{\mathcal M}_{\mathsf T}}} \xi_{\mathsf T}^*\lambda_g\cdot 
    \mathsf{Cont}_{\mathsf T}=\frac{1}{288}\cdot \frac{1}{(2g-2)!}\,,\end{equation}
    for all $g\geq 2$.
    The connection (if any)  between identities \eqref{numericalx} and \eqref{a1a1x} is not clear. 

\section{Wall-crossing calculation} \label{wallcr}
\subsection{Overview}
After reviewing Fulton-MacPherson
degenerations
and the definition of unramified maps, we present the calculation of
\begin{equation} \label{ce445}
\Tor^*([\A_{2}\times\A_{g-2}])
\in \R^{*}(\M_g^{\ct})
\end{equation}
using the wall-crossing formula of Theorem \ref{thm:wallcr}. 
The families Gromov-Witten class
$[\M^{\ct}_{g}(\pi_2)]^{\textnormal{red}}$, computed
by wall-crossing, is pushed-forward
along the morphism 
$\M^{\ct}_{g}(\pi_2)\rightarrow \M_g^{\ct}$
to obtain $\Tor^*([\A_2\times \A_{g-2}])\in \R^*(\M_{g}^{\ct})$. As noted in Section \ref{general}, the homomorphism property
for the pair $([\J_g],[\A_2\times \A_{g-2}])$ is implied by the equality
\begin{equation} \label{f55g}
\Tor^*([\A_2\times \A_{g-2}])= 
\Tor^*(\taut([\A_2\times \A_{g-2}])) \in \R^*(\M_g^{\ct})\, .
\end{equation}
We have checked
\eqref{f55g}  for $2\leq g \leq 8$ by this approach. 

\subsection{Fulton-MacPherson degenerations}
Stable maps from curves to a target variety allow for nodal degenerations of the domain curves, but do not degenerate the target. Unramified  maps \cite {KKO}, however, permit the target to degenerate. Before defining unramified maps, we recall the definition of the degenerations of the targets which arise.

Let $X$ be a nonsingular projective variety of dimension $r$, which for us will be an abelian variety. The configuration space of $n$ distinct labeled points in $X$ can be compactified by means of the Fulton-MacPherson space $\FM_{n}(X)$ constructed in \cite {FM}. The space $\FM_{n}(X)$ is obtained as an iterated blowup of the product $X^n$ along various diagonals in a specified order, encoding how the points of $X$ collide. There is a universal family \begin{equation}\label{uff} \mathcal W\to \FM_{n}(X)\,,\end{equation}
constructed as
an iterated blowup of $\FM_{n}(X)\times X,$ which carries a projection $$\rho_n: \mathcal W\to X\,.$$ {\it Fulton-MacPherson degenerations} $$\rho: W\to X$$ are obtained by taking $W$ to be a fiber of the universal family \eqref{uff}, for arbitrary $n$, and letting $\rho$ denote the restriction of $\rho_n$ to the fiber $W$.
An example of a Fulton-MacPherson degeneration is
\begin{equation}\label{1step}
W=\mathrm{Bl}_0(X)\cup_{E=D}\p^r,
\end{equation}
which is given by the degeneration to the normal cone of the point $0 \in X$. Here, $D$ is the divisor at infinity in $\mathbb P^{r}\cong \mathbb P(T_0X\oplus \mathbb C)$ and $E\cong \mathbb P(T_0X)$ is the exceptional divisor of the blowup of $X$. In general, $W$ is an iterated degeneration to the normal cones of regular points. 

\subsection{Unramified maps}
Let $(\rho:W\rightarrow X,x_1,\hdots, x_m)$ denote a marked Fulton-MacPherson degeneration of an abelian variety $X$. 
We require that the markings $x_1, \ldots, x_m$ are pairwise distinct and are contained in the nonsingular locus of $W$.

For a  principally polarized abelian variety $(X,\theta)$, we are
interested in  the {\em minimal curve class} 
$$\beta_{\min}=\frac{\theta^{r-1}}{(r-1)!}\in H_2(X,\mathbb{Z})\, .$$ We give the definition of unramified maps from a nodal curve $C$ to $W$ in the class $\beta_{\min}$ with prescribed multiplicity profiles at the markings $x_i$.

\begin{definition}
Let $(\rho:W\rightarrow X,x_1,\ldots,x_m)$ be a marked Fulton-MacPherson degeneration of $X$. Let $C$ be a curve with (at worst) nodal singularities. A map $f:C \rightarrow W$ has {\em curve class} $\beta_{\min}\in H_2(X,\mathbb{Z})$ if 
\begin{equation}
(\rho \circ f)_*[C]=\beta_{\min}\, . \qedhere
\end{equation}
\end{definition}

\begin{definition}[\hspace{-0.5pt}\cite{KKO}, \cite{Nesterov}]\label{ununun}
A map $f\colon C \rightarrow W$ is {\em unramified} if
\begin{itemize}
\item [(i)] $f$ is a logarithmic morphism \footnote{We use the canonical logarithmic structures coming from the singularities of $C$ and $W$, respectively.} with nowhere vanishing logarithmic derivative $$df^{\log} \colon T^{\log}_{C} \rightarrow f^*T^{\log}_{W}\, ,$$
\item [(ii)] the group of automorphisms $$\{(\gamma_1,\gamma_2)\in \Aut(C)\times \Aut(W,x_1, \hdots, x_m) \mid f\circ \gamma_1=\gamma_2\circ f \}$$ is finite.\footnote{For a general curve class $\beta$, a more intricate stability condition has to be imposed: see \cite[Definition 3.1.1 (4)]{KKO}.} \qedhere
\end{itemize}
 
\end{definition}

Fix ordered partitions $(\mu^1, \ldots, \mu^m)$, and write $$\mu^i=(\mu^i_1 ,\hdots, \mu^i_{\ell(\mu^i)})\,$$ for their parts. Let $\{p_{ij}\}$ be distinct  nonsingular points on $C$, for $1\leq i\leq m$, $1\leq j\leq \ell(\mu^i).$ 
\begin{definition}
The map $f:(C, \{p_{ij}\})\to (W, \{x_i\})$ has {\em multiplicity $\mu^i$ at the marked points $x_i \in W$} provided
$$
f^{-1}(x_i)=\sum^{\ell(\mu^i)}_{j=1}\mu^i_j \,p_{ij}\,,
$$
for $1\leq i\leq m$. We continue to require conditions (i) and (ii) of Definition \ref{ununun}, modified so  the logarithmic structure and the automorphisms of the domain also take into account the markings $\{p_{ij}\}$. 
\end{definition}

Let $\pi_r: \mathcal X_r\to \A_r$ denote the universal principally polarized abelian variety. Write
\[
\M^{\ct,\mathrm{un}}_{g_0}(\pi_r,\mu^1,\hdots,\mu^m)^\bullet
\]
for the moduli space of unramified maps parameterizing: \begin{itemize} 
\item marked Fulton-MacPherson degenerations 
$$\rho:W \to X,\, x_1, \ldots, x_m$$
of arbitrary principally polarized abelian varieties $(X, \theta)\in \A_r$,
\item unramified maps $$f:(C, \{p_{ij}\}) \to (W, \{x_i\})$$ having curve class $\beta_{\min}\in H_2(X, \mathbb Z)$ and multiplicity $\mu^i$ at the marked points $x_i\in W$, $1\leq i\leq m$. 
\end{itemize} 
The domain curves $C$ are assumed to be (at worst) nodal, of compact type, possibly disconnected, and of genus $g(C)=g_0$. The maps are considered up to automorphisms of Fulton-MacPherson degenerations, which include translations of abelian varieties and affine automorphisms of projective spaces. The markings of the domain are not explicitly indicated in the notation, for ease of reading. 
The analogously defined moduli space with connected domain curves $C$ will be denoted $\M^{\ct,\mathrm{un}}_{g_0}(\pi_r,\mu^1,\hdots,\mu^m)^\circ$. 
 
Because of the minimality of the curve class, we will only need multiplicities of the form 
\[ 
\mu^i=(1,\ldots, 1)\, .
\]
For other curve classes $\beta$, other multiplicities may also arise. 
\subsection{Definition of $\psi$-classes}
Let $\mathcal{FM}_{m}(\pi_r)$ be the moduli space of Fulton-MacPherson degenerations of abelian varieties with $m$ markings{\footnote{The markings are pairwise distinct and contained in the nonsingular locus of the Fulton-MacPherson degeneration. The marked degenerations are allowed to have an infinite group of automorphisms.}}. There exists a natural forgetful morphism   
\begin{equation}
\label{fmor}
\M^{\ct,\mathrm{un}}_{g_0}(\pi_r,\mu^1,\hdots,\mu^m)^\bullet \rightarrow \mathcal{FM}_{m}(\pi_r)\,. 
\end{equation}
For a subset $I \subset\{1,\hdots, m\}$, consider the divisor
\[ 
\mathcal{D}_I \subset  \mathcal{FM}_{m}(\pi_r)\, ,
\]
whose generic element parameterizes a Fulton-MacPherson degeneration of the form (\ref{1step}) where the marked points labeled by $I$ are contained on the bubble $\p^r$. The descendant class associated to the $i^{th}$ marking $x_i$ is defined by
\[ 
\Psi_i=\sum_{i\in I} \mathcal{D}_I\, . 
\]
Pulling back under the forgetful morphism \eqref{fmor}, we obtain the class $$\Psi_i\in \CH^1(\M^{\ct,\mathrm{un}}_{g_0}(\pi_r,\mu^1,\hdots,\mu^m)^\bullet)\, .$$
Up to certain normalization terms, these $\Psi$-classes can also be defined via vector bundles associated with the cotangent spaces at the markings, see \cite[Section 3.2]{DN} for more details. The same definition applies to $\M^{\ct,\mathrm{un}}_{g_0}(\pi_r,\mu^1,\hdots,\mu^m)^\circ$.  
\subsection{Star-shaped graphs}\label{section:star_shaped}
 \label{Star} A star-shaped graph $\mathsf{S}$ is a rooted graph whose non-root vertices are connected to the root by one edge.  We allow the star-shaped graphs to have  genus labels on  vertices and partition labels on edges, as depicted in Figure \ref{tree1}. In the figure, the non-root vertices are gray. The automorphism group  $\Aut(\mathsf{S})$ of a graph $\mathsf{S}$ is defined to be the automorphism group of the set of pairs $\{(g_1, \mu^1), \dots, (g_m, \mu^m) \}$. 

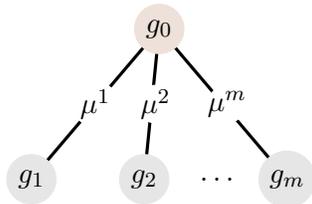
\begin{figure}[!ht]
\centering
\begin{tikzpicture}

    \node (center) at (0,0) {};
    \node (g1) at (-1.7,-2) {};
    \node (g2) at (-0.2,-2) {};
    \node (dots) at (0.8,-2) {\dots};
    \node (gk) at (1.7,-2) {};

    \draw[very thick] (center) -- (g1);
    \filldraw[fill=white, draw=white] (-0.85,-1) circle (.3cm);
    \node at (-0.85,-1) {$\mu^1$};

    \draw[very thick] (center) -- (g2);
    \filldraw[fill=white, draw=white] (-0.05,-1) circle (.25cm);
    \node at (-0.05,-1) {$\mu^2$};

    \draw[very thick] (center) -- (gk);
    \filldraw[fill=white, draw=white] (0.9,-1) circle (.25cm);
    \node at (0.9,-1) {$\mu^m$};

    \node[circle, fill=BrickRed!20, inner sep=2.5pt] at (0,0) {$g_0$};
    \node[circle, fill=gray!20, inner sep=2.5pt] at (-1.7,-2) {$g_1$};
    \node[circle, fill=gray!20, inner sep=2.5pt] at (-0.2,-2) {$g_2$};
    \node[circle, fill=gray!20, inner sep=2.5pt] at (1.7,-2) {$g_m$};

\end{tikzpicture}
\caption{Star-shaped graph with partition labels on edges.}
\label{tree1}
\end{figure}

\subsection{Blowup description for $r=2$} 
The wall-crossing formula in Theorem \ref{thm:wallcr} is a sum of contributions indexed by star-shaped graphs. For $r=2$, only two types of moduli spaces appear in the wall-crossing terms: 
\begin{equation} \label{c5ttg}
\M^{\ct,\mathrm{un}}_{2}(\pi_2,\underbrace{(1),\hdots, (1)}_{k})^\circ \ \ \ \text{and} \ \  \
\M^{\ct,\mathrm{un}}_{1}(\pi_2,(1,1),\underbrace{(1),\hdots, (1)}_{k})^\bullet\, .
\end{equation}
Indeed, for non-product abelian surfaces, the curves in the minimal class $\beta_{\min}$ are smooth genus 2 curves, while for product abelian surfaces $E_1 \times E_2$, the minimal curves are nodal genus 2 curves of compact type. 
The moduli  spaces \eqref{c5ttg}  can be described as iterated blowups of 
\[ 
\M^{\ct}_{2,k} \ \ \ \text{and} \ \ \ \mathcal{M}_{1,2+k}^\bullet=\bigcup_{k_1 +k_2=k} \M^{\mathrm{ct}}_{1,p_1+k_1}\times \M^{\mathrm{ct}}_{1,p_2+k_2}\, ,
\]
respectively, where $p_i$ are distinguished markings corresponding to the multiplicity $(1,1)$ point.  

Let us start with a description of the blowup centers for the first type of moduli spaces of unramified maps.  Given an integer $1 \leq \ell\leq k$,
we define
\[ 
Z_\ell \subset \M^\mathrm{ct}_{2,k}
\] to be
the locus with generic element parameterizing  a nodal genus $2$ curve with a rational bridge that connects two genus $1$ components and  contains $\ell$ marked points. For $\ell=1$, such a curve is depicted on the left  side of Figure \ref{fig:stable_map}. The locus $Z_\ell$  is of pure codimension 2 with possibly several non-disjoint nonsingular irreducible components (corresponding to all possible ways that the points can be distributed among the genus 1 components and the rational bridge).

We blow up $\M_{2,k}^{\ct}$ along the $Z_\ell$ in the following order:
\[
\widetilde{\M}_{2,k}=\mathrm{Bl}_{\tilde{Z}_1}\mathrm{Bl}_{\tilde{Z}_2} \hdots \mathrm{Bl}_{Z_k}\M^{\ct}_{2,k}\, .
\]
More precisely, we start by blowing up $Z_k$ and then proceed inductively: 
at the $i^{th}$ stage, we blow up the proper transform $\tilde{Z}_{k-i}$ of $Z_{k-i}$
for $i=0,\dots,k-1$. After the blowups at stage $i-1$, the proper transform $\tilde{Z}_{k-i}$ is nonsingular, since all of the smooth irreducible components
become disjoint, as explained in \cite{aitordenisjeremy}. In particular, the space
$\widetilde{\M}_{2,k}$ is nonsingular.

Similarly, given an integer $1 \leq \ell \leq k $, let 
\[ 
V_\ell \subset \mathcal{M}_{1,2+k}^\bullet
\]
be the locus with generic element parameterizing a pair of genus 1 curves with the distinguished markings $p_1$ and $p_2$ and $\ell$ other
markings on rational tails{\footnote{In other words, there are $\ell+2$ marked points including $p_1$ and $p_2$ on rational tails of the 
two elliptic curves with at least $2$ on each.}}.
As above, we define the iterated blowup along loci $V_\ell$, 
\[
\widetilde{\M}^\bullet_{1,2+k}=\mathrm{Bl}_{\tilde{V}_1}\mathrm{Bl}_{\tilde{V}_2} \hdots \mathrm{Bl}_{V_k}\M^\bullet_{1,2+k}\,.
\]
The moduli space $\widetilde{\M}^\bullet_{1,2+k}$ is also nonsingular (for the same reason).

\begin{thm}[\hspace{-0.5pt}\cite{aitordenisjeremy}] \label{thm:blowup}  There exist natural isomorphisms 
\[
\M^{\ct,\mathrm{un}}_{2}(\pi_2,\underbrace{(1),\hdots, (1)}_{k})^\circ \cong \widetilde{\M}_{2,k} \quad \text{and} \quad \M^{\ct,\mathrm{un}}_{1}(\pi_2,(1,1),\underbrace{(1),\hdots, (1)}_{k})^\bullet\cong \widetilde{\M}^\bullet_{1,2+k},
\]
such that the reduced virtual fundamental classes coincide with the standard fundamental classes. Moreover, each stage of the iterated blowup admits a moduli-theoretic interpretation in terms of maps to marked Fulton-MacPherson degenerations of abelian surfaces. 
\end{thm}

\subsection{Geometric interpretation of the blowup}

To explain how these blowups arise through the moduli problem of unramified maps, let us consider $\M^{\ct,\mathrm{un}}_{2}(\pi_2,(1))^\circ$. By associating 
to the marked point $x_1 \in W$ the fiber  $f^{-1}(x_1) \in C$, we obtain an identification of spaces
\[ 
\M^{\ct,\mathrm{un}}_{2}(\pi_2,(1))^\circ \cong \M^{\ct,\mathrm{un}}_{2,1}(\pi_2)^\circ,
\]
where the latter is the moduli space of unramified maps with marked points on the domain curve. If the domain $C$ is a nonsingular genus $2$ curve, then $C$ embeds into a non-product abelian surface $X=\textsf{Jac} (C)$ via the Abel-Jacobi map. If  $C$ is a nodal genus 2 curve with the marked point on a rational bridge, then the Abel-Jacobi map contracts the rational bridge to a point in $E_1\times E_2$. The resulting map is not unramified because the differential vanishes at all points on the rational bridge. Hence, the target $E_1\times E_2$ must  sprout a $\p^2$-bubble, over which the rational bridge embeds as a quadric. The geometry is illustrated in Figure \ref{fig:stable_map}. 

\begin{figure}[!ht]
   \centering{
    \fontsize{7pt}{9pt}\selectfont
   \def\svgwidth{90mm}
%% Creator: Inkscape 1.4.2 (ebf0e940d0, 2025-05-08), www.inkscape.org
%% PDF/EPS/PS + LaTeX output extension by Johan Engelen, 2010
%% Accompanies image file 'drawing_degeneration.pdf' (pdf, eps, ps)
%%
%% To include the image in your LaTeX document, write
%%   \input{<filename>.pdf_tex}
%%  instead of
%%   \includegraphics{<filename>.pdf}
%% To scale the image, write
%%   \def\svgwidth{<desired width>}
%%   \input{<filename>.pdf_tex}
%%  instead of
%%   \includegraphics[width=<desired width>]{<filename>.pdf}
%%
%% Images with a different path to the parent latex file can
%% be accessed with the `import' package (which may need to be
%% installed) using
%%   \usepackage{import}
%% in the preamble, and then including the image with
%%   \import{<path to file>}{<filename>.pdf_tex}
%% Alternatively, one can specify
%%   \graphicspath{{<path to file>/}}
%% 
%% For more information, please see info/svg-inkscape on CTAN:
%%   http://tug.ctan.org/tex-archive/info/svg-inkscape
%%
\begingroup%
  \makeatletter%
  \providecommand\color[2][]{%
    \errmessage{(Inkscape) Color is used for the text in Inkscape, but the package 'color.sty' is not loaded}%
    \renewcommand\color[2][]{}%
  }%
  \providecommand\transparent[1]{%
    \errmessage{(Inkscape) Transparency is used (non-zero) for the text in Inkscape, but the package 'transparent.sty' is not loaded}%
    \renewcommand\transparent[1]{}%
  }%
  \providecommand\rotatebox[2]{#2}%
  \newcommand*\fsize{\dimexpr\f@size pt\relax}%
  \newcommand*\lineheight[1]{\fontsize{\fsize}{#1\fsize}\selectfont}%
  \ifx\svgwidth\undefined%
    \setlength{\unitlength}{435.19038643bp}%
    \ifx\svgscale\undefined%
      \relax%
    \else%
      \setlength{\unitlength}{\unitlength * \real{\svgscale}}%
    \fi%
  \else%
    \setlength{\unitlength}{\svgwidth}%
  \fi%
  \global\let\svgwidth\undefined%
  \global\let\svgscale\undefined%
  \makeatother%
  \begin{picture}(1,0.63614617)%
    \lineheight{1}%
    \setlength\tabcolsep{0pt}%
    \put(0,0){\includegraphics[width=\unitlength,page=1]{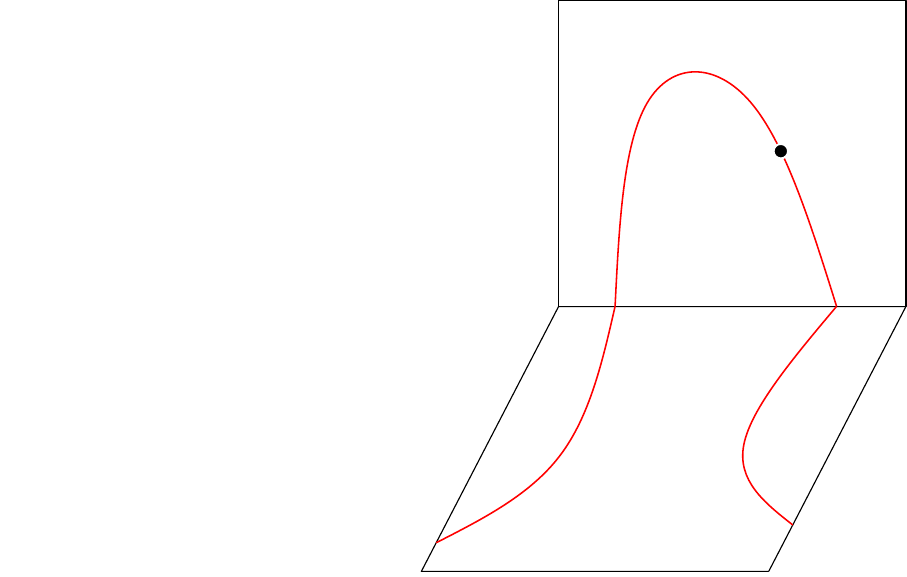}}%
    \put(0.68340039,0.02885531){\makebox(0,0)[t]{\lineheight{1.25}\smash{\begin{tabular}[t]{c}$Bl_0X$\end{tabular}}}}%
    \put(0.83744216,0.5998843){\makebox(0,0)[t]{\lineheight{1.25}\smash{\begin{tabular}[t]{c}$\mathbb{P}^2$\end{tabular}}}}%
    \put(0.72993544,0.42772795){\makebox(0,0)[t]{\lineheight{1.25}\smash{\begin{tabular}[t]{c}$g=0$\end{tabular}}}}%
    \put(0.60528787,0.19505227){\makebox(0,0)[t]{\lineheight{1.25}\smash{\begin{tabular}[t]{c}$g=1$\end{tabular}}}}%
    \put(0.8030621,0.19671426){\makebox(0,0)[t]{\lineheight{1.25}\smash{\begin{tabular}[t]{c}$g=1$\end{tabular}}}}%
    \put(0,0){\includegraphics[width=\unitlength,page=2]{A2A4/stable_map.pdf}}%
  \end{picture}%
\endgroup%

\caption{An example of an unramified map from a nodal genus $2$ curve to a Fulton-MacPherson degeneration of an abelian surface.}
\label{fig:stable_map}
}
\end{figure}

The intersection points of the quadric with the hyperplane at infinity are determined, and the $\p^2$-bubble is considered up to translations and dilations $\mathbb{G}^2_a\rtimes \mathbb{G}_m$. Overall, the moduli of such quadrics up to these automorphisms of $\p^2$ is a point. However, there is an extra degree of freedom associated with choosing a point on the quadric to which the marked point on the rational bridge is mapped -- there is a $\p^1$-moduli of such choices. As a result, we must 
blow up  $\M^{\ct}_{2,1}$ along the locus $Z_1$. With more marked points, there can be more Fulton-MacPherson expansions, which are realized via the iterated blowups of $Z_\ell$.

\subsection{Blowup description of classes} Using Theorem \ref{thm:blowup}, we can describe all the classes appearing in the wall-crossing formula of 
Theorem \ref{thm:wallcr} in terms of tautological and exceptional classes. 
More precisely, the classes that require a separate description are the cotangent line classes $\Psi_i$ and the hyperplane classes $H$. The Chern roots $\alpha_j$ can be identified with the pullbacks of the Chern roots of the Hodge bundles from $\M^{\ct}_{2}$ and $\M_{1,2}^\bullet$.

Let us start with  $\widetilde{\M}_{2,k}$. For a non-empty subset $I\subset \{1,\hdots,k\}$, let $E_I \subset \widetilde{\M}_{2,k}$ denote the exceptional divisor{\footnote{The divisors $E_I$ may be reducible.}}
corresponding to the locus
\[ 
Z_I \subset Z_{|I|},
\]
with generic element parameterizing a nodal genus 2 curve with a rational bridge that connects two genus 1 components and  contains marked points labeled by the set $I$. 
Let 
\[ 
E_i =\sum_{i\in I} E_{I}.
\]  
\begin{prop}[\hspace{-0.5pt}\cite{aitordenisjeremy}] With respect to the isomorphism 
\[\M^{\ct,\mathrm{un}}_{2}(\pi_2,\smash{\underbrace{(1),\hdots, (1)}_{k}})^\circ \cong \widetilde{\M}_{2,k}\]
of Theorem \ref{thm:blowup}, for $1 \leq i\leq k$, we have:
\vspace{-10pt}
\begin{align*}
\Psi_i&=\psi_i-\epsilon_i^*\psi_1+E_i\, , \\
\ev_i^*H&=\epsilon_i^*\psi_1- E_i\, ,\\
\ev_i^*\alpha_j&=\alpha_j\, ,
\end{align*}
where $\epsilon_i \colon \M^{\ct}_{2,k} \longrightarrow \M^{\ct}_{2,1}$ is the morphism which forgets all but the $i^{th}$ marking. The $\psi$ classes are pullbacks of their counterparts on $\M_{2,k}^{\ct}$.
\end{prop}

We now proceed with $\widetilde{\M}^\bullet_{1,2+k}$. Let $E$ be the total exceptional divisor of $\widetilde{\M}^\bullet_{1,2+k} \rightarrow \M^\bullet_{1,2+k}$, so $E$ is the union of all exceptional divisors of the iterated blowups. 

\begin{prop}[\hspace{-0.5pt}\cite{aitordenisjeremy}] With respect to the isomorphism
\[\M^{\ct,\mathrm{un}}_{1}(\pi_2,(1,1), \smash{\underbrace{(1),\hdots, (1)}_{k}})^\bullet\cong \widetilde{\M}^\bullet_{1,2+k}\]
of Theorem \ref{thm:blowup},  we have:
\vspace{-10pt}
\begin{align*}
\Psi_1&=\psi_{p_1}+\psi_{p_2}-E,  \\
 \Psi_{1+i}&=\psi_i\, ,  \ \ \ 1\leq i \leq k ,   \\
\ev_i^*H&=0\, , \ \ \  1\leq i \leq k+1,   \\
\ev_i^*\alpha_j=\alpha_j&=0\,, \ \ \    1\leq i \leq k+1, 
\end{align*}
where  the $\psi$ classes are pullbacks of their counterparts on $\M^\bullet_{1,2+k}$.
\end{prop}

The main computational challenge of the wall-crossing method for $r=2$ is to compute the pushforwards of powers of exceptional classes. Below, we will provide expressions for all relevant powers of exceptional classes for  $g=4$ and $g=5$.

\subsection{Calculations for $g=4$} \label{wc4}
 
The graphs involved in the wall-crossing formula of Theorem \ref{thm:wallcr} for $g=4$ are depicted in Figure \ref{trees4}. 
Exceptional classes appear only for the first two graphs, which correspond to blowups of $\M^{\ct}_{2,1}$ and $\M^{\ct}_{2,2}$ respectively. The last two graphs have too few markings for exceptional loci to occur: there must be at least $4$ markings on the root vertex for a locus $V_\ell \subset \M^\bullet_{1,2+k}$ to be non-empty. The pushforwards of powers of exceptional divisors  can be computed directly through the description of the centers of blowups and their normal bundles.

 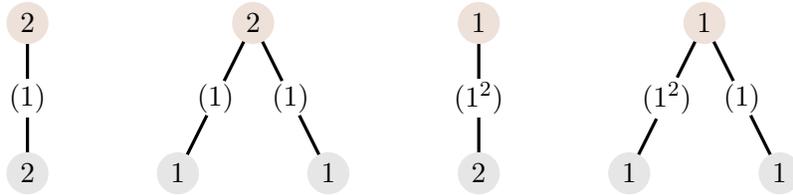
\begin{figure}[!ht] 
 	\centering 
 	\begin{tikzpicture}		
 		
 		\node (1) at (-3,0) {};
 		\node (2) at (-3,-2) {};

 		\draw[very thick] (-3,0)--(-3,-2);

 		\filldraw[ fill=white,draw=white] (-3,-1) circle (.25cm);
 		\node at (-3,-1) {$(1)$};
        
        \node[circle, fill=BrickRed!20, inner sep=2.5pt] at (-3,0) {$2$};
        \node[circle, fill=gray!20, inner sep=2.5pt] at (-3,-2) {$2$};

 		\node (1) at (0,0) {};
 		\node (2) at (-1,-2) {}; 
 		\node (3) at (1,-2) {}; 
 		
 		\draw[very thick] (0,0)--(-1,-2);
 		\draw[very thick] (0,0)--(1,-2);
 		
 		\filldraw[ fill=white,draw=white] (-0.5,-1) circle (.25cm);
 		\node at (-0.5,-1) {$(1)$};
 
 		\filldraw[ fill=white,draw=white] (0.5,-1) circle (.25cm);
 		\node at (0.5,-1) {$(1)$};

        \node[circle, fill=BrickRed!20, inner sep=2.5pt] at (0,0) {$2$};
        \node[circle, fill=gray!20, inner sep=2.5pt] at (-1,-2) {$1$};
        \node[circle, fill=gray!20, inner sep=2.5pt] at (1,-2) {$1$};

        \node (1) at (3,0) {};
 		\node (2) at (3,-2) {}; 
 		
 		\draw[very thick] (3,0)--(3,-2);
 		
 		\filldraw[ fill=white,draw=white] (3,-1) circle (.25cm);
 		\node at (3,-1) {$(1^2)$};

        \node[circle, fill=BrickRed!20, inner sep=2.5pt] at (3,0) {$1$};
        \node[circle, fill=gray!20, inner sep=2.5pt] at (3,-2) {$2$};
        
        \node (1) at (6,0) {};
 		\node (2) at (5,-2) {}; 
 		\node (3) at (7,-2) {}; 
 		
 		\draw[very thick] (6,0)--(5,-2);
 		\draw[very thick] (6,0)--(7,-2);
 		
 		\filldraw[ fill=white,draw=white] (5.5,-1) circle (.3cm);
 		\node at (5.5,-1) {$(1^2)$};
 		
 		\filldraw[ fill=white,draw=white] (6.5,-1) circle (.25cm);
 		\node at (6.5,-1) {$(1)$};   

        \node[circle, fill=BrickRed!20, inner sep=2.5pt] at (6,0) {$1$};
        \node[circle, fill=gray!20, inner sep=2.5pt] at (7,-2) {$1$};
        \node[circle, fill=gray!20, inner sep=2.5pt] at (5,-2) {$1$};

 	\end{tikzpicture}\caption{Graphs for $r=2$ and $g=4$, where $(1^2)$ denotes the partition $(1,1)$.} \label{trees4} 
\end{figure}	 

The first graph corresponds to the blowup $\tau \colon \widetilde{\M}_{2,1} \rightarrow \M^{\ct}_{2,1}$ along $Z_1$. The normal bundle of $Z_1\subset \M^{\ct}_{2,1}$ is trivial, therefore we have
\[ 
\tau_*(E^2_1)=-Z_1\, , \quad \tau_*(E_1^\ell)=0 \ \ \ \text{for} \ \ \ \ell\neq2\, .
\]
Here, $E_1$ is the unique exceptional divisor.

The second graph corresponds to the iterated blowup 
$\tau \colon \widetilde{\M}_{2,2} \rightarrow \M^{\ct}_{2,2}$ along $Z_2$ and $Z_1$.
Inspecting equation \eqref{ifun}, we see that the $I$-functions $I_{1,(1)}(z)$ appearing in Theorem \ref{thm:wallcr} have $z$-degree equal to $1$ in this case. Hence at most the first powers of $E_1$ and $E_2$ contribute, and
\[
\tau_*(E_i)=0\, , \ \ \ \tau_*(E_1 \cdot E_2)=\tau_*(E_{\{1,2\}}^2)=-Z_{\{1,2\}}\, ,
\]
where $Z_{\{1,2\}}\subset \M^{\ct}_{2,2}$ is the locus $Z_I$ for $I=\{1,2\}$. 

The above cases cover all of the exceptional contributions to the wall-crossing formula for $g=4$. The remainder of the calculation lies in the  tautological rings
$$\R^*(\M^{\ct}_{2,k})\ \ \ \text{and}\ \ \  \R^*(\M^{\bullet}_{1,2+k})\, $$
and can be computed using
 \textsf{admcycles}.
 We find that 
the pushforward of
$[\M^{\ct}_{4}(\pi_2)]^{\textnormal{red}}$
along
$$\M^{\ct}_{4}(\pi_2)\rightarrow \M_4^{\ct}$$
equals $\Tor^*(\taut([\A_2\times \A_2]))=42\lambda_1\lambda_3\in \R^*(\M_4^{\ct})$.

\subsection{Calculations for $g=5$}

The graphs involved in the wall-crossing formula of Theorem \ref{thm:wallcr} for $g=5$ are depicted in Figure \ref{trees5}. 
 Although the  complexity of the graphs does not  increase much, the number of exceptional classes involved in the calculation jumps. 
 The exceptional classes appear only for the first three graphs  (the graphs without $(1,1)$-edges). For the sixth  graph corresponding to $\widetilde{\M}^\bullet_{1,2+2}$, the $z$-degree of $I$-function $I_{1,(1,1)}(z)$ is $0$, so no substitution by $\Psi_1$ is required.  

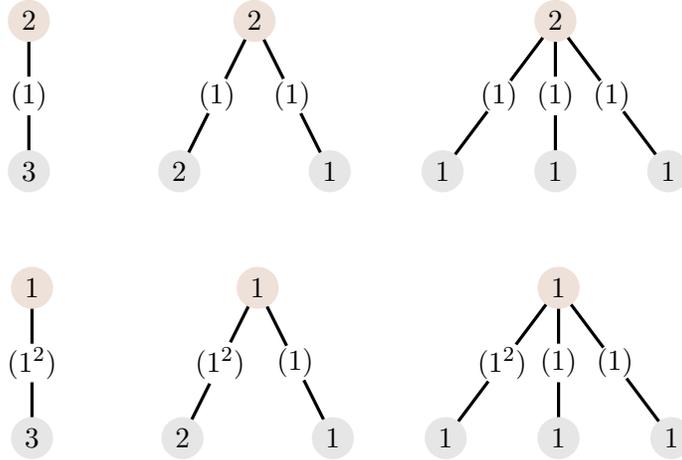
\begin{figure}[!ht]
	\centering \vspace{0.1cm}
	\begin{tikzpicture}		
		
	\node (1) at (-3,0) {};
 		\node (2) at (-3,-2) {};

 		\draw[very thick] (-3,0)--(-3,-2);

 		\filldraw[ fill=white,draw=white] (-3,-1) circle (.25cm);
 		\node at (-3,-1) {$(1)$};
        
        \node[circle, fill=BrickRed!20, inner sep=2.5pt] at (-3,0) {$2$};
        \node[circle, fill=gray!20, inner sep=2.5pt] at (-3,-2) {$3$};
 		
 		\node (1) at (0,0) {};
 		\node (2) at (-1,-2) {}; 
 		\node (3) at (1,-2) {}; 
 		
 		\draw[very thick] (0,0)--(-1,-2);
 		\draw[very thick] (0,0)--(1,-2);

 		\filldraw[ fill=white,draw=white] (-0.5,-1) circle (.25cm);
 		\node at (-0.5,-1) {$(1)$};
 
 		\filldraw[ fill=white,draw=white] (0.5,-1) circle (.25cm);
 		\node at (0.5,-1) {$(1)$};

        \node[circle, fill=BrickRed!20, inner sep=2.5pt] at (0,0) {$2$};
        \node[circle, fill=gray!20, inner sep=2.5pt] at (-1,-2) {$2$};
        \node[circle, fill=gray!20, inner sep=2.5pt] at (1,-2) {$1$};

        \node (1) at (3+1,0) {};
		\node (2) at (3+1,-2) {}; 
		\node (3) at (2+0.5,-2) {}; 
		\node (4) at (4+1.5,-2) {}; 
		
		\draw[very thick] (3+1,0)--(3+1,-2);
		\draw[very thick] (3+1,0)--(2+0.5,-2);
		\draw[very thick] (3+1,0)--(4+1.5,-2);
		
		\filldraw[thick, fill = white] (3+1,0) circle (.25cm) node at (1) {$0$};
		\filldraw[thick, fill=white] (3+1,-2) circle (.25cm) node at (2) {$1$};
		\filldraw[thick, fill=white] (2+0.5,-2) circle (.25cm) node at (3) {$1$};
		\filldraw[thick, fill=white] (4+1.5,-2) circle (.25cm) node at (4) {$1$};
		
		\filldraw[ fill=white,draw=white] (3+1,-1) circle (.25cm);
		\node at (3+1,-1) {$(1)$};
		
		\filldraw[ fill=white,draw=white] (2.5+0.75,-1) circle (.25cm);
		\node at (2.5+0.75,-1) {$(1)$};
		
		\filldraw[ fill=white,draw=white] (3.5+1.25,-1) circle (.25cm);
		\node at (3.5+1.25,-1) {$(1)$};

        \node[circle, fill=BrickRed!20, inner sep=2.5pt] at (3+1,0) {$2$};
        \node[circle, fill=gray!20, inner sep=2.5pt] at (3+1,-2) {$1$};
        \node[circle, fill=gray!20, inner sep=2.5pt] at (2+0.5,-2) {$1$};
        \node[circle, fill=gray!20, inner sep=2.5pt] at (4+1.5,-2) {$1$};
		
	\end{tikzpicture}
	
	\vspace{0.5cm} 
	\hspace{-0.3cm} 
	
	\begin{tikzpicture}

	\node (1) at (-3,0) {};
 		\node (2) at (-3,-2) {}; 
     
 		\draw[very thick] (-3,0)--(-3,-2);
 		
 		\filldraw[ fill=white,draw=white] (-3,-1) circle (.25cm);
 		\node at (-3,-1) {$(1^2)$};
        
        \node[circle, fill=BrickRed!20, inner sep=2.5pt] at (-3,0) {$1$};
        \node[circle, fill=gray!20, inner sep=2.5pt] at (-3,-2) {$3$};
 		
 		\node (1) at (0,0) {};
 		\node (2) at (-1,-2) {}; 
 		\node (3) at (1,-2) {}; 
 		
 		\draw[very thick] (0,0)--(-1,-2);
 		\draw[very thick] (0,0)--(1,-2);
 		
 		\filldraw[ fill=white,draw=white] (-0.5,-1) circle (.3cm);
 		\node at (-0.5,-1) {$(1^2)$};
 		
 		\filldraw[ fill=white,draw=white] (0.5,-1) circle (.25cm);
 		\node at (0.5,-1) {$(1)$};

        \node[circle, fill=BrickRed!20, inner sep=2.5pt] at (0,0) {$1$};
        \node[circle, fill=gray!20, inner sep=2.5pt] at (-1,-2) {$2$};
        \node[circle, fill=gray!20, inner sep=2.5pt] at (1,-2) {$1$};

     \node (1) at (3+1,0) {};
		\node (2) at (3+1,-2) {}; 
		\node (3) at (2+0.5,-2) {}; 
		\node (4) at (4+1.5,-2) {}; 
		
		\draw[very thick] (3+1,0)--(3+1,-2);
		\draw[very thick] (3+1,0)--(2+0.5,-2);
		\draw[very thick] (3+1,0)--(4+1.5,-2);
		
		\filldraw[thick, fill = white] (3+1,0) circle (.25cm) node at (1) {$0$};
		\filldraw[thick, fill=white] (3+1,-2) circle (.25cm) node at (2) {$1$};
		\filldraw[thick, fill=white] (2+0.5,-2) circle (.25cm) node at (3) {$1$};
		\filldraw[thick, fill=white] (4+1.5,-2) circle (.25cm) node at (4) {$1$};
		
		\filldraw[ fill=white,draw=white] (3+1,-1) circle (.25cm);
		\node at (3+1,-1) {$(1)$};
		
		\filldraw[ fill=white,draw=white] (2.5+0.75,-1) circle (.3cm);
		\node at (2.5+0.75,-1) {$(1^2)$};
		
		\filldraw[ fill=white,draw=white] (3.5+1.25,-1) circle (.25cm);
		\node at (3.5+1.25,-1) {$(1)$};

        \node[circle, fill=BrickRed!20, inner sep=2.5pt] at (3+1,0) {$1$};
        \node[circle, fill=gray!20, inner sep=2.5pt] at (3+1,-2) {$1$};
        \node[circle, fill=gray!20, inner sep=2.5pt] at (2+0.5,-2) {$1$};
        \node[circle, fill=gray!20, inner sep=2.5pt] at (4+1.5,-2) {$1$};
				
	\end{tikzpicture}\caption{Graphs for $r=2$ and $g=5$, where $(1^2)$ denotes the partition $(1,1)$.} \label{trees5} 
\end{figure}

The first graph corresponds to $\widetilde{\M}_{2,1}$, and the same analysis applies as in the case of $g=4$. For the second graph, we analyze the moduli space  $\widetilde{\M}_{2,2}$. Since the degrees in $z$ of the corresponding $I$-functions in Theorem \ref{thm:wallcr} are $3$ and $1$ respectively, we must only compute pushforwards of monomials whose degree in $E_1$ is at most 3 and whose degree in $E_2$ is at most 1. An analysis of the normal bundles of the proper transforms as well as the intersections of the irreducible components of $Z_1 \cup Z_2=Z_{\{1,2\}} \cup Z_{\{1\}}\cup Z_{\{2\}}$
shows:
\[
\begin{aligned}
\tau_*(E_i) &= 0\,,&\qquad \tau_*(E_1\cdot E_2) &= -Z_{\{1,2\}} \,,\\
\tau_*(E_1^2) &= -Z_{\{1\}}-Z_{\{1,2\}} \,,&\qquad \tau_*(E_1^2\!\cdot E_2) &= -\psi_{q_1|Z_{\{1\}}} + (\psi_{q_1}+\psi_{q_2})_{|Z_{\{1,2\}}}\,, \\
\tau_*(E_1^3) &= -\psi_{q_1|Z_{\{1\}}} + (\psi_{q_1}+\psi_{q_2})_{|Z_{\{1,2\}}} \,,&\qquad \tau_*(E_1^3\!\cdot E_2) &= 0\,,
\end{aligned}
\]
where $\psi_{q_1}$ is the $\psi$-class on $Z_{\{1\}}$ associated to the node between the rational bridge and the genus $1$ curve containing the second marking and $\psi_{q_1}$ and $\psi_{q_2}$ are the $\psi$-classes on $Z_{\{1,2\}}$ associated to the two nodes between the rational bridge and the genus $1$ curves. 

For $\widetilde{\M}_{2,3}$, at most the first powers of each exceptional class $E_i$ appear:
\begin{align*}
\tau_*(E_i)&=0\, ,\\
\tau_*(E_{i}\cdot E_{j})&= - Z_{\{i,j\}}-Z_{\{1,2,3\}}\, , \\
\tau_*(E_1 \cdot E_2 \cdot E_3)&=-\psi_{q_1 |Z_{\{1,2\}}}-\psi_{q_1 |Z_{\{1,3\}}}-\psi_{q_1 |Z_{\{2,3\}}}+ (\psi_{q_1}+\psi_{q_2})_{|Z_{\{1,2,3\}}},
\end{align*}
where $\psi$-classes $\psi_{q_i}$ are defined in the same way as in the case of two markings. 
Although the powers are not high, the number of exceptional components increases considerably for three markings: $Z_1\cup Z_2\cup Z_3$ has 10 irreducible components. In general, the blowup centers
\[
Z_1\cup \hdots \cup  Z_k
\]
have 
$(3^k-2^k-1)/2+1$
components, giving rise to the same number of exceptional components on the blowup. The explosion of exceptional components is the main source of computational complexity for the wall-crossing method. 

After an \textsf{admcycles}
calculation,  
we find that 
the pushforward of
$[\M^{\ct}_{5}(\pi_2)]^{\textnormal{red}}$
along
$$\M^{\ct}_{5}(\pi_2)\rightarrow \M_5^{\ct}$$
equals $\Tor^*(\taut([\A_2\times \A_3]))=22\lambda_2\lambda_4\in \R^*(\M_5^{\ct})$.

\subsection{Higher genus}
\label{wch}
The calculation of the
contributions obtained by pushing forward exceptional divisors on iterated blowups
of $\widetilde{\M}_{2,k}$
and $\widetilde{\M}^\bullet_{1,2+k}$ is computationally challenging. The natural setting for these calculations is
the logarithmic Chow ring: the algebra of piecewise polynomials on the moduli space of
tropical curves. The required logarithmic intersection calculus has been implemented in \textsf{logtaut} in 
\textsf{admcycles}. 
Using the wall-crossing formulas, we have verified the
homomorphism property for $([\J_g],[\A_2\times \A_{g-2}])$ for $2\leq g \leq 8$ using the code \cite{codeFeusi}.

The algorithm scales linearly with the number of star-shaped graphs and the number of
components of the blowup centers. As such, the total run time grows super-exponentially in $g$,
and the computation is feasible only for small values of $g$. 

\section{Tautological projections on \texorpdfstring{$\X_g^s \rightarrow \A_g$}{Xgs to Ag}}\label{tautcomp}

\subsection{Overview} We start by  defining the tautological ring $\R^*(X^s)$ associated to a self-product of a fixed principally polarized abelian variety $(X,\theta)$ together with a corresponding tautological projection operator 
$$\taut^s_X:\CH^*(X^s)\rightarrow \R^*(X^s)\, .$$
The construction naturally extends to the fiber product{\footnote{In Section \ref{gouf}, the universal family and the 
fiber product were denoted by $\pi_g$ and
$\pi_g^s$. We drop the subscript $g$ here for notational convenience.}} 
$$\pi^s:\X_g^s=\underbrace{\X_g\times_{\A_g} \cdots \times_{\A_g} \X_g}_{s}
\rightarrow \A_g$$
of the universal family of abelian varieties $\pi:\X_g\rightarrow \A_g$.
We then derive closed-form expressions for the tautological projections of natural product cycles on $\X_g^s$.  

\subsection{The tautological ring of products of abelian varieties} 
\label{vee4r}
Let $(X, \theta)$ be a principally polarized abelian variety of dimension $g$. 
There is a canonical{\footnote{Starting with an arbitrary lift $\widetilde \theta$ of the polarization to Chow, the symmetric lift is given by $$\theta=\frac{1}{2}(\widetilde \theta+(-1)^*\widetilde \theta)\in \CH^1(X)\,.$$ The symmetric lift $\theta$ is independent of the choice of $\widetilde \theta$. We denote both the principal polarization and the symmetric Chow class lift by $\theta$.}}  symmetric rational Chow class $\theta\in \mathsf {CH}^1(X)$ of the principal polarization $\theta\in \mathsf H^2(X,\mathbb{Z})$ to a divisor class of $X$.
Let $X^s$ be the $s$-fold product for $s\geq 1$. Let $$\theta_i\in \mathsf{CH}^1(X^s)\,, \quad 1\leq i\leq s\,,$$ be the pullbacks of the theta classes on each of the $s$ factors. 
Let $\mathcal{P}_{ij}$ be the Poincar\'e line bundle for each pair of factors $(i,j)$, trivialized along both zero sections:
$$
\mathcal P_{ij} = \pi_{ij}^*(\mathcal P)\, , \quad \pi_{ij} : X^s \longrightarrow X \times X \cong X\times X^\vee\, , \quad 1\leq i, j\leq s\,, \quad i\neq j\,.
$$ We define the classes $$\eta_{ij} = c_{1}(\mathcal P_{ij})\in \mathsf {CH}^1(X^s)\,, \quad 1\leq i, j\leq s\, , \quad i\neq j\, .$$  Our conventions imply $\eta_{ij}=\eta_{ji}$, so we will often take $i<j$. For convenience, let $\eta_{ii}=\theta_{i}$.
\begin{definition}
The \emph{tautological ring} of $X^s$ is the subring $\R^*(X^s) \subset \CH^*(X^s)$ generated by the classes $\theta_i$ for $1\leq i\leq s$ and the classes $\eta_{ij}$ for $1\leq i<j\leq s.$  \end{definition}

The ring $\R^*(X^s)$ is explicitly described in Proposition \ref{p32} below. The result is likely known to the experts, see \cite [Section 6]{GrushevskyHulekTommasiStable}, but we give a detailed presentation here.

We study  $\R^*(X^s)\subset \CH^*(X^s)$ by considering the images  of tautological classes under the cycle class map \begin{equation}\label{cyclemap}\CH^*(X^s) \rightarrow \mathsf H^{2*}(X^s, \mathbb Q)\, .\end{equation} Let $V\cong \mathbb Q^{2g}$ be the standard symplectic vector space.
We identify $\mathsf H^1(X, \mathbb Q)\cong V$, and, therefore obtain an identification
$$
\mathsf H^*(X^s)\cong \wedge^*(V\otimes \mathbb Q^s)\, .
$$
In symplectic coordinates  $\{x_1, y_1, \ldots , x_g , y_g\}$ on $X$, we can write 
\begin{align}\label{tddd}
\theta &= \sum_{i=1}^g dx_i \wedge dy_i \in \wedge^2 V \cong \mathsf H^2(X, \mathbb Q)\, ,\\
\eta &= c_1(\mathcal P) = \sum_{i=1}^g dx_i^{(1)} \wedge dy_i^{(2)} - dy_i^{(1)} \wedge dx_i^{(2)} \in \wedge^2(V \otimes \mathbb Q^2)\cong \mathsf H^2(X \times X, \mathbb Q)\, .\nonumber
\end{align} The superscripts in the formula for $\eta$ indicate pullbacks from the factors. 
These elements are invariant under the action of the symplectic group $\operatorname{Sp}(V)$ on $\wedge^2V$ and $\wedge^2 (V \otimes \mathbb Q^2)$ respectively. 

More generally, let $$\inv=\left(\wedge^*(V \otimes \mathbb Q^s)\right)^{\operatorname{Sp}(V)}$$ be the $\operatorname{Sp}(V)$-invariant subring of $\wedge^*(V \otimes \mathbb Q^s)$. On tautological classes, the cycle map \eqref{cyclemap} then factors through $$\R^*(X^s) \rightarrow \inv\,.$$ 

The structure of the ring of invariants $\inv$ can be made explicit. We fix a basis $e_1, \ldots, e_s$ of $\mathbb Q^s$. We consider the induced basis $e_I$ of the symmetric algebra $\operatorname{Sym}^{\bullet}(\mathbb Q^s)$ indexed by non-decreasing strings $I$ of integers chosen from the set $\{1, \ldots, s\}$. For each $k\geq 1$, there is a canonical embedding $$\iota_k :\operatorname{Sym}^{2k}(\mathbb Q^s)\hookrightarrow\operatorname{Sym}^{k}(\operatorname{Sym}^2\mathbb Q^s)$$ constructed as follows. For each non-decreasing string $I$ of length $2k$, we set 
$$
\iota_k(e_I) = \sum_{(J_1, \ldots , J_k)}e_{J_1} \ldots e_{J_k}\, ,
$$
where the sum is taken over all possible partitions $(J_1, \ldots, J_k)$ of $I$ with $k$ parts, each consisting of two elements of the string $I$.{\footnote{The sum has $\frac{1}{k!}\binom{2k}{2,\ldots,2}=(2k-1)!!$ terms.}} The following result is proven using classical invariant theory. 

\begin{thm}[\hspace{-0.5pt}\protect{\cite[Theorem 3.4]{Thompson07}}]\label{t30}
  The ring of invariants has the following presentation $$\inv\cong \frac{\operatorname{Sym}^\bullet(\operatorname{Sym}^2 \mathbb Q^s)}{\langle\operatorname{im}(\iota_{g+1})\rangle}\, .$$ 
\end{thm}

The isomorphism in the theorem is induced by $$\wedge^2(V\otimes \mathbb Q^s)\to \text{Sym}^2\mathbb Q^s, \quad (v\otimes e_i)\wedge (w\otimes e_j)\mapsto \omega (v, w) e_{ij}\,,$$ where $\omega$ is the symplectic form on $V$. Using \eqref{tddd}, it follows that the cycle map \eqref{cyclemap} is given on generators by \begin{equation}\label{ongen}\eta_{ij}\to 2e_{ij}, \quad 1\leq i<j\leq s, \quad \theta_i\to e_{ii}, \quad 1\leq i\leq s\,.\end{equation} In particular, $\R^*(X^s)$ surjects onto $\inv$. In fact, more is true. 

\begin{prop}\label{p32}
    The cycle map gives an isomorphism
    $$
    \mathsf{R}^*(X^s) \cong \inv\,.
    $$
    Therefore, the tautological ring $\R^*(X^s)$ is independent (as a $\mathbb{Q}$-algebra) of the choice of the abelian variety $(X, \theta)$.
\end{prop}

\begin{proof}
    We only need to show that the cycle map is injective. If $a_1, \ldots , a_s$ are formal variables, and $I$ is a non-decreasing string of length $2g+2$ chosen from the set $\{1, 2, \ldots, s\}$, then the coefficient of $a_{I}=\prod_{i\in I} a_i$ in the expression
    $$
    \left(\sum_{1\leq i, j \leq s} a_i a_j e_{ij}\right)^{g+1}
    $$
    is precisely $\frac{(2g+2)!}{|\text{Aut}(I)|}\cdot \frac{1}{(2g+1)!!}\cdot \iota_{g+1}(e_I)$. Therefore, using \eqref{ongen}, all the relations in the image of $\R^*(X^s)$ under the cycle class map come from expanding the basic relation
    \begin{equation}\label{eq:basicrel}
        \left(\sum_{i=1}^{s} a_i^2 \theta_i + \sum_{1\leq i< j\leq s}a_i a_j \eta_{ij}\right)^{g+1}=0\, 
    \end{equation}
    as a polynomial in the $a_i$. The relation \eqref{eq:basicrel}  is the pullback of the identity $\theta^{g+1}=0 \in \H^*(X)$ under the morphism
    $$
    \varphi : X^s \rightarrow X : (x_1, \ldots, x_s) \mapsto a_1x_1 + \ldots + a_sx_s\, 
    $$ 
    and, therefore, also holds in the Chow ring of $X^s$. The symmetric property of $\theta\in \CH^1(X)$ is used here to calculate $\varphi^*\theta$ in  Chow:  \begin{equation}\label{varphi}\varphi^*\theta=\sum_{i=1}^{s} a_i^2 \theta_i + \sum_{1\leq i< j\leq s}a_i a_j \eta_{ij}\in \CH^1(X^s)\,.\qedhere\end{equation}\end{proof}
    
\begin{prop}\label{prop:univ family Gorenstein}
    The ring $\R^*(X^s)$ is Gorenstein, with respect to the integration map, with socle in degree $gs$ generated by
    $$
    \frac{\theta_1^g}{g!}\ldots \frac{\theta_s^g}{g!} \, .
    $$
\end{prop}
\begin{proof}
    By Proposition \ref{p32}, it suffices to show that $\inv$ is Gorenstein. Since $\Sp(V)$ is reductive, we can decompose the cohomology $\mathsf H^*(X^s, \mathbb Q)$ into sums of irreducible representations, with $\inv$ occurring as the invariant summand. The top degree part $\mathsf H^{2gs}(X^s)$ has rank $1$ and is spanned by the $\Sp(V)$-invariant class
    $$
    \frac{\theta_1^g}{g!}\ldots \frac{\theta_s^g}{g!} \, .
    $$ 
    Since the intersection pairing on $\mathsf H^*(X^s,\mathbb{Q})$ is perfect, the restriction to the invariant summand must also be perfect. Hence, $\inv$ is Gorenstein. 
\end{proof}
    
The integration map $\R^{gs} (X^s)\rightarrow \mathbb Q$ has a useful interpretation as follows. Let $\mathsf{Sym}^2(s)$ be the set of unordered pairs $I=\{i, j\}$ with $1\leq i, j\leq s$, where repetitions $i=j$ are allowed. For any assignment $$\mathsf {a} : \mathsf{Sym}^{2}(s) \rightarrow \mathbb Z_{\geq 0}\,,$$ we obtain a tautological monomial in $\R^*(X^s)$:
$$
\frac{\eta^{\mathsf a}}{{\mathsf a!}} = \prod_{I\in \mathsf{Sym}^2(s)} \frac{{\eta}^{\mathsf a(I)}}{\mathsf a(I)!}\,,
$$
where, for $I=\{i, j\}\in \mathsf{Sym}^2(s)$, we let $\eta_{I}=\eta_{ij}$ (with the convention $\eta_{ii}=\theta_i$). The monomial $\eta^\mathsf{a}$ has codimension $\sum_{I\in \mathsf{Sym}^2(s)} \mathsf a(I).$ 

\begin{lem} \label{matrix}
    Let $M$ be an $s \times s$ symmetric matrix with formal coefficients $m_{ij}$, $1 \leq i\leq j\leq s$. Then, for any function $\mathsf a : \mathsf{Sym}^2(s) \to \Z_{\geq 0},$ we have
    \begin{equation*}
        \int_{X^s}\frac{\eta^{\mathsf a}}{\mathsf a!} = [m^{\mathsf a}] \det(M)^g\, ,
    \end{equation*}
    where $[m^{\mathsf a}]$ extracts the coefficient\footnote{Here, for $I=\{i, j\}$, we set $m_I=m_{ij}$.} of 
    $\prod_{I\in \mathsf{Sym}^2(s)} m_I^{\mathsf a(I)}$.
\end{lem}
\begin{proof}
    Assume $M$ has integer coefficients, and let $\mathcal L_M\to X^s$ denote the line bundle
    $$
    \mathcal L_M = \prod_{i=1}^s\theta_i^{m_{ii}}\otimes \prod_{1 \leq i < j\leq s}\mathcal P_{ij}^{m_{ij}}\,.
    $$
    The associated morphism $$\varphi_{{\mathcal L}_M} : X^s \rightarrow \widehat{X^s} \cong X^s\,$$ takes the form $$(x_1, \ldots, x_s) \mapsto \left(\sum_{j=1}^{s} m_{1j}x_j, \ldots, \sum_{j=1}^{s} m_{sj}x_j\right)\,.$$ The degree of $\varphi_{\mathcal{L}_M}$ is easily calculated to be $(\det M)^{2g}.$ Therefore, by \cite [Corollary 3.6.2]{BL}, $$\chi(\mathcal L_M)^2 = 
    \deg(\varphi_{\mathcal{L}_M}) = \det(M)^{2g}\,.$$ By Hirzebruch-Riemann-Roch, we obtain
    $$
    \int_{X^s}\frac{(\sum_{i<j} m_{ij}\eta_{ij} + \sum_{i}m_{ii}\theta_i)^{gs}}{(gs)!} = \chi(\mathcal L_M) = \pm\det(M)^g\, ,
    $$
    which we view as a polynomial identity in the variables $m_{ij}$. The sign is easily seen to be positive by comparing the coefficients of $(m_{11} \cdots m_{ss})^g$ on both sides. 
\end{proof}

\subsection{Tautological projection} \label{tproj}
Proposition \ref{prop:univ family Gorenstein} shows that the pairing $$\CH^*(X^s)\times \mathsf {R}^{gs-*}(X^s)\rightarrow \CH^{gs}(X^s)\rightarrow \mathbb Q\,$$ is nondegenerate when restricted to tautological classes on the leftmost factor. The identification $\left(\mathsf {R}^{gs-*}(X^s)\right)^{\vee} \cong \mathsf R^*(X^s)$ gives rise to a canonical projection operator
$$
\taut_X^s: \CH^*(X^s)\rightarrow \mathsf R^*(X^s)\,.
$$
The basic property of the tautological projection is
$$
\int_{X^s} \alpha \cdot \beta = \int_{X^s} \taut_X^s(\alpha)\cdot \beta\,,
$$
for all $\alpha\in \CH^*(X^s)$ and $\beta\in \R^*(X^s).$
\begin{prop}
    The operator $\taut_X^s$ is a $\mathbb{Q}$-algebra homomorphism if and only if every class in $\CH^*(X^s)$ is tautological in cohomology. In particular, $\taut_X^s$ is a homomorphism for a general abelian variety $(X, \theta)\in \A_g$,
    but not for every abelian variety $(X,\theta)$.
\end{prop}

\begin{proof}
    Assume first that the homomorphism property holds. Let $\alpha, \beta\in \CH^*(X^s)$ be classes of complementary degrees. Since $\alpha\cdot \beta$ is a top degree class, we have
    $$
    \int_{X^s} \alpha\cdot \beta = \int_{X^s} \taut_X^s(\alpha\cdot \beta) = \int_{X^s} \taut_X^s(\alpha)\cdot \taut_X^s(\beta) = \int_{X^s} \taut_X^s(\alpha)\cdot \beta\,,
    $$
    where in the first and last equality, the definition of the tautological projection is used. Since $\beta$ is arbitrary, the cycle $\alpha - \taut_X^s(\alpha)$ must be numerically trivial, hence homologically trivial by \cite[Theorem 4]{Lib}. Equivalently, the image of $\alpha$ in cohomology must be tautological. 
    
    Conversely, assume that every Chow class on $X^s$ is tautological in cohomology. For every class  $\alpha \in \CH^*(X^s)$, we can write
    $$
    \alpha = \alpha_0 + \alpha_1\, ,
    $$
    where $\alpha_0$ is tautological, and $\alpha_1$ is homologically trivial. Having such a decomposition forces $\alpha_0=\taut_X^s(\alpha)$. Indeed, we can easily verify the defining property
    $$
    \int_{X^s} \alpha \cdot \gamma = \int_{X^s} (\alpha_0+\alpha_1)\cdot \gamma = \int_{X^s} \alpha_0\cdot \gamma
    $$
    for all $\gamma\in \R^*(X^s).$ Therefore, for all $\alpha, \beta\in \CH^*(X^s)$, we have
    $$
    \alpha\cdot \beta = (\alpha_0+\alpha_1)\cdot (\beta_0+\beta_1)= \alpha_0\cdot \beta_0 + (\alpha_0\cdot \beta_1+\alpha_1\cdot \beta_0+ \alpha_1\cdot \beta_1)\, .
    $$
    Since the last $3$ terms are homologically trivial,  we obtain
    $$
    \taut_X^s(\alpha\cdot \beta) = \alpha_0\cdot \beta_0=\taut_X^s(\alpha)\cdot \taut_X^s(\beta)\, .
    $$ 
    
    Let $(X,\theta)\in \A_g$ be a general abelian variety.  The algebra of Hodge classes $\mathsf{Hdg}(X^s)$ on $X^s$ is generated in degree $1$, see \cite[Proposition 1.6] {Haz} which applies thanks to \cite [Proposition 17.4.2]{BL}. 
    Moreover, for a general $X$, the rational N\'eron-Severi group of $X^s$ is generated by $\theta_i$, for $1\leq i\leq s$, and $\eta_{ij}$, for $1\leq i<j\leq s.$ \footnote{The N\'eron-Severi group is the Rosati invariant part of $\text{Hom}(X^s, \widehat X^s)_{\mathbb Q} = \text{Hom}(X^s, X^s)_{\mathbb Q} = \mathsf{Mat}_s(\mathbb Q)$, which is the space of symmetric matrices, for $X$ general. The standard basis of symmetric matrices corresponds to the classes $\theta_i$ and $\eta_{ij}$.}
    Since the image of $\CH^*(X^s)$ lies in $\mathsf{Hdg}(X^s),$ and the latter is tautological,  every rational Chow class on $X^s$ is tautological in cohomology. 
\end{proof}

\subsection{Properties of tautological rings}

By definition, tautological rings are  preserved under pullbacks via  projections forgetting a factor: 
\begin{equation}\label{ddfs}
\text{pr}:X^{n+1}\to X^n\, , \ \ \ \
\text{pr}^*:\mathsf{R}^*(X^n) \rightarrow \mathsf{R}^*(X^{n+1})\, .
\end{equation}
By \cite[Theorem 9.10]{BMP}, tautological rings are also preserved under the corresponding pushforwards, 
\begin{equation} \label{ddfr}
\text{pr}_*:\mathsf{R}^*(X^{n+1}) \rightarrow \mathsf{R}^*(X^{n})\, .
\end{equation}
In fact, much more general versions of the properties \eqref{ddfs} and \eqref{ddfr} hold.

Let $\mathsf A=(a_{ij})$ be a matrix of size $t\times s$ with integer coefficients. The matrix $\mathsf{A}$ induces a natural morphism \begin{equation}\label{fa}f_{\mathsf A}: X^s \to X^t, \quad (x_1, \ldots, x_s) \to \left(\sum_{j=1}^{s} a_{1j} x_j, \ldots, \sum_{j=1}^{s} a_{tj} x_j\right)\,.\end{equation}
The diagonal inclusion $X \rightarrow X \times X$ is obtained from the
case $\mathsf{A}=\begin{pmatrix} 1\\1 \end{pmatrix}\, .$

\begin{prop}\label{pushfa} The tautological rings are preserved under pushforward, $$(f_{\mathsf A})_*: 
\mathsf{R}^*(X^{s}) \rightarrow \mathsf{R}^*(X^{t})\, .$$ 
\end{prop}
\proof 

For all classes $\alpha \in \CH^*(X^s)$, we have
$$
(f_{\mathsf A})_*\alpha = \widehat{\text{pr}}_* \left({\text{pr}}^*\alpha \cdot \left[\Gamma_{\mathsf A}\right]\right)\, ,
$$
where $\Gamma_{\mathsf A}\subset X^{s+t}$ is the graph of $f_{\mathsf A}$, and $\text{pr},$ $\widehat{\text{pr}}$ are the projections to $X^s$ and $X^t$ respectively. By the pushforward property \eqref{ddfr}, if  $\left[\Gamma_{\mathsf A}\right] \in \R^* %{t g}
(X^{s+t})$, then 
$(f_{\mathsf A})_*\alpha \in \R^*(X^t)$.
We may further assume $t=1$ by using the following equation:
$$
\left[\Gamma_{\mathsf A}\right]=\widehat{\text{pr}}_1^* \left[\Gamma_{1}\right]\cdots \widehat{\text{pr}}_t^*\left[\Gamma_t\right]\,,
$$
where $\Gamma_i\subset X^{s+1}$ is the graph of the morphism
$$
(x_1, \ldots, x_s)\to \sum_{j=1}^{s} a_{ij} x_j\,.
$$
Here, for $1\leq i\leq t$, the projection $\widehat{\text{pr}}_i:X^{s+t}\to X^{s+1}$ records the first $s$ coordinates as well as the coordinate in position $s+i$.  

For $t=1$, we write $(a_1, \ldots, a_s)$ for the entries of $\mathsf A$, and we will prove $\left[\Gamma_{\mathsf A}\right]\in \R^g(X^{s+1})$. Setting 
$$
\varphi:X^{s+1}\to X, \quad (x_1, \ldots, x_s, x_{s+1})\to x_{s+1}-a_1 x_1-\ldots -a_s x_s\,,
$$
we have
$$
\left[\Gamma_{\mathsf A}\right] = \varphi^*([0])=\varphi^* \left(\frac{\theta^g}{g!}\right)\,.
$$
Finally, to show  $\varphi^*\theta\in \CH^1(X^{s+1})$ is tautological, 
we use the identity
$$
\varphi^*\theta = \theta_{s+1} + \sum_{i=1}^{s} a_i^2\theta_i  + \sum_{1\leq i<j\leq s} a_i a_j \eta_{ij} - \sum_{i=1}^{s} a_i \eta_{i, s+1}\,,
$$
which is a particular case of \eqref{varphi}. \qedhere

\begin{prop}\label{pullfa} The tautological rings are preserved under pullback, $$(f_{\mathsf A})^*: 
\mathsf{R}^*(X^{t}) \rightarrow \mathsf{R}^*(X^{s})\, .$$ 
\end{prop}

\begin{proof} By the definition of $\R^*(X^t)$, we can reduce
to the case where $t=1$ or $t=2$. The $t=1$ case is a consequence of
\eqref{varphi}. For $t=2$, we write $$\mathsf{A}=\begin{pmatrix} a_1 & \ldots & a_s \\ b_1 & \ldots & b_s \end{pmatrix}\, ,$$
    %\end{pmatrix}
%(a_1, \ldots, a_s)$ and $(b_1, \ldots, b_s)$ 
and we use \eqref{varphi} and the equation
$$(f_{\mathsf{A}})^*(\eta_{12}) = \sum_{1\leq i<j\leq s} (a_i b_j+a_jb_i) \,\eta_{ij} + 2\sum_{i=1}^{s}a_i b_i \,\theta_i\,.$$ The latter is easily verified by induction on $s$ and the See-saw Theorem. 
\end{proof}

\subsection{The tautological ring of $\X_g^s$}\label{s5454} We now lift the constructions
of Section \ref{vee4r} for a fixed abelian variety $(X,\theta)$ to the moduli space $\A_g$. Let $\pi:\mathcal X_g\to \mathcal A_g$ be the universal abelian variety, and let $\theta\in \mathsf {CH}^1(\X_g)$ be the symmetric theta class normalized by requiring the restriction of $\theta$ to the zero section of $\pi$ to vanish. Over the $s$-fold fibered product $\pi^s: \mathcal X_g^s\to \mathcal A_g$, we have divisor classes
$$
\big\{\, \theta_i\in \CH^1(\X_g^s)\, \big\}_{1\leq i \leq s}\ \ \  \text{and} \ \ \
\big\{\, \eta_{ij} \in \CH^1(\X_g^s)
\, \big\}_{1\leq i\neq j\leq s}\, ,
$$
defined exactly as in Section \ref{vee4r} by pullback from the factors. As before, we have $\eta_{ij}=\eta_{ji}$, and we set $\eta_{ii}=\theta_i$. 
\begin{definition}\label{deftaut}
    The \emph{tautological ring} of $\X_g^s$ is the subalgebra 
    $$\R^*(\mathcal X_g^s)\subset \CH^*(\X_g^s)$$ generated over $\mathsf R^*(\A_g)$ by the classes $\theta_i,$ for $1\leq i\leq s$, and $\eta_{ij}$ for all $1\leq i< j\leq s.$ The subalgebra 
    $$    \R^*_{\mathsf{vert}}(\mathcal X_g^s) \subset \R^*(\mathcal X_g^s)$$ generated 
    over $\mathbb{Q}$ by the classes $\theta_i$, for $1\leq i\leq s$, and $\eta_{ij}$
    for all $1\leq i< j\leq s$ 
    is the \emph{vertical} tautological ring.
\end{definition}

The proofs of Propositions \ref{pushfa} and \ref{pullfa} carry over to families of abelian varieties, hence the same pushforward and pullback properties  also hold for the tautological rings $\R^*(\mathcal X_g^s)$ and $\R^*_{\mathsf{vert}}(\X_g^s)$.
\vskip.1in
Let $\overline {\X}_g$ be a nonsingular toroidal compactification of $\X_g$ lying over a nonsingular toroidal compactification $\overline{\A}_g$ of $\A_g$. Using the vanishing $\lambda_g=0$ on $\overline{\X}_g^s\smallsetminus \X_g^s$ established in \cite{CMOP}, we can define the $\lambda_g$-pairing\footnote{We may assume $\overline{\X}_g^s$ is nonsingular \cite [Chapter VI]{CL}, so that the intersection product is well-defined.} on $\CH^*(\X_g^s)$ exactly as in \eqref{par2}:
\begin{equation}\label{eq:Lgpairing_sfold}
    \CH^*(\X_g^s) \times \CH^{gs + \binom{g}{2}-*}(\X_g^s) \rightarrow \mathbb Q\, , \quad (\alpha, \beta)=\int_{\overline \X_g^s}\bar \alpha\cdot \bar \beta\cdot \lambda_g\,.
\end{equation} 
\vskip.05in

When $s=1$, we have the isomorphism
$$
\R^*_{\textrm{vert}}(\X_g)=\mathbb Q[\theta]/(\theta^{g+1})\, .
$$
The vanishing $\theta^{g+1}=0\in \CH^{g+1}(\X_g)$ holds under our normalization conventions as explained in \cite [Theorem 2.1]{GZ}, see also \cite{DM, Hai}.
For $s=2,$ the vertical tautological ring was studied and shown to be Gorenstein in \cite [Theorem 3.3]{GZ}. The tautological ring for families of abelian varieties was also considered in \cite[Section 9.4] {BMP}, \cite[Section 6]{GrushevskyHulekTommasiStable} and \cite{MoonenChow}. The tautological ring $\R^*(\X^s_g)$ coincides with the definition of \cite{BMP} for the universal family.

\begin{thm}\label{gorgor}
    The following isomorphisms determine the tautological rings of $\X_g^s$ for all $s$:
    \begin{itemize}
    \item[\emph{(i)}] $\R^*_{\mathsf{vert}}(\X_g^s) \cong \R^*(X^s)$ for every fixed principally polarized abelian variety $(X,\theta)$. Therefore,
    $$
    \R^*_{\mathsf{vert}}(\X_g^s)
    \stackrel{\sim}{=} \inv\, .
    $$

    \item[\emph{(ii)}] The natural map $\R^*(\A_g) \otimes \R^*_{\mathsf{vert}}(\X_{g}^s) \longrightarrow \R^*(\X_{g}^s)$ is an isomorphism{\footnote{All tensor products here are $\otimes_{\mathbb{Q}}$.}}. Therefore, $$\R^*(\X_g^s)\cong \mathsf L_{g-1}\otimes \inv\, ,$$ where $\mathsf L_{g-1}$ is the Chow ring of the Lagrangian Grassmannian of $(g-1)$-dimensional Lagrangian subspaces of $\mathbb C^{2g-2}.$   
    \end{itemize} 
    Moreover, the pairing \eqref{eq:Lgpairing_sfold} is perfect when restricted to $\R^*(\X_g^s)$.
\end{thm}

\begin{proof}
    There is a natural surjection \begin{equation}\label{surjjj}\mathsf R^*(X^s)\to \mathsf {R}^*_{\textnormal{vert}}(\X_g^s)\,,\end{equation} sending the classes $\theta_i$ and $\eta_{ij}$ on the left to the analogous classes on the right. The surjection \eqref{surjjj} is well-defined because all relations between $\theta_i$ and $\eta_{ij}$  on $X^s$ also hold on the universal family $\X_g^s$. Indeed, by the proof of Proposition \ref{p32}, the ideal of relations in $\mathsf R^*(X^s)$ is generated by pullbacks of the identity $$\theta^{g+1}=0\in \CH^*(X)$$ under multiplication and addition maps. However, the parallel identity 
    $$\theta^{g+1}=0\in \CH^*(\X_g)$$ also holds for 
    the normalized theta divisor on the universal abelian variety $\pi:\X_g\to \A_g$.  
    
    Using \eqref{surjjj}, we obtain a natural map \begin{equation}\label{mortau}\tau:\mathsf R^*(\A_g)\otimes \mathsf R^*(X^s)\to \mathsf R^*(\A_g)\otimes \mathsf R_{\textnormal{vert}}^*(\X^s_g)\to\R^*(\X_g^s)\,,\end{equation} which is surjective as well. The tensor product $\mathsf R^*(\A_g)\otimes \mathsf R^*(X^s)$ is Gorenstein with socle in degree $\binom{g}{2}+gs$, generated by the class \begin{equation}\label{cll}\lambda_1\cdots \lambda_{g-1} \otimes \frac{\theta_1^g}{g!} \cdots \frac{\theta_s^g}{g!}\,.\end{equation} The image of the socle generator under $\tau$ in $\mathsf R^*(\X_g^s)$ is non-zero: 
    the $\lambda_g$-evaluation \eqref{eq:Lgpairing_sfold} of the $\tau$-image \eqref{cll} is  $$\int_{\overline{\X}_g^s}\lambda_1\cdots \lambda_g \cdot \frac{\bar \theta_1^g}{g!} \cdots \frac{\bar \theta_s^g}{g!}=\int_{\bar \A_g} \lambda_1\cdots \lambda_g=\gamma_g\neq 0\,.$$ The constant $\gamma_g$ is given by \eqref{constantintro}, and the bars indicate any extensions of the theta classes over the boundary. 
    
    On general grounds, {\em if $A$ is Gorenstein with socle in degree $d$, and if $\tau: A\to B$ is a surjective graded morphism which is non-zero in degree $d$, then $\tau$ must be an isomorphism}.\footnote{If $a\in \text{Ker } \tau$ and $a\neq 0$, we can find $a^* \in A$ such that $a a^*$ equals the socle generator for $A$. Then $$\tau(aa^*)=\tau(a)\tau(a^*)=0\, $$ which contradicts the non-vanishing of 
    $\tau$ in degree $d$.} Item (ii) then follows by applying the general statement
    to the morphism  \eqref{mortau}: $$\tau:\mathsf R^*(\A_g)\otimes \mathsf R^*(X^s)\to \mathsf R^*(\X_g^s)\,.$$ For item (i), we apply the same statement to the morphism \eqref{surjjj}. Indeed, the non-vanishing  of the $\tau$-image of \eqref{cll} also proves that \eqref{surjjj} is nonzero in degree $gs$.
\end{proof}

By the same reasoning as in Section \ref{tproj}, Theorem \ref{gorgor} yields a projection operator
$$
\taut^s : \CH^*(\X_{g}^s) \to \R^*(\X_g^s)\, .
$$
In addition, we have the projection operators
$$
\taut: \CH^*(\A_g)\to \R^*(\A_g)\, \quad \text{and} \quad \taut^s_X: \CH^*(X^s)\to \R^*(X^s)\, ,
$$
where $(X,\theta)$ is a fixed abelian variety. 

To understand the relationship between the above three projection operators, consider first the following general situation. Let $f: X \rightarrow S$ be a principally polarized abelian scheme of dimension $g$ over an irreducible base $S$, corresponding to a proper morphism $\iota:S\to \A_g.$ We have the constructions:
\begin{itemize}
    \item [(i)] $X^{s}=X\times_{S}  \cdots \times_{S}X$ is the $s$-fold fibered product, and $f^s:X^s\to S$ is the canonical morphism, 
    \item [(ii)] $j: X^s \longrightarrow \X_g^s$ is the map induced by $\iota$,
    \item [(iii)] $X_{\textnormal{gen}}$ is the principally polarized abelian variety over the generic point of $S$.
\end{itemize}
The geometry is expressed by the diagram:
\begin{center}
\begin{tikzcd}
X^s \arrow[r, "j"] \arrow[d, "f^s"'] 
  & \mathcal X_g^s \arrow[r] \arrow[d, "\pi^s"'] 
  & \overline{\mathcal X}_g^s \arrow[d, "\bar{\pi}^s"] \\
S \arrow[r, "\iota"'] 
  & \mathcal A_g \arrow[r] 
  & \ \overline{\mathcal A}_g\,.
\end{tikzcd}
\end{center}

\begin{lem}\label{lem:integraction_projections}
    Let $\alpha \in \CH^m( X^s)$. Assume, for every $\beta\in \mathsf R^k_{\textnormal{vert}}(\X_g^s)$ with $m+k\neq gs$,  the vanishing \begin{equation}\label{cond}f^s_*(\alpha \cdot j^*\beta) =0\in \CH^*(S)\,
    \end{equation} 
    holds. Then, we have
    $$
    \taut^s (j_*\alpha) = \taut_{X_{\textnormal{gen}}}^s(\alpha_{\textnormal{gen}})\cdot \left(\pi^s\right)^*\taut(\iota_*[S])\in \R^*(\X_g^s)\, , $$ where $\alpha_{\textnormal{gen}}\in \CH^*(X_{\textnormal{gen}}^s)$ is the restriction of $\alpha$ to the generic fiber. 
\end{lem}
\noindent Here, to make sense of the right hand side, we use the canonical identification $\R^*(X_{\textnormal{gen}}^s)\cong \R_{\textnormal{vert}}^*(\X_g^s)$ established in Theorem \ref{gorgor}.
  
\begin{proof}
    The result follows from the definitions via standard arguments, but we provide the details for completeness. It suffices to verify that for all classes $\delta\in \R^*(\X_g^s)$, we have 
    $$
    \int_{\overline \X_g^s} \overline{j_*\alpha} \cdot \bar{\delta} \cdot (\bar \pi^s)^*\lambda_g=\int_{\overline \X_g^s} {\left(\bar{\taut_{X_{\textnormal{gen}}}^s(\alpha_{\textnormal{gen}})}\cdot \left(\bar\pi^s\right)^*\taut(\iota_*[S]) \right)\cdot \bar {\delta} \cdot (\bar \pi^s)^*\lambda_g}\,.
    $$
    The bars indicate extensions to the compactification $\bar \X_g^s$. The tautological classes pulled back from $\A_g$ admit canonical extensions to $\overline{\A}_g$ by \cite{FC}, which we assume throughout (and do not indicate explicitly in the notation). With the aid of Theorem \ref{gorgor}, we write
    $$
    \delta=\beta\cdot (\pi^s)^*\gamma, \quad \beta\in \R^k_{\textnormal{vert}}(\X_g^s), \quad \gamma\in \R^*(\A_g)\,.
    $$
    Therefore, we need to show
    $$
    \int_{\overline \X_g^s} \overline{j_*\alpha} \cdot \bar{\beta} \cdot (\bar \pi^s)^*(\gamma \cdot \lambda_g)=\int_{\overline \X_g^s}{\left(\bar{\taut_{X_{\textnormal{gen}}}^s(\alpha_{\textnormal{gen}})}\cdot \left(\bar \pi^s\right)^*\taut(\iota_*[S])\right)\cdot \bar {\beta} \cdot (\bar \pi^s)^*(\gamma \cdot \lambda_g})\,.
    $$
    We evaluate both sides by pushing forward under $\overline \pi^s:\bar \X_g^s\to \overline \A_g$. We need to establish
    \begin{equation}\label{formula}\int_{\bar \A_g}
    \overline \pi^s_*\left(\overline{j_*\alpha} \cdot \bar \beta\right) \cdot \gamma\cdot \lambda_g= \int_{\bar \A_g} \overline \pi^s_*\left(\bar{\taut_{X_{\textnormal{gen}}}^s(\alpha_{\textnormal{gen}})}\cdot\bar \beta \right)\cdot \mathsf{taut}(\iota_*[S])\cdot \gamma \cdot \lambda_g\,.
    \end{equation} 

    We first claim that both sides of \eqref{formula} vanish if $k+m\neq gs$. Indeed, for the left side, the restriction of the class  $\overline \pi^s_*\left(\overline{j_*\alpha} \cdot \bar \beta\right)$ to $\A_g$ equals $$\pi^s_*(j_*\alpha\cdot \beta)=\pi^s_*\,j_*(\alpha\cdot j^*\beta)=\iota_* \,f^s_*(\alpha\cdot j^*\beta)=0\,,$$ by the hypothesis of the Lemma. Therefore, the class $\overline \pi^s_*\left(\overline{j_*\alpha} \cdot \bar \beta\right)$ is supported on the boundary $\bar \A_g\smallsetminus \A_g$. Since $\lambda_g$ annihilates the boundary terms, the left side of \eqref{formula} vanishes. The same argument applies to the right side of \eqref{formula}, but for different reasons. The corresponding vanishing $$\pi^s_*(\taut_{X_{\textnormal{gen}}}(\alpha)\cdot \beta)=0$$ follows by Theorem \ref{gorgor} (i), see also  \eqref{ppush} below.

    We next consider the case $m+k=gs$. The restriction of the class $\overline \pi^s_*\left(\overline{j_*\alpha} \cdot \bar \beta\right)$ to $\A_g$ equals 
    $$\iota_* \,f^s_*(\alpha\cdot j^*\beta)=\mathsf{const}\cdot \iota_* ([S])\,,$$ where the last equality follows for dimension reasons. The constant in the above equation can be computed on the generic fiber of $f^s$.  Moreover, $\overline \pi^s_*\left(\bar{\taut_{X_{\textnormal{gen}}}^s(\alpha_{\textnormal{gen}})}\cdot\bar \beta\right)$ restricts on $\A_g$ to $\mathsf{const}\cdot [\A_g]$, for the same constant (as a consequence of the definition of the tautological projection $\mathsf{taut}_{X_{\textnormal{gen}}}$ on the generic fiber).

    Putting everything together, equation \eqref{formula} becomes $$\int_{\bar \A_g} \overline{\iota_*([S])} \cdot \gamma\cdot \lambda_g= \int_{\bar \A_g} \mathsf{taut}(\iota_*[S])\cdot \gamma \cdot \lambda_g\,, $$ 
    which holds by the definition of the projection operator $\taut$ on $\A_g$.
\end{proof}
    
\begin{rem}\label{lemmaistrue}
    The assumption in Lemma \ref{lem:integraction_projections} places a restriction on the class $\alpha$. However, Lemma 
    \ref{lem:integraction_projections}
    applies to classes satisfying a natural condition explained here.
    
    For any abelian scheme $p: Y\to S$ of fiber dimension $r$, a class $\alpha\in \CH^*(Y)$ has {\em weight} $w$ if the multiplication maps $n:Y\to Y$ act as $$n^*\alpha = n^{w}\alpha\,.$$ If $w\neq 2r$, then 
     \begin{equation}\label{ppush}p_*\left(\alpha\right)=0\in \CH^*(S)\,\end{equation} 
     as a consequence of the calculation
     $$p_*\left(\alpha\right)=\frac{1}{n^{2r}} p_*(n_*(n^*\alpha))=n^{w-2r} p_*n_*(\alpha) = n^{w-2r} p_*(\alpha)\, .$$ 
     By a fundamental result of Beauville \cite {Bea}, all classes $\alpha\in \CH^*(Y)$ admit a decomposition 
    \begin{equation}\label{bdec}
    \alpha = \alpha_1 + \ldots + \alpha_t\, ,
    \end{equation} where $\alpha_i$ has weight $i$. 
    
    We apply the weight decomposition to the abelian scheme $f^s:X^s\to S$ considered in Lemma \ref{lem:integraction_projections}.
    The classes $\theta_i$ and $\eta_{ij}$ on $\X_g^s$ have weight $2$, so every vertical tautological class $\beta\in \R^k_{\textnormal{vert}}(\X_g^s)$ has weight $2k$.
    If $\alpha\in \CH^m(X^s)$ has a decomposition \eqref{bdec} satisfying \begin{equation} \label{ncondit}\alpha_i=0 \text{ for all } i\leq 2m-1\, ,\end{equation}
    then the vanishing \eqref{cond} holds because of the vanishing \eqref{ppush}.

    In particular, the vanishing \eqref{cond} holds when $\alpha$ is the fundamental class of an abelian subscheme of $X^s$.

 To see an example of a class that does not satisfy \eqref{cond}, let 
 $$\mathcal C_{g,1} \rightarrow \mathcal M_{g,1}=S$$ 
 be the universal curve. Let $\X\rightarrow \M_{g,1}$ be the universal Jacobian, and let
 $\alpha \in \CH^{g-1}(\X)$ be the fundamental class of the universal Abel-Jacobi map $ \mathcal C_{g,1} \rightarrow \X$ defined
 via the marking. A straightforward calculation shows 
 $$f_*(\alpha \cdot \theta^2) = (2g-2)\psi + \kappa_1 \neq 0\, $$ 
 whenever $g \geq 2$. 
\end{rem}

\subsection{The homomorphism property for $\CH^*(\X_g^s)$}
The homomorphism property for the universal fiber product $\X_g^s$ is formulated by exactly following Definition \ref{d3333} for $\A_g$.
\begin{definition}
    Let  $\alpha$, $\beta \in \CH^*(\X_g^s)$. The pair $(\alpha, \beta)$ satisfies the \emph{homomorphism property} if
    \begin{equation} \label{eq:sfold_hom_prop}
    \taut^{s}(\alpha \cdot \beta) = \taut^s(\alpha) \cdot \taut^s(\beta) \in \R^*(\X_g^s)\,. \qedhere
    \end{equation}
\end{definition}
Just as for $\A_g$, the homomorphism property holds if either $\alpha$ or $\beta$ are tautological. Hence \eqref{eq:sfold_hom_prop} provides an obstruction to classes being tautological. 
If Conjecture \ref{conj:Hom} is true, then the homomorphism property \eqref{eq:sfold_hom_prop} holds, by the
projection formula, for any pair $(\alpha, \beta)$ with $\alpha \in \CH^*(\X_g^s)$ and $\beta$ pulled back from $\CH^*(\A_g)$.  
We will further investigate the homomorphism property on $\X_g^s$ in  Corollary \ref{ajppp} of Section \ref {KerCon}.

\subsection{Proof of Theorem \ref{thm:proj_product}}
\label{pr13}
We compute the tautological projection of the generalized product locus of Definition \ref{prlocus}:
$$    \PR_{g,s} = \underbrace{(\A_1 \times \X_{g-1}) \times_{\A_g} \cdots \times_{\A_g}  (\A_1 \times \X_{g-1})}_{s}
    \rightarrow \X_g^s\, .
    $$
The calculation relies on Lemma \ref{lem:integraction_projections}. More precisely, consider the family of abelian varieties $$f: \mathcal X_1\times \mathcal X_{g-1} \to \mathcal A_1\times \A_{g-1}$$ induced by the product morphism $\iota: S = \mathcal A_1 \times \A_{g-1} \to \A_g$. The cycle $$\alpha = [\PR_{g, s}]\in \CH^s((\X_1\times \X_{g-1})\times_{\A_1\times \A_{g-1}}\cdots \times _{\A_1\times \A_{g-1}}(\X_1\times \X_{g-1}))$$ satisfies the vanishing condition \eqref{cond} by Remark \ref{lemmaistrue}. As a result, we have
\begin{eqnarray*}
    \taut^s (   [\PR_{g, s}]  ) & = & \taut_{X_{\textnormal{gen}}}^s(\alpha_{\textnormal{gen}})\cdot \left(\pi^s\right)^*\taut(\iota_*[S]) \\
        & = & \taut_{X_{\textnormal{gen}}}^s(\alpha_{\textnormal{gen}})\cdot  \frac{g}{6|B_{2g}|} \lambda_{g-1}
\end{eqnarray*}   
in $\R^*(\X_g^s)$. Theorem \ref{thm:proj_product} is then a consequence of the following calculation.

\begin{prop}\label{prop:vertical_projection_product}
    Let $E$ and $Y$ be principally polarized abelian varieties of dimensions 1 and $g-1$ respectively. Let $X = E \times Y$ with the induced polarization. Then,
    $$
    \taut_{X}^s\left([(0 \times Y)^s]\right)=\frac{1}{\kappa_{g,s}} \det\begin{pmatrix}
        \theta_1 & \eta_{12}/2 & \ldots & \eta_{1s}/2\\
        \eta_{12}/2 & \theta_2 & \ldots & \eta_{2s}/2\\
        \vdots & \vdots & \ddots & \vdots \\
        \eta_{1s}/2& \eta_{2s}/2  & \ldots & \theta_s
        \end{pmatrix}\, \in \R^*(X^s)\, ,
    $$
    where $\kappa_{g,s} = \prod_{k=0}^{s-1}\left(g+\frac{k}{2}\right)$.
\end{prop}
\begin{proof}
    Let $m_{ij}$ be formal variables such that $m_{ij}=m_{ji}$ for $1\leq i, j\leq s$. Consider the differential operators
    $$
    \partial_{ij} = 2^{\delta_{ij}-1}\frac{\partial}{\partial m_{ij}}\,,
    $$
    where $\delta$ denotes the Kronecker delta. For $\mathsf p : \mathsf{Sym}^2(s) \to \mathbb Z_{\geq 0}$ of codimension $s=\sum_{I\in \textsf{Sym}^2(s)} \mathsf p(I)$, let $\alpha (\mathsf p)$ denote the coefficient of
    $$
    \prod_{I\in \mathsf{Sym}^2(s)} \left(\frac{\partial}{\partial m_I}\right)^{\mathsf p(I)}
    $$
    in the expansion of $\det(\partial_{ij} ).$

    Fix a function $\mathsf a: \mathsf{Sym}^2(s)\to \mathbb Z_{\geq 0}$. From Lemma \ref{matrix}, we obtain
    $$
    \int_{X^s} \frac{\eta^{\mathsf a}}{{\mathsf a}!}\cdot [\PR_{g,s}] = \int_{Y^s} \frac{\eta^{\mathsf a}}{{\mathsf a}!}=[m^{\mathsf a}] \det(M)^{g-1}\,,
    $$ for the matrix $M=(m_{ij}).$
    We need the following Capelli identity \cite{Capelli} which allows us to differentiate powers of determinants of symmetric matrices:
    $$
    \det(\partial_{ij})\det(M)^g = \kappa_{g,s}\det(M)^{g-1}\, .
    $$
    Therefore, 
    \begin{align*}
    \int_{X^s}\frac{\eta^{\mathsf a}}{{\mathsf a}!}\cdot [\PR_{g,s}] &= \frac{1}{\kappa_{g,s}}\, [m^{\mathsf a}] \det(\partial_{i,j})\det(M)^g \\
    &=\frac{1}{\kappa_{g,s}} \sum_{\mathsf p}\frac{(\mathsf a+\mathsf p)!}{{\mathsf a}!}\, \alpha(\mathsf p)\, \,[m^{\mathsf a + \mathsf p}]\det(M)^{g}\\
    &= \frac{1}{\kappa_{g,s}} \,\sum_{\mathsf p}\alpha(\mathsf p)\int_{X^s} \frac{\eta^{\mathsf a+\mathsf p}}{{\mathsf a}!}\\
    &= \frac{1}{\kappa_{g,s}}\,\int_{X^s}  \frac{\eta^{\mathsf a}}{\mathsf a!}\cdot \left(\sum_{\mathsf p}\alpha(\mathsf {p}) \cdot \eta^{\mathsf p}\right)\,.
    \end{align*} 
    From the definition of $\alpha(\mathsf{p})$, we see $$\sum_{\mathsf p}\alpha(\mathsf p) \cdot \eta^{\mathsf p}=\det\begin{pmatrix}
        \theta_1 & \eta_{12}/2 & \ldots & \eta_{1s}/2\\
        \eta_{12}/2 & \theta_2 & \ldots & \eta_{2s}/2\\
        \vdots & \vdots & \ddots & \vdots \\
        \eta_{1s}/2& \eta_{2s}/2  & \ldots & \theta_s
        \end{pmatrix}\,.
    $$
    The Proposition then follows from the definition of tautological projection. 
\end{proof}

\begin{rem}
    Let $X = Z \times Y$ be a product of principally polarized abelian varieties  with $Z$ of dimension $r$ and $Y$ of dimension $g-r$. The argument used in the proof of Proposition \ref{prop:vertical_projection_product} also yields:
    $$
    \taut_X^s\left([(0 \times Y)^s]\right) =\frac{1}{\kappa_{g,s} \cdots \kappa_{g-r+1,s}}\det\begin{pmatrix}
    \theta_1 & \eta_{12}/2 & \ldots & \eta_{1s}/2\\
    \eta_{12}/2 & \theta_2 & \ldots & \eta_{2s}/2\\
    \vdots & \vdots & \ddots & \vdots \\
    \eta_{1s}/2& \eta_{2s}/2  & \ldots & \theta_s
    \end{pmatrix}^r \in \R^*(X^s)\, .
    $$
    Therefore, by Lemma \ref{lem:integraction_projections}, the tautological projection of the cycle $$\underbrace{(\A_r\times \X_{g-r})\times_{\A_g}\cdots \times_{\A_g} (\A_r\times \X_{g-r})}_{s}\to \X_g^s$$ is given by
    $$
    \frac{1}{\kappa_{g,s} \cdots \kappa_{g-r+1,s}}\det\begin{pmatrix}
    \theta_1 & \eta_{12}/2 & \ldots & \eta_{1s}/2\\
    \eta_{12}/2 & \theta_2 & \ldots & \eta_{2s}/2\\
    \vdots & \vdots & \ddots & \vdots \\
    \eta_{1s}/2& \eta_{2s}/2  & \ldots & \theta_s
    \end{pmatrix}^r\cdot \taut([\A_r\times \A_{g-r}])\, \in \R^*(\X^s_g)\, ,
    $$
    where $\taut ([\A_r\times \A_{g-r}])\in \R^*(\A_g)$ was computed in \cite [Theorem 1]{CMOP}. 
\end{rem}

\begin{rem}\label{r49}
    The class of the generalized product locus $[\PR_{g,s}]\in \CH^*(\X_g^s)$ will rarely be tautological. Indeed, we have $$\pi_*^s((\theta_1 \cdots \theta_s)^{s(g-1)}\cdot [\PR_{g,s}] ) = ((g-1)!)^s\,  [\A_1 \times \A_{g-1}]\,,$$ which is 
    non-tautological{\footnote{The class is expected to be non-tautological for all $g \geq 8$.}} for $g=6$, $g=12$, and every even $g\geq 16$ \cite {COP, Iribar}. The classes $[\PR_{5,1}]$ and $[\PR_{4, 2}]$ are also not tautological by Corollary \ref{cor:nont}. \end{rem}
    
\section{Generalized products and Abel-Jacobi maps}\label{ajmaps}

\subsection{Overview} We consider the pullbacks of the generalized product cycles under Abel-Jacobi maps. Using the methods of Sections \ref{exccalc} and \ref{wallcr}, we give two proofs that these pullbacks are tautological. We also discuss the relationship between the Abel-Jacobi pullbacks and the kernels of the $\lambda_g$-pairing on $\R^*(\M_{g, n}^{\ct})$. 
\subsection {Abel-Jacobi maps} \label{amap} Let $\mathsf{A} = (a_{ij}\in \mathbb{Z})$ be an $s \times n$ matrix
 for which  the sum of the entries of every row of $\mathsf{A}$ is $0$. The matrix $\mathsf{A}$ determines a proper morphism
\begin{equation}\label{eq:abel_jacobi_maps}
    \aj_{\mathsf{A}} : \Mct_{g,n} \to \X_g^s
    \end{equation}
    which is defined
    on $(C,p_1,\ldots,p_n)\in \M_{g,n}$ by
    $$\aj_{\mathsf{A}}(C, p_1, \ldots , p_n) \mapsto \left(\mathsf{Jac}(C),\mathcal O_C\left(\sum_{j=1}^n a_{1j}p_j \right), \ldots, \mathcal O_C\left(\sum_{j=1}^n a_{sj}p_j\right)\right)\,.$$
The unique extension of $\aj_{\mathsf{A}}$ over
$\M_{g,n}^{\ct}$ involves canonical twists (and lies at the beginning of the
theory of the double ramification cycle \cite{GZ2,Hai,JPPZ,P1}). The twists are required in order for the  corresponding line bundles to have degree zero on every component of the curve. 

Two cases are of particular geometric interest. For the $s\times (s+1)$ matrix 
\[
\mathsf{Z}=
\begin{pmatrix}
-1 & 1 & 0 & \cdots & 0 \\
-1 & 0 & 1 & \cdots & 0 \\
\vdots & \vdots & \vdots & \ddots & \vdots \\
-1 & 0 & 0 & \cdots & 1
\end{pmatrix}\,
\]
the induced morphism  
\begin{equation}\label{eq:abel_jacobi_classical}
    \aj_{\mathsf{Z}}: \M_{g, s+1}^{\ct} \to \X_g^s, \quad (C, p_0, p_1, \ldots , p_{s})\mapsto(\Jac(C), \mathcal O(p_1-p_0), \ldots, \mathcal O(p_{s} - p_{0}))
\end{equation}
is the $s$-fold product of the classical Abel-Jacobi map constructed via the marking $p_0$. A second fundamental case corresponds to the $s\times 2s$ matrix 
\[
\mathsf{B}=
\begin{pmatrix}
1 & -1 & 0 & 0 & \cdots & 0 & 0 \\
0 & 0 & 1 & -1 & \cdots & 0 & 0 \\
\vdots & \vdots & \vdots & \vdots & \ddots & \vdots & \vdots \\
0 & 0 & 0 & 0 & \cdots & 1 & -1
\end{pmatrix}\,,
\]
yielding the morphism
\begin{equation}\label{eq:AJ_even}
    \aj_{\mathsf{B}} : \M_{g, 2s}^{\ct} \to \X_g^s, \quad (C, p_1, p_2, \ldots , p_{2s-1}, p_{2s})\mapsto(\Jac(C), \mathcal O_C(p_1-p_2), \ldots, \mathcal O_C(p_{2s-1} - p_{2s}))\, .
\end{equation} 

The Abel-Jacobi map \eqref{eq:AJ_even} arises geometrically as follows. If $\Gamma$ is the stable graph with one vertex of genus $g$ and $s$ loops, then the composition
$$
\Mct_{g,2s} \stackrel{\xi_\Gamma}{\longrightarrow} \Mbar_{g+s} \stackrel{{\Tor\ }}{\longrightarrow}\Abar_{g+s}
$$
can be factorized as
$$
\Mct_{g,2s} \rightarrow \mathcal T\rightarrow \Abar_{g+s}\, ,
$$
where $\Abar_{g+s}$ is a toroidal compactification for which $\Tor$ extends and $\mathcal T$ is a torus bundle over $\X_{g}^s$. The composite $\Mct_{g,2s} \rightarrow\mathcal T \rightarrow\X_{g}^s$ is $\aj_{\mathsf B}$. A proof of the above factorization for $s=1$ is given in \cite[Lemma 5.1]{AitorEulerchar}, but the argument applies to  any $s$.

The pullback along an Abel-Jacobi map $\aj_{\mathsf A}$ preserves tautological classes \cite {GZ2, Hai}. For a vector $v = (v_1, \ldots , v_n) \in \mathbb Z^n$, let 
$$
\theta(v) = \frac{1}{2}\sum_{i=1}^n v_i^2\psi_i - \frac{1}{4}\sum_{h=0}^g \sum_{S}v_S^2 \,\delta_{h, S} \,\in \, \R^1(\Mct_{g,n})\, ,
$$
where the last sum is over all subsets $S\subset \{1, \ldots, n\}$. In the summand, $v_S = \sum_{i \in S} v_i$, and $$\delta_{h,\, S}\in \R^1(\M_{g,n}^{\ct})$$ is the class of the boundary divisor parameterizing curves with a genus $h$ component which carries exactly the markings of $S$. \footnote{Both $\delta_{h, S}$ and $\delta_{g-h, S^c}$ are allowed in the sum.} Then, we have 
\begin{align}\label{thetapull}
    \aj_{\mathsf{A}}^* \theta_i &= \theta(a_{i1}, \ldots , a_{in})\,, \\ \nonumber
    \aj_{\mathsf{A}}^* \eta_{i j} &= \theta(a_{i1} + a_{j1}, \ldots , a_{in} + a_{jn})-\theta(a_{i1}, \ldots , a_{in})-\theta(a_{j1}, \ldots , a_{jn})\,.
\end{align}

\subsection{The Abel-Jacobi pullback of \texorpdfstring{$\PR_{g,s}$}{PRgs}}\label{s6363} The Abel-Jacobi pullbacks of the generalized product cycles, $$\aj_{{\mathsf A}}^* \,[\PR_{g,s}]\in \CH^*(\M_{g, n}^{\ct})\, ,$$ can be computed via the methods of either Section \ref{exccalc} or Section \ref{wallcr}. These calculations show
$$
\aj_{{\mathsf A}}^*\, [\PR_{g, s}]\in \R^*(\Mct_{g,n})\,,
$$
see Theorem \ref {thm:aj^*PRis taut} below.  More generally, we can consider any class of the form $\mathsf{p}_*([\alpha \times \X_{g-1}^s])$, where $\alpha \in \R^*(\X_1^{s})$ and $$\mathsf{p}:(\X_1\times \X_{g-1})^s=\underbrace{(\X_1 \times \X_{g-1}) \times_{\A_g} \cdots \times_{\A_g}  (\X_1 \times \X_{g-1})}_{s}\rightarrow \X_g^s\,
$$
is the product map. The fibered product over $\A_g$ will be understood in the notation throughout Section \ref{ajmaps}. As a special case,
$$
[\PR_{g,s}]=\mathsf{p}_*([\alpha \times \X_{g-1}^s])\in \R^*(\X_g^s) \text {\ \ for }\alpha=\theta_1\cdots \theta_s\in \R^s(\X_1^s)\,.
$$ 

The Abel-Jacobi pullbacks can be expressed in terms of families Gromov-Witten classes for the universal elliptic curve $\pi : \mathcal E \to \mathcal M_{1,1}.$ To explain the formula, we require the following notation. 
\begin{itemize} 
\item For $n\geq 1$, let $\M_{g,n}^{\ct}(\pi ,1)$ be the space of stable maps to the universal family $\pi$ with compact type domain curves. Stability of the map implies stability of the $n$-pointed domain curve  because there are no nonconstant morphisms $\mathbb P^1\to E$ to any  elliptic curve $E$. 
\item Let $\tau$
denote the forgetful map  $$\tau:\M_{g,n}^{\ct}(\pi ,1)\to \M_{g, n}^{\ct}\, .$$ No stabilization is needed to define $\tau$. The dependence on $n$ is dropped in the notation.
\item There are $n$ evaluation maps 
$$
\operatorname{ev}_i : \M_{g,n}^{\ct}(\pi ,1) \longrightarrow \mathcal E\, ,
$$
for $i = 1, \ldots , n$, which can be combined as follows. For any $s\times n$ matrix ${\mathsf A}$ as above, we obtain a morphism 
$$
\operatorname{ev}_{\mathsf A} : \M_{g,n}^{\ct}(\pi ,1) \longrightarrow \underbrace{\mathcal E\times_{\M_{1,1}}\times {\cdots} \times_{\M_{1,1}}\mathcal E}_{s}\cong \X_1^s
$$
given by \begin{equation}\label{evaa}\ev_{\mathsf A} (f: C\to E, p_1, \ldots , p_n) = \left(E, \sum_{j=1}^{n} a_{1j}f(p_j), \ldots , \sum_{j=1}^{n} a_{sj}f(p_j)\right)\,.\end{equation}
\item For $n\geq 1$,  let  $\mathcal Z_n=\M_{g,n}^{\ct, 0}(\pi ,1)$ be the substack corresponding to maps  satisfying $f(p_1) = 0$. 
\item The moduli space $\M_{g,n}^{\ct}(\pi ,1)$ carries a virtual fundamental class of dimension $2g-1+n$ from Gromov-Witten theory. Using the diagram 
\[
\begin{tikzcd}
\M_{g,n}^{\ct,0}(\pi,1)
  \arrow[r]
  \arrow[d]
&
\M_{g,n}^{\ct}(\pi,1)
  \arrow[d, "\ev_1"]
\\
\M_{1,1}
  \arrow[r, "\mathsf{z}"']
&
\ \mathcal{E}\,,
\end{tikzcd}
\] 
we endow the space $\M_{g,n}^{\ct, 0}(\pi ,1)$ with a virtual fundamental class of dimension $2g-2+n$ given by the Gysin pullback \begin{equation} \label{pttv} \left[\M_{g,n}^{\ct, 0}(\pi ,1)\right]^\vir={\mathsf{z}}^! \left[\M_{g,n}^{\ct}(\pi ,1) \right]^{\vir},\end{equation} where $\mathsf{z}$ denotes the zero section of the family $\pi$. 
\end {itemize}

The classes \eqref{pttv} are compatible with the forgetful morphisms $q:\M_{g, n}^{\ct}\to \M_{g, n-1}^{\ct},$ for $n\geq 2$. More precisely a formal computation using the definitions and commutativity of Gysin pullbacks shows \begin{equation}\label{comp}q^!\left[\M_{g,n-1}^{\ct, 0}(\pi ,1)\right]^\vir=\left[\M_{g,n}^{\ct, 0}(\pi ,1)\right]^\vir\in \CH_{2g-2+n}(\M_{g,n}^{\ct, 0}(\pi ,1))\,. \end{equation}

Consider the Torelli map $\Tor_n : \Mct_{g,n} \longrightarrow \A_g$ for $n\geq 1$. We have the following diagram
\[
\begin{tikzcd}
\M_{g,n}^{\ct,0}(\pi,1)\ar[rd]\ar[dd]\ar[rr] & &
\A_1\times \A_{g-1}
  \arrow[dd]
\\
& \M_{g,1}^{\ct,0}(\pi,1)
  \arrow[ru]
  \arrow[dd]
& \\
\M_{g, n}^{\ct}\arrow[rd, "q"] \arrow[rr, "\quad \quad \quad \Tor_n"] & &
\A_g\\
&\M_{g, 1}^{\ct}
  \arrow[ru, "\Tor_1"']
&
\end{tikzcd}
\]
where $q: \M_{g, n}^{\ct} \to \M_{g, 1}^{\ct}$ denotes the forgetful morphism. The left front cell is easily seen to be Cartesian. The right front cell is also Cartesian by  \cite [Proposition 21]{COP}. Furthermore, $$\left[\M_{g,1}^{\ct}(\pi ,1)\right]^\vir=\Tor_1^!([\A_1\times \A_{g-1}])\in \CH_{2g-1}(\M^{\ct,0}_{g,1}(\pi, 1))\, ,$$
by \cite{GreerLian}.
A similar identity holds for all $n\geq 1$: \begin{equation}\label{compp}\Tor_n^!([\A_1\times \A_{g-1}])=q^! \Tor_1^!([\A_1\times \A_{g-1}])= q^!\left[\M_{g,1}^{\ct, 0}(\pi ,1)\right]^\vir = \left[\M_{g,n}^{\ct,0}(\pi ,1) \right]^{\vir} \,,\end{equation} where \eqref{comp} was used for the last equality. 

\begin{lem}\label{lem:generalizedNLGW}
    The Abel-Jacobi pullbacks are related to Gromov-Witten theory by: $$\aj_{\mathsf{A}}^*\left(\mathsf{p}_*([\alpha \times \X_{g-1}^s])\right)=
   \tau_* \left(\operatorname{ev}_{\mathsf{A}}^*(\alpha) \cap [\M^{\ct,0}_{g,n}(\pi ,1)]^{\vir}\right)\, \in \CH^{g-1+s}(\M_{g,n}^{\ct})\, .
    $$
\end{lem}
\begin{proof}
    The proof is a formal consequence of the compatibility between the basic operations in intersection theory. Consider the following fiber diagram:
    \begin{equation}\label{di1}
        \begin{tikzcd}
    {\M_{g,n}^{\ct, 0}(\pi ,1)} \arrow[dd, "\tau"] \arrow[rd] \arrow[rr, "j"] &  & (\X_1 \times \X_{g-1})^s \arrow[r, "\text{pr}"] \arrow[ld] \arrow[dd, "\mathsf p"] & \X_1^s \\
    & \A_1 \times \A_{g-1} \arrow[dd] & & \\
    {\M_{g,n}^{\ct}} \arrow[rd, "\Tor_n"] \arrow[rr, "\quad\quad\quad\quad\quad\aj_{\mathsf A}"] &  & \X_g^s \arrow[ld, "\pi^s"] &  \\
    & \A_g &  &             
    \end{tikzcd}
    \end{equation} Using \cite [Theorem 6.2(a), Proposition 6.4]{Fulton}, we find 
    \begin{align*}
    \aj_{\mathsf A}^*\left(\mathsf{p}_*([\alpha \times \X_{g-1}^s])\right)&=\tau_*\left((\text{pr}\circ j)^*\alpha \cap\aj_{\mathsf A}^!([\X_1^s\times \X_{g-1}^{s}])\right)\\ &=\tau_*\left((\text{pr}\circ j)^*\alpha \cap\Tor_n^!([\A_1\times \A_{g-1}])\right)\\&=
    \tau_*\left((\text{pr}\circ j)^*\alpha \cap [\M_{g,n}^{\ct, 0}(\pi ,1)]^\vir\right) \\& =\tau_*\left(\ev_{\mathsf A}^*\alpha \cap [\M_{g,n}^{\ct, 0}(\pi ,1)]^\vir\right)\, .
    \end{align*} Here, we have used \eqref{compp} and  $\text{pr}\circ j=\ev_{\mathsf A}$.
\end{proof}

\begin{thm}\label{thm:aj^*PRis taut}
    For every $\alpha \in \R^*(\X_1^s)$, the Abel-Jacobi pullback $\aj_{\mathsf A}^*\left( \mathsf{p}_*\left([\alpha \times \X_{g-1}^s]\right)\right)$ belongs to the tautological ring of $\Mct_{g,n}$. 
\end{thm}
Theorem \ref{thm:ajaj} is a special case corresponding to the morphism \eqref{eq:AJ_even} and $\alpha=\theta_1\cdots \theta_s\in \R^s(\X_1^s)$.\vskip.1in

\noindent {\it First Proof of Theorem \ref{thm:aj^*PRis taut}.} We use the analysis of the Torelli pullback in Section \ref{exccalc}. Consider the fiber diagram 
\begin{equation}\label{di2}
\begin{tikzcd}
\mathcal Z_n \ar[rd]\ar[dd, "\tau"  yshift=-1.5ex]\ar[rr] & &
\A_1\times \A_{g-1}
  \arrow[dd]
\\
& \mathcal Z
  \arrow[ru]
  \arrow[dd, "\tau" yshift=-1.5ex]
& \\
\M_{g, n}^{\ct}\arrow[rd, "q"] \arrow[rr, "\quad \quad \quad \Tor_n"] & &
\A_g\\
&\M_{g}^{\ct}
  \arrow[ru, "\Tor"']
&
\end{tikzcd}
\end{equation}
 where $q$ is the forgetful map. 
By Theorem \ref{thm:contributions}, the results of \cite {COP} show  $$\Tor^! [\A_1\times \A_{g-1}]=\sum_{\mathsf T}\frac{1}{|\Aut (\mathsf T)|}\left(\iota_{\mathsf T}\right)_* \mathsf{Cont}_{\mathsf T}\,,$$ where $\iota_{\mathsf T}: \M_{\mathsf{T}}^{\ct} \to \mathcal Z$ are the strata of the fibered product $\mathcal Z$ indexed by extremal rooted trees with a genus $1$ root, and $$\mathsf{Cont}_{\mathsf T}\in \R^*(\M_{\mathsf T}^{\ct})\,.$$ Of course, $\tau\circ \iota_{\mathsf T}=\xi_{\mathsf T}$ is the boundary map in $\M_g^{\ct}$.

The strata of the fibered product $\mathcal Z_n$ are indexed by 
extremal rooted trees $\bar {\mathsf T}$ with $n$ leaves and a genus $1$ root, whose stabilization after forgetting the leaves is a rooted tree $\mathsf T$ as before. Then \begin{align}\label{torn}\Tor_n^![\A_1\times \A_{g-1}]&=q^! \Tor^! [\A_1\times \A_{g-1}]= \sum_{\mathsf T} \frac{1}{|\Aut(\mathsf T)|} q^! \left(\iota_{\mathsf T}\right)_* \mathsf{Cont}_{\mathsf T}\\&= \sum_{\bar{\mathsf T}} \frac{1}{|\Aut(\bar{\mathsf T})|} \left(\iota_{\bar {\mathsf T}}\right)_* \mathsf{Cont}_{\bar{\mathsf T}}\,,\nonumber \end{align} for tautological classes $\mathsf{Cont}_{\bar{\mathsf T}}\in \R^*(\M^{\ct}_{\bar {\mathsf T}})$ obtained by pulling back $\mathsf{Cont}_{\mathsf T}\in \R^*(\M^{\ct}_{\mathsf T})$ under the tautological maps. Here, $\iota_{\bar {\mathsf T}}:\M_{\bar {\mathsf T}}\to \mathcal Z_n$ is the natural inclusion, so that $$\tau\circ \iota_{\bar {\mathsf T}}=\xi_{{\bar {\mathsf T}}}$$ is the boundary map in $\M_{g, n}^{\ct}.$ \footnote{By \eqref{compp}, we see that for $n\geq 1$, the left hand side of \eqref{torn} equals 
$[\mathcal Z_n]^{\vir},$ under the identification $\mathcal Z_n\cong \M_{g, n}^{\ct, 0} (\pi, 1).$}

Fix a graph $\overline {\mathsf T}$, and  let $m$ be the valence of the root (the number of edges and leaves). The root corresponds to a genus $1$ curve with $m$ markings. The marking set $\{1, 2,\ldots, n\}$ is canonically partitioned into subsets $N_1, \ldots, N_m$, where $N_j$ is the singleton set corresponding to a leaf attached to the root or $N_j$ is the set of markings which lie on components attached to the root elliptic curve at a given node. 
Let ${\mathsf A}_{\bar {\mathsf T}}$ be the 
$s\times m$ matrix with entries 
\begin{equation}\label{eq:modified_matrix1}
\left(\mathsf A_{\bar{\mathsf{T}}}\right)_{ij} = \sum_{k \in N_j}a_{ik}\, .
\end{equation} Since $\mathsf A$ has vanishing row sums, $\mathsf {A}_{\mathsf {\bar T}}$ also has vanishing row sums. We construct the associated Abel-Jacobi map $\M_{1, m}^{\ct} \to \X_1^s.$ Composing with the natural projection $\M^{\ct}_{\overline{\mathsf T}} \to \M_{1, m}^{\ct}$, we obtain a morphism $$\aj_{\overline{\mathsf T}}:\M^{\ct}_{\overline{\mathsf T}} \to \X_1^s.$$

Finally, consider the diagram obtained by putting together \eqref{di1} and \eqref{di2}: 
\[
\begin{tikzcd}
& {\mathcal Z}_n
  \arrow[rr, "j"]
  \arrow[dd, "\tau" yshift=-2.5ex]
  \arrow[rrr, bend left=18, "\ev_{\mathsf A}"]
  \arrow[ld]
  \arrow[rd]& &
{(\X_1\times \X_{g-1})^s}
  \arrow[ld, "\pi^s"]
  \arrow[dd, "\mathsf {p}"]
  \arrow[r, "\textnormal{pr}"']
  & 
  \X_1^s\\  
\mathcal Z \arrow[rr]\arrow[dd]&&
{\A_1\times \A_{g-1}}
  \arrow[dd]
&&\\
&\M_{g, n}^{\ct}
  \arrow[rr, "\quad \quad \quad \aj_{\mathsf A}"]
  \arrow[rd, "\Tor_n"]
  \arrow[ld, "q"]
&&
\X_g^s
\arrow[ld, "\pi^s"]
&\\
\M_{g}^{\ct} 
\arrow [rr, "\Tor"]
&&
{\A_g}
&
&
\end{tikzcd}
\]
We intersect \eqref{torn} with $\ev_{\mathsf A}^*\alpha = j^* \text{pr}^*\alpha$. Using Lemma \ref{lem:generalizedNLGW}, equation \eqref{compp}, and the projection formula, we obtain \begin{align*}\aj_{\mathsf A}^* \left(\mathsf{p}_*\left([\alpha \times \X_{g-1}^s]\right)\right) &=\tau_*(\ev_{\mathsf A}^*\alpha \cap \Tor_n^!([\A_1\times \A_{g-1}]))=\sum_{\bar{\mathsf T}} \frac{1}{|\Aut(\bar {\mathsf T})|}\tau_*\left(\iota_{\bar{\mathsf {T}}}\right)_* \left(\iota_{\bar{\mathsf {T}}}^* \ev_{\mathsf A}^*(\alpha) \cap  \mathsf{Cont}_{\bar{\mathsf T}}\right)\\&=\sum_{\bar{\mathsf T}} \frac{1}{|\Aut(\bar {\mathsf T})|} \left(\xi_{\bar{\mathsf {T}}}\right)_* \left( \aj_{\bar {\mathsf T}}^*(\alpha) \cap  \mathsf{Cont}_{\bar{\mathsf T}}\right)\, .
\end{align*} The latter expression is tautological on $\M_{g, n}^{\ct}$ since 
the Abel-Jacobi pullbacks preserve tautological classes
by \eqref{thetapull}.
\vskip.1in

\vskip.1in
\noindent {\it Second Proof of Theorem \ref{thm:aj^*PRis taut}.} The starting point of the second proof is Lemma \ref{lem:generalizedNLGW}. The wall-crossing formula established in \cite{Nesterov} is also valid for the class \begin{equation}\label{tauclass}\tau_* \left(\operatorname{ev}_{\mathsf A}^*(\alpha)\cap [\M_{g,n}^{\ct,0}(\pi,1)]^{\vir}\right)\,
\end{equation}
with insertions.
The resulting expression is a sum over star-shaped graphs as in Section \ref{section:star_shaped}, where the root has genus $g_0=1$, and the partitions $\mu^i = (1)$ for all $i$. The difference is that the graphs are also allowed to carry leaves on non-root vertices, corresponding to the marked points $p_1, \ldots , p_{n}$. We denote the set of such graphs by $\mathsf{Star}(g,n)$. For each star graph $\mathsf S$ with $m$ non-root vertices, let
\[ 
N_j \subset \{1, \dots, n\}\, , \quad j \in \{1, \hdots, m\},
\]
be  the (possibly empty) set of leaves on the $i$-th vertex. Define the $s \times m$ matrix ${\mathsf A}_{\mathsf{S}}$ by
\begin{equation}\label{eq:modified_matrix}
\left({\mathsf A}_{\mathsf{S}}\right)_{ij} = \sum_{k \in N_j}a_{ik}\, .
\end{equation} We have an isomorphism 
$$\mathcal M_{1}^{\operatorname{ct, un}}(\pi_1, (1), \ldots , (1)) \cong \Mct_{1,m}\,.$$
Under the above identification, the evaluation morphism from the root component, 
\[ 
\operatorname{ev}_{\mathsf{S}} \colon \mathcal M_{1}^{\operatorname{ct, un}}(\pi_1, (1), \ldots , (1)) \longrightarrow \E^{s} = \X_1^s,
\]
is easily seen to be the Abel-Jacobi map $\M_{1, m}^{\ct}\to \X_1^s$ associated to the matrix ${\mathsf{A}_\mathsf S}$. We have $\ev_{\mathsf S}=\xi_{\mathsf S}\circ \ev_{\mathsf A}.$

The analogue of Theorem \ref{thm:wallcr} here takes the form
\begin{multline*}
\operatorname{ev}_{\mathsf{A}}^*(\alpha)\cap [\M_{g,n}^{\ct,0}(\pi,1)]^{\vir} \\
= \sum_{\mathsf{S} \in \mathsf{Star}(g,n)}\frac{1}{|\operatorname{Aut}({\mathsf S})|}(\xi_{\mathsf S})_*\left( \mathrm{ev}_{\mathsf{S}}^*(\alpha)\cap [\mathcal M_{1}^{\operatorname{ct, un}}(\pi_1, (1), \ldots , (1))]^\vir \boxtimes \prod^m_{i=1} I_{g_i,N_i}(-\Psi_i)\right)\,,
\end{multline*}
where $[\mathcal M_{1}^{\operatorname{ct, un}}(\pi_1, (1), \ldots , (1))]^\vir$ equals the fundamental class $[\M_{1, m}^{\ct}]$, and $I_{g_i,N_i}(z)$ is the marked $I$-function, see \cite [Proposition 1]{DenisMax} for the exact expression in terms of tautological classes. As proven in Section \ref{amap}, the Abel-Jacobi pullbacks preserve tautological classes $$\operatorname{ev}_{\mathsf S}^*\alpha \in \R^*(\M_{1, m}^{\ct})\,.$$ Therefore, after pushing forward under $\tau:\M_{g,n}^{\ct, 0}(\pi ,1)\to \M^{\ct}_{g, n}$,  the class \eqref{tauclass} is tautological on $\M_{g, n}^{\ct}$. \qed
\vskip.1in

\subsection{The Gorenstein kernels and generalized products} \label{KerCon}
The $\lambda_g$-pairing on $\R^*(\M_{g, n}^{\ct})$ is 
$$
\R^*(\M_{g, n}^{\ct}) \times \R^{2g-3+n-*}(\M_{g, n}^{\ct}) \to  \mathbb Q\,,\, \quad (\alpha, \beta) \mapsto \int_{\Mbar_{g, n}}\overline{\alpha} \cdot \overline{\beta}\cdot \lambda_g\,.$$ 
Contrary to the situation on $\R^*(\A_g)$ and $\R^*(
\X_g^s)$, the $\lambda_g$-pairing is {\it not} in general perfect on $\R^*(\M_{g, n}^{\ct})$.
The first failure was found in $\R^5(\M_{2,8}^{\ct})$ by Petersen 
\cite{pet}, and  failures in
$\R^5(\M_{6}^{\ct})$
and 
$\R^5(\M^{\ct}_{5,2})$
were predicted by Pixton in \cite{PixtonPrincetonPhd}. By recent results of Canning-Larson-Schmitt \cite {CLS}, the $\lambda_g$-pairing fails to be perfect whenever
$2g+n \geq 12$ and $g\geq 2$.
The strategy of \cite{CLS}  is to establish the non-vanishing of the Gorenstein kernels in
\begin{equation}\label{eq:gorenstein_failures}
    \R^5(\Mct_{6}),\,\,\, \R^5(\Mct_{5,2}),\, \,\, \,\R^5(\Mct_{4,4}),\,\,\, \R^5(\Mct_{3,6}), \,\,\, \,\R^5(\Mct_{2,8})\, ,
\end{equation}
and then to propagate these  Gorenstein kernels to the other pairs $(g, n)$ using properties (i)-(v) stated in Section \ref{section:intro_GorensteinKernel}.
 
An explicit geometric class in the Gorenstein kernel was constructed in \cite {COP}: the non-zero class $$\Tor^*\Big([\A_1 \times \A_{5}]- \taut([\A_1 \times \A_5])\Big)\in \R^5(\M_6^{\ct})$$ generates $\mathsf{K}_6\subset \R^*(\M_6^{\ct})$. 

We would like to explain geometrically
the failure of the $\lambda_g$-pairing to be perfect for the remaining cases in \eqref {eq:gorenstein_failures} using generalized products $\PR_{g,s}$
and their Abel-Jacobi pullbacks. Consider the Abel-Jacobi map $\aj_{\mathsf B}$ defined by \eqref{eq:AJ_even}, and let $$\Delta_{g,s} = \aj_{\mathsf{B}}^*\Big([\PR_{g,s}]- \taut^s([\PR_{g,s}])\Big) \in \R^{g-1+s}(\Mct_{g,2s})\,.$$
Theorems \hyperref[thm:10']{10$'$}, \ref{thm:inkernel} and \ref{thm:k44} from Section \ref{gouf} are gathered into the following result.
\begin{thm}\label{prop:lies_in_kernel}
    The class $\Delta_{g,s}$ satisfies the following properties:
    \begin{itemize}
        \item[\emph{(i)}] $\Delta_{g,s}\in \mathsf{K}_{g,2s}$ for all $g$ and $s\in \{0,1\}$,
         \item[\emph{(ii)}] $\Delta_{g,s}\in \mathsf{K}_{g,2s}$ for all $g\leq 4$ and $s=2$,
        \item [\emph{(iii)}] $\Delta_{g,s}\neq 0$ for $(g,s) \in \{{(6,0)},(5,1),(4,2)\}\,.$
    \end{itemize}
\end{thm}
By \cite[Theorem 7]{CLS}, $\mathsf{K}_{5,2}$ is  1-dimensional. Therefore,
Theorem \hyperref[thm:10']{10$'$} 
is a consequence of 
parts (i) and (iii) of 
Theorem \ref{prop:lies_in_kernel}
for 
$\Delta_{5,1}\in \mathsf{K}_{5,2}$. Corollary \ref{cor:nont} stating that the classes $$[\PR_{5, 1}]\in \CH^{5}(\X_5), \quad [\PR_{4, 2}]\in \CH^{5}(\X^2_4)$$ 
are {\em not} tautological is a consequence of the non-vanishings in part (iii) of Theorem \ref{prop:lies_in_kernel}.

Theorem \ref{prop:lies_in_kernel} covers the first three cases listed in \eqref{eq:gorenstein_failures}.
The only obstruction to covering all the cases of \eqref{eq:gorenstein_failures}  
is computational. At the moment, the remaining cases require too much time and memory to compute.

%As a consequence of the non-vanishings in part (iii) of Theorem \ref{prop:lies_in_kernel}, we obtain the following results.

%\begin{cor}\label{ajnottaut}
%The classes $[\PR_{5, 1}]\in \CH^{5}(\X_5)$ and $[\PR_{4, 2}]\in \CH^{5}(\X^2_4)$ are {\em not} tautological.\end{cor}

We have also checked the vanishings $$\Delta_{2, 2}=0\in \CH^{3}(\M^{\ct}_{2, 4})\, , \quad \Delta_{2, 3}=0\in \CH^{4}(\M^{\ct}_{2, 6})\, , \quad \Delta_{3, 2}=0\in \CH^{4}(\M^{\ct}_{3, 4})\, ,$$ which are  consistent with the vanishings of the Gorenstein kernels in these cases proven in \cite{CLS}. 

For every $s \times n$ matrix ${\mathsf A}$, we define $$\Delta_{g,\mathsf {A}} = \aj_{\mathsf A}^* \Big([\PR_{g,s}]-\taut^{s}([\PR_{g,s}])\Big)\in \R^{g-1+s}(\Mct_{g,n})\,.$$
Based on the  calculations of Theorem \ref{prop:lies_in_kernel}, we propose the following general statement about the Gorenstein kernels of $\R^*(\M_{g,n}^{\ct})$.
\begin{speculation}\label{spec-ker} The Gorenstein kernels 
	 $\mathsf{K}_{g,n} \subset \R^*(\M_{g,n}^{\ct})$
     satisfy:
	\begin{itemize}
		\item[$(\mathsf{P1})$] 
        $\Delta_{g,{\mathsf A}}\in \mathsf{K}_{g,n}$ for
         all $g$ and all $s\times n$ matrices ${\mathsf A}$.
        \item[$(\mathsf{P2})$] $\{\mathsf{K}_{g,n}\subset \R^*(\M_{g,n}^{\ct})\}_{g,n}$ is the {\em smallest} system of ideals which is closed under properties $\textnormal{(i)-(v)}$ of Section \ref{section:intro_GorensteinKernel} and for which $(\mathsf{P1})$ holds.
	\end{itemize}
\end{speculation}

Calculations suggest that Speculation \ref{spec-ker} is correct at least for $\mathsf{K}_{g\leq 7,0}$. Pixton suggests that the first interesting case of Speculation \ref{spec-ker} is for codimension 7 cycles in $\mathsf{K}_{8,0}$.

\subsection{Proof of Theorem \ref{prop:lies_in_kernel}} \label{s65}
%We now turn to the proof of Theorem \ref{prop:lies_in_kernel} (i). 

The verifications of Thereom \ref{prop:lies_in_kernel} (ii) and (iii) are through $\mathsf{admcycles}$, and the corresponding code can be found in \cite{codeFeusi}. In these cases, we know the 3-spin relations are complete \cite[Theorem 7, proof of Proposition 33]{CLS}. With the current computational tools, the proof that $\Delta_{4,2} \neq 0$ requires additional optimization to produce the matrix of 3-spin relations, which is explained in \cite[Section 7]{CLS} and was used for \cite[proof of Proposition 33]{CLS}.

\begin{proof}[Proof of Theorem \ref{prop:lies_in_kernel} (i)]
    The case $s=0$ corresponds to \cite [Theorem 4]{COP}, so we need only consider $s=1$. Then $$\Delta_{g, 1} = \aj^*\Big([\PR_{g, 1}] - \frac{1}{6|B_{2g}|} \lambda_{g-1}\cdot \theta\Big)\in \R^{g}(\M_{g, 2}^{\ct})\,,$$ for the Abel-Jacobi map $$\aj: \M_{g, 2}^{\ct}\to \X_g, \quad (C, p_1, p_2)\mapsto (\mathsf {Jac}(C), \mathcal O(p_1-p_2))\,.$$ 

    By Lemma \ref{lem:generalizedNLGW} and the definition of $[\M_{g,2}^{\ct, 0}(\pi ,1)]^{\vir}$, we obtain that 
    \begin{align*}
    \aj^*[\PR_{g,1}] &= \tau_*(\operatorname{ev}_2^*(0) \cap [\M_{g,2}^{\ct, 0}(\pi ,1)]^{\vir} )\\
    &=  \tau_*(\operatorname{ev}_1^*(0) \cdot \operatorname{ev}_2^*(0) \cap [\Mct_{g,{2}}(\pi ,1)]^{\vir}) \in \R^g(\M_{g, 2}^{\ct})\, .
    \end{align*}
    By \cite[Proposition 7]{ILPT}, we have
    $$
    \lambda_g \cap [\Mbar_{g,2}(\pi ,1)]^{\vir} = \frac{1}{24}\lambda_{g-1}\cap [\Mbar_{g,2}(E,1)]^{\vir}\, ,
    $$
    where $E$ is a fixed elliptic curve. We conclude
    \begin{align*}
    \overline{\aj^*([\PR_{g,1}])}\cdot \lambda_g  = \frac{1}{24}\tau_*(\operatorname{ev}_1^*(0) \cdot \operatorname{ev}_2^*(0) \cap [\Mbar_{g,{2}}(E ,1)]^{\vir})\cdot \lambda_{g-1} =\frac{1}{12|B_{2g}|}(\psi_1 + \psi_2) \lambda_g \lambda_{g-1}\, ,
    \end{align*}
    where the last equality is a special case of \cite[Proposition 6.8]{OP23}.
    
    On the other hand,
    \begin{align*}
    \overline{\aj^*\left(\frac{1}{6|B_{2g}|}\cdot \theta\cdot \lambda_{g-1}\right)}\cdot \lambda_g = \frac{1}{6|B_{2g}|}\cdot \frac{1}{2}\left(\psi_1 + \psi_2 - \sum_{h=1}^{g-1} \delta_{h, \{p_1\}}\right) \lambda_g \lambda_{g-1}\,,
    \end{align*} where \eqref{thetapull} was used in the last equality. 
    The class $\lambda_g\lambda_{g-1}$ annihilates the boundary contributions $\delta_{h, \{p_1\}}$ as can be seen from the Mumford relations $\lambda^2_h=0$ and $\lambda_{g-h}^2=0$. 

    We therefore obtain \begin{equation}\label{ldd}\lambda_g\cdot \bar {\Delta}_{g, 1}=0 \in \R^{2g}(\Mbar_{g, 2})\, ,\end{equation} which immediately implies that the class $\Delta_{g, 1}\in \R^g(\M_{g, 2}^{\ct})$ is in the kernel of the $\lambda_g$-pairing. 
\end{proof}

Theorem \ref{prop:lies_in_kernel} can be used to obtain cycles on the universal abelian variety that satisfy the homomorphism property \eqref {eq:sfold_hom_prop}. The Abel-Jacobi map
$$
\aj:\M_{g, 2}^{\ct}\to \X_g, \quad (C, p_1, p_2) \to (\textsf{Jac} (C), \mathcal O( p_1-p_2))
$$
is proper. The  {\it Abel-Jacobi cycle} is defined as the pushforward
$$
[\AJ_g]=\aj_*([\M_{g, 2}^{\ct}]) \in \CH^{\binom{g-1}{2}}(\X_g)\,.
$$ 

\begin{cor} \label{ajppp}
    The pair $([\AJ_g],[\PR_{g, 1}])$ satisfies the homomorphism property on $\X_g$.
\end{cor}
We expect this result to hold more generally for pairs of the form $$([\AJ_{g,\mathsf{A}}],[\PR_{g, s}])$$ on $\X_g^s$, where $[\AJ_{g, \mathsf A}]\in \CH^*(\X_g^s)$ denotes the Abel-Jacobi cycle for arbitrary $g$ and $\mathsf{A}.$

\begin{proof}
    We must verify the equality $$\taut ([\AJ_g]\cdot [\PR_{g, 1}]) = \taut([\AJ_g])\cdot \taut( [\PR_{g, 1}]) \in \R^*(\X_g)\,.$$ For brevity, we omit here the superscripts, writing $\taut$ for the tautological projection on the universal abelian variety $\X_g$. Equivalently, we will prove $$\taut\left(
    [\AJ_g]\cdot
    \big([\PR_{g, 1}] - \taut ([\PR_{g, 1}])\big) \right)=0 \in \R^*(\X_g)\,.$$ By  definition of tautological projection, the above vanishing  amounts to showing \begin{equation}\label{xbar} \int_{\bar \X_g}  
    \overline{[\AJ_g]} \cdot
    \overline{[\PR_{g, 1}] - \taut ([\PR_{g, 1}])} \cdot   \bar{\alpha} \cdot \lambda_g=0\end{equation} for all tautological classes $\alpha\in \R^*(\X_g)$.  We fix
    here any smooth compactification $\bar{\X}_g$ of $\X_g$. 
    
    The crux of the argument is the vanishing \begin{equation}\label{vvv}\overline{\aj^*([\PR_{g, 1}] - \taut ([\PR_{g, 1}]))}\cdot \lambda_g=0 \in \R^{2g}(\Mbar_{g, 2})\end{equation} established in \eqref{ldd}. A minor complication is that the map $$\aj: \M_{g, 2}^{\ct} \to \X_g\hookrightarrow \bar{\X_g}$$ does not extend to $\Mbar_{g, 2}.$ However, such an extension always exists after suitable blowups, yielding morphisms $$\pi: \widetilde \M\to \Mbar_{g, 2}\, , \quad \widetilde {\aj}: \widetilde {\M}\to \bar \X_g\,.$$ We may assume $\pi$ is an isomorphism over the compact type locus $\M_{g, 2}^{\ct}$, which we view as embedded in $\widetilde \M$. Furthermore, we may take $\widetilde \M$ to be a nonsingular proper Deligne-Mumford stack.
    
    The two cycles $\overline {[\AJ_g]}$ and $\widetilde {\aj}_*([\widetilde \M])$ agree on $\X_g$, so $$\overline {[\AJ_g]}=\widetilde {\aj}_*([\widetilde \M])+\text{classes supported on } \bar \X_g\smallsetminus \X_g\,.$$ Since $\lambda_g$ annihilates the boundary terms, we can rewrite the left hand side of \eqref{xbar} as \begin{equation}\label{xbar1} \int_{\bar \X_g}  
    \widetilde {\aj}_*([\widetilde \M])\cdot
    \overline{[\PR_{g, 1}] - \taut ([\PR_{g, 1}])} \cdot   \bar \alpha \cdot \lambda_g = \int_{\widetilde \M}  \widetilde{ \aj}^*\overline{[\PR_{g, 1}] - \taut ([\PR_{g, 1}])} \cdot \widetilde \aj^*{\bar \alpha} \cdot \lambda_g \,.\end{equation}  The cycles $$\widetilde \aj^*\overline{[\PR_{g, 1}] - \taut ([\PR_{g, 1}])}\ \text{ and } \pi^*\overline{\aj^*([\PR_{g, 1}] - \taut ([\PR_{g, 1}]))}$$ agree on $\M_{g, 2}^{\ct}\hookrightarrow \widetilde \M$. Using again that $\lambda_g=0$ kills the boundary contributions, we rewrite \eqref{xbar1} as $$\int_{\widetilde \M} \pi^*\overline{\aj^*([\PR_{g, 1}] - \taut ([\PR_{g, 1}]))} \cdot \widetilde {\aj}^*{\bar {\alpha}} \cdot \lambda_g=0\,,$$ as needed. The last equality follows by pulling back \eqref{vvv} to $\widetilde \M$. 
\end{proof} 

\bibliographystyle{amsplain}
\bibliography{refs}
\end{document}